\documentclass
[
    a4paper,
    DIV=10,
    abstracton
]
{scrartcl}

\usepackage{
    amsmath,
    amssymb,
    amsthm,
    bm,
    graphicx,
    caption,
    subcaption,
    enumerate,
    xcolor,
    mathtools,
    authblk,
    dsfont,
    todonotes,
    bm,
    bbm
}

\usepackage[utf8]{inputenc}
\usepackage[pdffitwindow=false,
            plainpages=false,
            pdfpagelabels=true,
            pdfpagemode=UseOutlines,
            pdfpagelayout=SinglePage,
            bookmarks=false,
            colorlinks=true,
            hyperfootnotes=false,
            linkcolor=blue,
            urlcolor=blue!30!black,
            citecolor=green!50!black]{hyperref}

\newcommand{\R}{\mathbb{R}}                                     
\newcommand{\X}{\mathbb{X}}                                     
\newcommand{\innerprod}[2]{\left\langle #1,\, #2 \right\rangle} 

\newcommand{\subfiguretitle}[1]{{\scriptsize{#1}} \\[0.5ex]}

\newcommand\xqed[1]{\leavevmode\unskip\penalty9999 \hbox{}\nobreak\hfill \quad\hbox{#1}}
\newcommand{\exampleSymbol}{\xqed{$\triangle$}}

\mathtoolsset{centercolon} 
\newtheorem{theorem}{Theorem}[section]
\newtheorem{corollary}[theorem]{Corollary}
\newtheorem{lemma}[theorem]{Lemma}
\newtheorem{proposition}[theorem]{Proposition}
\newtheorem{definition}[theorem]{Definition}
\theoremstyle{definition}
\newtheorem{example}[theorem]{Example}
\newtheorem{remark}[theorem]{Remark}

\title{Transition manifolds of \\ complex metastable systems}

\subtitle{Theory and data-driven computation of effective dynamics}

\author[1]{Andreas Bittracher}
\author[1]{P\'eter Koltai}
\author[1]{Stefan Klus}
\author[1]{Ralf Banisch}
\author[2]{Michael Dellnitz}
\author[1,3]{Christof Sch\"utte}

\affil[1]{Department of Mathematics and Computer Science, Freie Universit\"at Berlin, Germany}
\affil[2]{Department of Mathematics, Paderborn University, Germany}
\affil[3]{Zuse Institute Berlin, Germany}
\date{}

\begin{document}
\maketitle

\begin{abstract}
We consider complex dynamical systems showing metastable behavior but no local separation of fast and slow time scales. The article raises the question of whether such systems exhibit a low-dimensional manifold supporting its effective dynamics. For answering this question, we aim at finding nonlinear coordinates, called reaction coordinates, such that the projection of the dynamics onto these coordinates preserves the dominant time scales of the dynamics. We show that, based on a specific reducibility property, the existence of good low-dimensional reaction coordinates preserving the dominant time scales is guaranteed. Based on this theoretical framework, we develop and test a novel numerical approach for computing good reaction coordinates. The proposed algorithmic approach is fully local and thus not prone to the curse of dimension with respect to the state space of the dynamics. Hence, it is a promising method for data-based model reduction of complex dynamical systems such as molecular dynamics.
\end{abstract}

\section{Introduction}

With the advancement of computing power, we are able to simulate and analyze more and more complicated and high-dimensional models of dynamical systems, ranging from astronomical scales for the simulation of galaxies, over planetary and continental scales for climate and weather prediction, down to molecular and sub-atomistic scales via, e.g., Molecular Dynamics (MD) simulations aimed at gaining insight into complex biological processes. Particular aspects of such processes, however, can often be described by much simpler means than the full process, thus \emph{reducing} the full dynamics to some \emph{essential} behavior or \emph{effective dynamics} in terms of some essential observables of the system. Extracting these observables and the related effective dynamics from a dynamical system, though, is one of the most challenging problems in computational modeling \cite{FGH14a}.

One prominent example of dynamical reduction is arguably given by a  variety of  multiscale systems with explicit fast-slow time scale separation, mostly singularly perturbed systems, where either the fast component is considered in a quasi-stationary regime (i.e.\ the slow components are fixed and assumed not to change for the observation period), or the effective behavior of the fast components is injected into the slow processes, e.g. by averaging or homogenization~\cite{PaSt08}. Much of the recent attention has been directed to the case where the deduction of the slow (or fast) effective dynamics is not possible by purely analytic means, due to the lack of an analytic description of the system, or because the complexity of the system renders this task unfeasible~\cite{FGH14a, FGH14b,CKLMN08, DTGCK16, NLCK06, SEKC09, CrMa17,hmm_like_2007,eqfree-algo}. However, all of these approaches still depend on some local form of time scale separation between the ``fast'' and the ``slow'' components of the dynamics.

The focus of this work is on specific multiscale systems \emph{without} local dynamical slow-fast time scale separation, but for which a reduction to an effective dynamical behavior supported on some low-dimensional manifold is still possible. The dynamical property lying at the heart of our approach is that there is a time scale separation in the \emph{global kinetic} behavior of the process, as opposed to the aforementioned slow-fast behavior encoded in the \emph{local dynamics}. Here, global kinetic behavior means that the multiple scales  show up if we consider the \emph{Fokker--Planck equation} associated with the dynamics, say~$\dot{u} = \mathcal{L}u$, where the Fokker--Planck operator~$\mathcal{L}$ will have several small eigenvalues, while the rest of its spectrum is significantly larger. Such dynamical systems exhibit \emph{metastable} behavior and the slow time scales are the time scales of statistical relaxation between the main metastable sets, while there is no time scale gap for the local dynamics within each of the metastable regions \cite{BovierII,SchuetteSarich2013}.   

Global time scale separation  induced by metastability has been analyzed for deterministic \cite{deju:99} and stochastic dynamical systems \cite{direct-hmc,hums:02} for more than a decade. A typical trajectory of a metastable dynamical system will spend most time within the metastable sets, while rare transitions between these sets happen as sudden ``jumps'' roughly along low-dimensional \emph{transition pathways} that connect the metastable sets \cite{Dellago2009_review,PNAS09,tpt2010}.
For an example, see Figure~\ref{fig:bananapot}.
\begin{figure}[htb]
    \centering
    \begin{minipage}{0.5\textwidth}
        \centering
        \subfiguretitle{a)}
        \includegraphics[width=.9\textwidth]{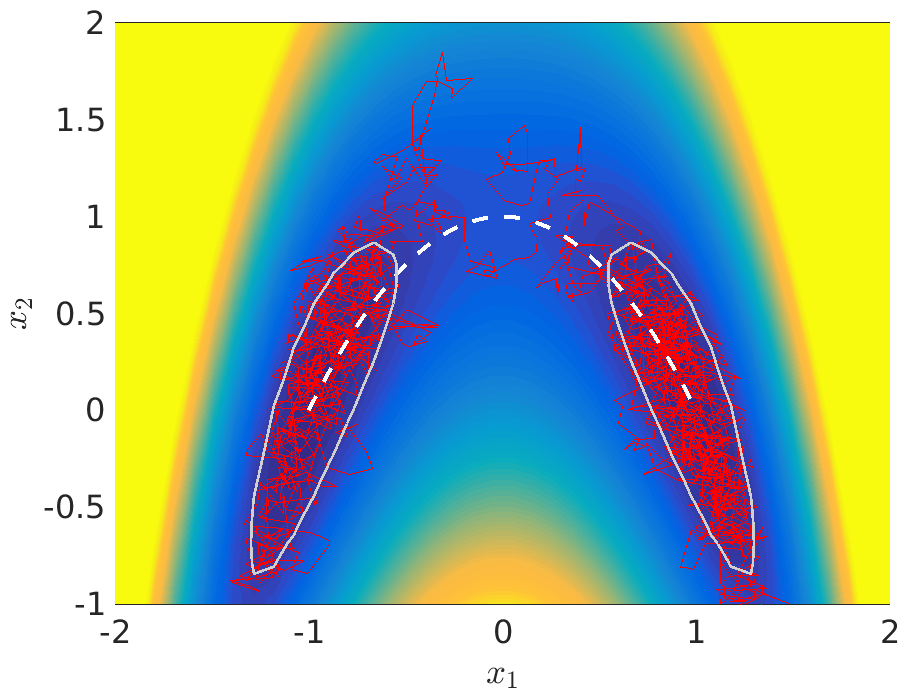}
    \end{minipage}
    \begin{minipage}{0.49\textwidth}
        \centering
        \subfiguretitle{b)}
        \includegraphics[width=\textwidth]{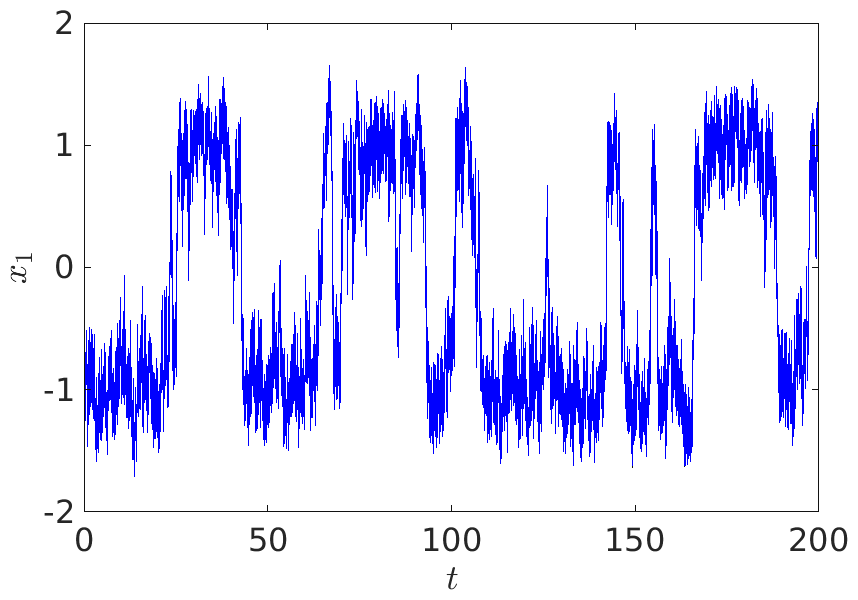}
    \end{minipage}
    \caption{a) Curved double-well potential with two metastable sets (areas encircled by light grey lines) around the global minima $(-1,0)$ and $(1,0)$. In a typical trajectory (red line), transitions between the metastable sets are rare events and generally happen along the transition path (white dashed line). b) The $x_1$-component of a longer trajectory that shows multiple rare transitions (or events).}
    \label{fig:bananapot}
\end{figure}

The tool to describe the global kinetic behavior of a metastable system is the so-called \emph{transfer operator} (the evolution operator of the Fokker--Planck equation), which acts on functions on the state space. The time scale separation we rely on here implies a spectral gap for this operator. This fact has been exploited to find low-dimensional representations of the global kinetics in form of Markov chains whose (discrete) states represent the metastable sets while the transition probabilities between the states approximate the jump statistics between the sets on long time scales. Under the name ``Markov State Models'' (MSM), this approach has led to a variety of methods \cite{A19-1,SchuetteSarich2013} with broad application, e.g.,  in molecular dynamics, cf.~\cite{direct-hmc,pande2010everything,msm_milestoning,CN14}. 
This reduction comes with a price: Since the relaxation kinetics is described just by jumps between the metastable sets in a (finite) discrete state space, any information about the transition process and its dynamical features is lost. A variety of approaches have been developed for complementing the MSM approach appropriately \cite{discreteTPT}, but a continuous (in time and space) low-dimensional effective description based on MSMs allowing to understand the transition mechanism is infeasible.

In another branch of the literature, again heavily influenced by molecular dynamics applications, model reduction techniques have been developed that assume the existence of a low-dimensional \emph{reaction coordinate} or \emph{order parameter} in order to construct an effective dynamics or kinetics: Examples are free energy based techniques~\cite{Torrie1977,metadynamics-review}, trajectory-based sampling techniques~\cite{Elber,NE-FFS,TIS,tika2013}, methods based on diffusive processes~\cite{BeHu10,ZHS16,PaSt08}, and many more that rely on the assumption that the reaction coordinates are known.
The problem of actually \emph{constructing} good reaction coordinates remains an area of ongoing research \cite{RC_overview}, to which this paper contributes. Typically, reaction coordinates are either postulated using system specific expert knowledge~\cite{CaTh93,SoEtAl96}, an approximation to the dominant eigenfunctions of the transfer operator is sought~\cite{SchuetteSarich2013,CN14,tika2013}, or machine learning techniques are proposed \cite{Dinner}. Froyland et al.~\cite{FGH14a} show that these eigenfunctions are indeed optimal --- in the sense of optimally representing the slow dynamics --- but for high dimensional systems computational reaction coordinate identification still is often infeasible. In the context of \emph{transition path theory}~\cite{VE06}, the committor function is known to be an ideal~\cite{LuVE14} reaction coordinate. In \cite{Henkelman2012}, the authors construct a level set of the committor using support vector machines, but the computation of reaction coordinates is infeasible for high-dimensional systems. The main problem in computing reaction coordinates for high-dimensional metastable systems results from the fact that all of these algorithms try to solve a \emph{global} problem in the entire state space that cannot be decomposed easily into purely local computations.

In this article, we elaborate on the definition, existence and algorithmic identification of reaction coordinates for metastable systems: We define reaction coordinates as a small set of \emph{nonlinear} coordinates on which a reduced system~\cite{LeLe10, ZHS16} can be defined having the same dominant time scales (in terms of transfer operator eigenvalues) as the original system. We then consider a low-dimensional state space on which the reduced dynamics is a Markov process. Thus, our approach utilizes concepts and transfer operator theory developed previously, but in our case the projected transfer operator is still \emph{infinite-dimensional}, in stark contrast to its reduction to a stochastic matrix in the MSM approach. 

The contribution of this paper is twofold. First, we develop a conceptual framework that identifies \emph{good} reaction coordinates as the ones that parameterize a low-dimensional \emph{transition manifold}~$\mathbb{M}$ in the function space $L^1$, which is the natural state space of the Fokker--Planck equation $\dot u = \mathcal{L}u$ associated with the dynamics. The property which defines~$\mathbb{M}$ is that, on moderate time scales $t_\text{fast} < t \ll t_\text{slow}$, the \emph{transition density functions} of the dynamics concentrate around $\mathbb{M}$. We provide evidence that such an $\mathbb{M}$ indeed exists due to metastability and the existence of transition pathways. Crucially, the dimension of $\mathbb{M}$ is often lower then the number of dominant eigenfunctions.

Second, we present an algorithm to construct approximate reaction coordinates. Our algorithm is data-driven and \emph{fully local}, thus circumventing the main problem of previously proposed algorithms: In order to compute the value of the desired reaction coordinate $\xi$ at a location $x$ in the state space $\X$, only the ability to simulate short trajectories initialized at $x$ is needed. In particular, we assume no a priori knowledge of metastable sets, no global equilibration, and we do not need to resolve the slow time scales numerically. The algorithm is built on two pillars:
\begin{enumerate}
\item The simulation time scale~$t$ can be chosen a lot smaller than the dominant time scales~$t_\text{slow}$ of the system, such that it is feasible to simulate many short trajectories of length~$ t $.
\item We utilize \emph{embedding techniques} inspired by the seminal work of Whitney~\cite{Whi36} and the recent work \cite{DHZ16} that allows one to take almost any mapping into a Euclidean space of more than twice the dimension of the manifold~$\mathbb{M}$ and to obtain a one-to-one image of it.
\end{enumerate}
These two pillars together with the low-dimensionality of~$\mathbb{M}$ imply that we can represent the image of the reaction coordinate in a space with moderate (finite) dimension. Then, we can use established \emph{manifold learning} techniques~\cite{NLCK06, CKLMN08, SEKC09} to obtain a parametrization of the manifold in the embedding space and pull this parametrization back to the original state space, hence obtaining a reaction coordinate.

The locality of the algorithm also implies that reaction coordinates are only computed in the region of state space where sampled points are available. This is a common issue with manifold learning algorithms; here it manifests as the transition manifold being reliably learned only in regions we have good sampling coverage of. However, recently several methods have appeared in the literature that allow a fast exploration of the state space. These methods do not provide equilibrium sampling, but instead try to rapidly cover the essential part of the state space with sampling points. This can be achieved with enhanced sampling methods such as Umbrella Sampling \cite{wham1, wham2}, Metadynamics \cite{meta1, meta2}, Blue-Moon sampling \cite{blue}, Adaptive Biasing Force method \cite{darve2008adaptive}, or Temperature-Accelerated Molecular Dynamics \cite{maragliano2006temperature}, as well as trajectory-based techniques like Milestoning \cite{faradjian2004computing}, Transition Interface Sampling \cite{moroni2004investigating}, or Forward Flux Sampling \cite{becker2012non}. Alternatively, several techniques like the equation-free approach \cite{eqfree-algo}, the heterogeneous multiscale method (HMM) \cite{weinan2003heterognous} and methods based on diffusion maps \cite{chiavazzo2016imapd} have been developed to utilize short unbiased MD trajectories for extracting information that allows much larger timesteps. This can be combined with reaction coordinate based effective dynamics \cite{ZHS16, ZS17}.

In principle, the method we present in this article may be combined with any enhanced sampling technique in order to generate sampling points that cover a large part of the state space. For simplicity, we will use long MD trajectories to generate our sampling points, but we do not require that the points are distributed according to an equilibrium distribution.

The paper is organized as follows: Section~\ref{sec:Transfer operators and their properties} introduces transfer operators, which describe the global kinetics of the stochastic process. Based on these transfer operators, we define metastability, i.e.\ the existence of dominant time scales. In Section~\ref{sec:projTO}, we describe the model reduction techniques \emph{Markov state modeling} and \emph{coordinate projection} that are designed to capture the dominant time scales of metastable systems. Furthermore, we characterize \emph{good} reaction coordinates. In the first part of Section~\ref{sec:identifying RC}, we show that our dynamical assumption ensures the existence of good reaction coordinates, then in the second part we describe our approach to compute them. Several numerical examples are given in Section~\ref{sec:Numerical Examples}. Concluding remarks and an outlook are provided in Section~\ref{sec:Conclusion}.

\section{Transfer operators and their properties}
\label{sec:Transfer operators and their properties}

As mentioned in the introduction, global properties of dynamical systems such as meta\-stable sets or a partitioning into fast and slow subprocesses can be obtained using transfer operators associated with the system and their eigenfunctions. In this section, we will introduce different transfer operators needed for our considerations.

\subsection{Transfer operators}

In what follows, $ \mathsf{P}[\,\cdot \mid \mathfrak{E}] $ denotes probabilities conditioned on the event $ \mathfrak{E} $ and $ \mathsf{E[\cdot \mid \mathfrak{E}]} $ the expectation value. Furthermore, $ \{\mathbf{X}_t\}_{t \ge 0} $ is a stochastic process defined on a state space $ \X \subset \R^n $.

\begin{definition}[Transition density function]
Let $ \mathbb{A} $ be any measurable set, then the \emph{transition density function} $ p^t \colon \X \times \X \to \R_{\ge 0} $ of a time-homogeneous stochastic process $ \{\mathbf{X}_{t}\}_{t \ge 0} $ is defined by
\begin{equation*}
    \mathsf{P}[\mathbf{X}_t \in \mathbb{A} \mid \mathbf{X}_0 = x] = \intop_{\mathbb{A}} p^t(x,y) \, \mathrm{d}y.
\end{equation*}
That is, $ p^t(x,y) $ is the conditional probability density of $ \mathbf{X}_t = y $ given that $ \mathbf{X}_0 = x $.
\end{definition}

With the aid of the transition density function, we can now define transfer operators. Note, however, that the transition density is in general not known explicitly and needs to be estimated from simulation data.
In what follows, we assume that there is a \emph{unique} equilibrium density $ \varrho $ that is invariant under $ \{\mathbf{X}_t\}_{t \ge 0} $, that is, it satisfies
\[
\varrho(x) = \intop_\X p^t(y,x) \varrho(y) \, \mathrm{d}y,
\]
a.e.~on $\X$. Let $\mu$ denote the associated invariant measure $ \mathrm{d}\mu = \varrho \, \mathrm{d}x $.

\begin{definition}[Transfer operators] \label{def:transferops}
Let $ p \in L^1(\X) $ be a probability density\footnote{We denote by~$L^q$ the space (equivalence class) of $q$-integrable functions with respect to the Lebesgue measure. $L^q_{\nu}$ denotes the same space of function, now integrable with respect to the measure~$\nu$.}, $ u = p / \varrho \in L_\mu^1(\X) $ be a probability density with respect to the equilibrium density $ \varrho $,  and $ f \in L^{\infty}(\X) $ an observable of the system. For a given lag time $ t $:
\begin{enumerate}[(a)]
\item The \emph{Perron--Frobenius operator} $ \mathcal{P}^t \colon L^1(\X) \to L^1(\X) $ is defined by the unique linear extension of
\begin{equation*}
    \mathcal{P}^t p(x) = \intop_{\X} p^t(y,x) \, p(y) \, \mathrm{d}y
\end{equation*}
to~$L^1(\X)$.
\item The \emph{Perron--Frobenius operator $ \mathcal{T}^t \colon L_\mu^1(\X) \to L_\mu^1(\X) $ with respect to the equilibrium density} is defined by the unique linear extension of
\begin{equation*}
    \mathcal{T}^t u(x) = \intop_{\X} \frac{\varrho(y)}{\varrho(x)} \, p^t(y, x) \, u(y) \,\mathrm{d}y
\end{equation*}
to~$L^1_{\mu}(\X)$.
\item The \emph{Koopman operator} $ \mathcal{K}^t \colon L^{\infty}(\X) \to L^{\infty}(\X) $ is defined by
\begin{equation} \label{eq:Koopman operator}
    \mathcal{K}^t f(x) = \intop_{\X} p^t(x,y) \, f(y) \,\mathrm{d}y
                       = \mathsf{E}[f(\mathbf{X}_t) \mid \mathbf{X}_0 = x].
\end{equation}
\end{enumerate}
All these are well-defined non-expanding operators on the respective spaces.
\end{definition}

The equilibrium density $ \varrho $ satisfies $ \mathcal{P}^t \varrho=\varrho$, that is, $ \varrho $ is an eigenfunction of $ \mathcal{P}^t $ with associated eigenvalue $ \lambda_0 = 1 $. The definition of $ \mathcal{T}^t $ relies on $\varrho$, we have ~$\varrho\, \mathcal{T}^tu = \mathcal{P}^t (u \varrho)$.

Instead of their natural domains from Definition~\ref{def:transferops}, all our transfer operators are considered on the following Hilbert spaces:~$\mathcal{P}^t:L^2_{1/\mu}(\X) \to L^2_{1/\mu}(\X)$, $\mathcal{T}^t:L^2_{\mu}(\X) \to L^2_{\mu}(\X)$, and~$\mathcal{K}^t:L^2_{\mu}(\X) \to L^2_{\mu}(\X)$. They are still well-defined non-expansive operators on these spaces~\cite{BaRo95,SchCa92,KNKWKSN17}.

Furthermore, we will need the notion of reversibility for our considerations. Reversibility means that the process is statistically indistinguishable from its time-reversed counterpart.

\begin{definition}[Reversibility]
A system is said to be \emph{reversible} if the detailed balance condition
\begin{equation*}
    \varrho(x) \, p^t(x,y) = \varrho(y) \, p^t(y,x)
\end{equation*}
is satisfied for all $ x, y \in \X$.
\end{definition}

In what follows, we will assume that the system is reversible.

One prominent example for a class of SDEs satisfying uniqueness of the equilibrium density and reversibility is given by  
\begin{equation}
    \mathrm{d}\mathbf{X}_t = -\nabla V(\mathbf{X}_{t})\,\mathrm{d}t + \sqrt{2\beta^{-1}} \, \mathrm{d}\mathbf{W}_t\,.
    \label{eq:overdampedLangevin}
\end{equation}
Here,~$V$ is called the potential,~$\beta$ is the non-dimensionalized inverse temperature, and $\mathbf{W}_t$ is a standard Wiener process. The process generated by~\eqref{eq:overdampedLangevin} is ergodic and thus admits a unique positive equilibrium density, given by $\varrho(x)=\exp(-\beta V(x))/Z$, under mild growth conditions on the potential~$V$~\cite{MaSt02,MaStHi02}. Note that the subsequent considerations hold for all stochastic processes that satisfy reversibility and ergodicity with respect to a unique positive invariant measure and are \emph{not} limited to the class of dynamical systems given by (\ref{eq:overdampedLangevin}). See \cite{SchuetteSarich2013} for a discussion of a variety of stochastic dynamical systems that have been considered in this context.

As a result of the detailed balance condition, the Koopman operator $ \mathcal{K}^t $ and the Perron--Frobenius operator with respect to the equilibrium density $ \mathcal{T}^t $ become identical and we obtain
\begin{equation*}
    \innerprod{\mathcal{P}^t f}{g}_{1/\mu} = \innerprod{f}{\mathcal{P}^t g}_{1/\mu}
    \quad \text{and} \quad
    \innerprod{\mathcal{T}^t f}{g}_\mu        = \innerprod{f}{\mathcal{T}^t g}_\mu\,,
\end{equation*}
i.e.\ all the transfer operators become self-adjoint on the respective Hilbert spaces from above.
Here~$\langle\cdot,\cdot\rangle_{\mu}$ and~$\langle\cdot,\cdot\rangle_{1/\mu}$ denote the natural scalar products on the weighted spaces~$L^2_{\mu}$ and~$L^2_{1/\mu}$, respectively. 

\subsection{Spectral decomposition}

Due to the self-adjointness, the eigenvalues $ \lambda_i^t $ of $ \mathcal{P}^t $ and $ \mathcal{T}^t $ are real-valued and the eigenfunctions form an orthogonal basis with respect to $ \innerprod{\cdot}{\cdot}_{1/\mu} $ and $ \innerprod{\cdot}{\cdot}_\mu $, respectively. In what follows, we assume that the spectrum of $ \mathcal{T}^t $ is purely discrete given by (infinitely many) isolated eigenvalues. This assumption is made for the sake of simplicity. It is actually not required for the rest of our considerations; it would be sufficient to assume that the spectral radius $R$ of the essential spectrum of $ \mathcal{T}^t $ is strictly smaller than 1, and some isolated eigenvalues of modulus larger than $R$ exist. It has been shown that this condition is satisfied for a large class of metastable dynamical systems, see \cite[Sec. 5.3]{SchuetteSarich2013} for details.
For example, the process generated by~\eqref{eq:overdampedLangevin} has purely discrete spectrum under mild growth and regularity assumptions on the potential $V$.

Under this condition, ergodicity implies that the dominant eigenvalue $ \lambda_0 $ is the only eigenvalue with absolute value $ 1 $ and we can thus order the eigenvalues so that
\begin{equation*}
    1 = \lambda_0^t > \lambda_1^t \ge \lambda_2^t \ge \dots.
\end{equation*}
The eigenfunction of~$\mathcal{T}^t$ corresponding to $ \lambda_0 = 1 $ is the constant function $ \varphi_0 = \mathds{1}_{\X} $. Let $ \varphi_i $ be the normalized eigenfunctions of $ \mathcal{T}^t $, i.e.~$ \innerprod{\varphi_i}{\varphi_j}_\mu = \delta_{ij} $, then any function $ f \in L_{\mu}^2(\X) $ can be written in terms of the eigenfunctions as $ f = \sum_{i=0}^\infty \innerprod{f}{\varphi_{i}}_\mu \, \varphi_i $. Applying $ \mathcal{T}^t $ thus results in
\begin{equation*}
    \mathcal{T}^t f = \sum_{i=0}^\infty \lambda_i^t \, \innerprod{f}{\varphi_i}_\mu \, \varphi_i.
\end{equation*}
For more details, we refer to~\cite{KNKWKSN17} and references therein.

\subsection{Implied time scales} \label{ssec:implied_ts}

For some $ d \in \mathbb{N} $, we call the $ d+1 $ dominant eigenvalues $ \lambda_0^t, \dots, \lambda_d^t $ of $ \mathcal{T}^t $ the \emph{dominant spectrum} of $ \mathcal{T}^t $, i.e.
\begin{equation*}
    \sigma_\textrm{dom}(\mathcal{T}^t) := \{ \lambda_0^t, \dots, \lambda_d^t \}.
\end{equation*}
Usually, $ d $ is chosen in such a way that there is a \emph{spectral gap} after $ \lambda_d^t $, i.e.~$ 1-\lambda_d^t \ll \lambda_d^t - \lambda_{d+1}^t $. The \emph{(implied) time scales} on which the associated dominant eigenfunctions decay are given by
\begin{equation} \label{eq:implied time scales}
    t_i = -t/\log(\lambda_i^t).
\end{equation}
If $ \mathcal{T}^t $ is a semigroup of operators, then there are $ \kappa_i \le 0 $ with $ \lambda_i^t = \exp(\kappa_i t) $ such that $t_i = -\kappa_i^{-1} $ holds. Assuming there is a spectral gap, the dominant time scales satisfy $ t_1 \ge \ldots \ge t_d \gg t_{d+1} $. These are the time scales of the \emph{slow} dynamical processes, also called \emph{rare events}, which are of primary interest in applications. The other, \emph{fast} processes are regarded as fluctuations around the relative equilibria (or \emph{metastable states}) between which the relevant slow processes travel.

\section{Projected transfer operators and reaction coordinates}
\label{sec:projTO}

The purpose of dimension reduction in molecular dynamics is to find a reduced dynamical model that captures the dominant time scales of the system correctly while keeping the model as simple as possible. In this section, we will introduce two different projections and the corresponding projected transfer operators. The goal is to find suitable projections onto the slow processes.

\subsection{Galerkin projections and Markov state models}

One frequently used approach to obtain a reduced model is \emph{Markov state modeling}. The goal is to find a model that is as simple as possible and yet correctly reproduces the dominant time scales. Given a fixed $ t > 0 $, most authors \cite{PNAS09,pande2010everything} refer to a Markov state model (MSM) as a matrix $ T^t \in \R^{(d+1) \times (d+1)} $ such that
\begin{equation} \label{eq:MSMapprox}
    \sigma_\textrm{dom}(\mathcal{T}^t) \approx \sigma_\textrm{dom}(T^t),
\end{equation}
and it has been studied in detail under which condition this can be achieved~\cite{DjurdjevacSarichSchuette2012,SarichNoeSchuette2010}. 

There are different ways of constructing an MSM, maybe the most intuitive one is also the simplest: Let the entries of $ T^t $ be the transition rates between metastable sets. A typical molecular system with $ d $ dominant time scales will have $ d+1 $ metastable sets $\mathbb{C}_1, \dots, \mathbb{C}_{d+1} $ (also called \emph{cores}) and its dynamics is characterized by transitions between these sets and fluctuations inside the sets (see Figure~\ref{fig:bananapot} for an illustration). Since the fluctuations are on faster time scales, we neglect them by setting \cite{direct-hmc}
\begin{equation}
    T^t_{\textrm{core}, ij} = \mathsf{P}_{\mu}[\mathbf{X}_t\in\mathbb{C}_j\,\big\vert\, \mathbf{X}_0\in\mathbb{C}_i]\,,
    \label{eq:coreMSMprob}
\end{equation}
where $ \mathsf{P}_{\mu} $ denotes the probability measure conditioned to the initial condition $ \mathbf{X}_0 $ being distributed according to $\mu$. Thus,~$T^t_{\textrm{core}, ij}$ is the probability that the process in equilibrium jumps to the metastable set~$\mathbb{C}_j$ in time~$t$, given that it started in the metastable set~$\mathbb{C}_i$. Note that~\eqref{eq:coreMSMprob} can be equivalently rewritten as
\begin{equation}
    T^t_{\textrm{core}, ij} 
        = \frac{\innerprod{\mathcal{T}^t\mathds{1}_{\mathbb{C}_i}}{\mathds{1}_{\mathbb{C}_j}}_\mu}
               {\innerprod{\mathds{1}_{\mathbb{C}_i}}{\mathds{1}_{\mathbb{C}_i}}_\mu}\,,
    \label{eq:coreMSMfun}
\end{equation}
where $ \mathds{1}_{\mathbb{C}_i} $ is the characteristic function of the set $ \mathbb{C}_i $.

Equation~\eqref{eq:coreMSMfun} readily suggests that $ T^t_\textrm{core} $ is a projection of the transfer operator $ \mathcal{T}^t $, namely its \emph{Galerkin projection} onto the space spanned by the characteristic functions $ \mathds{1}_{\mathbb{C}_1},\ldots, \mathds{1}_{\mathbb{C}_{d+1}} $ \cite{direct-hmc}.

\begin{definition}[Galerkin projection]\label{def:galerkin_projection}
Given a set of basis functions $\psi_1,\ldots,\psi_m\in L^2_{\mu}(\X)$, let $\mathbb{V} := \mathrm{span}\{\psi_1,\ldots,\psi_m\}$ and $ \psi:=(\psi_1,\ldots,\psi_m)^\intercal $. The projection to~$\mathbb{V}$ or, equivalently, to $\psi$, $\Pi_{\mathbb{V}} = \Pi_{\psi} \colon L^2_{\mu}(\X)\to \mathbb{V}$ is defined as
\begin{equation*}
    \innerprod{\Pi_{\psi}f - f}{g}_\mu = 0 \qquad \forall\, f\in L^2_{\mu}(\X),\, \forall\,g\in\mathbb{V}\,.
\end{equation*}
The residual projection is given by~$\Pi_{\psi}^{\perp} = \mathrm{Id} - \Pi_{\psi}$, where $\mathrm{Id}$ is the identity.
The Galerkin projection of~$\mathcal{T}^t$ to~$\mathbb{V}$ is given by the linear operator~$T^t \colon \mathbb{V} \to \mathbb{V}$ satisfying
\begin{equation*}
    \innerprod{\mathcal{T}^t f - T^t f}{g}_\mu = 0 \qquad \forall\, f,g\in\mathbb{V}\,.
\end{equation*}
\end{definition}

Equivalently, $ T^t = \Pi_{\psi} \mathcal{T}^t $. We also denote the extension of $ T^t $ to the whole $ L^2_{\mu}(\X) $, given by $ \Pi_{\psi}\mathcal{T}^t\Pi_{\psi} $, by $ T^t $. Furthermore, we denote the matrix representation of $ T^t $ with respect to the basis $ (\psi_0,\dots,\psi_d) $ by $ T^t $ as well. Either it will be clear from the context which of the objects $T^t$ is meant or it will not matter; e.g., the dominant spectrum is the same for all of them.

We see that~$T_\textrm{core} $ is the matrix representation of the Galerkin projection with respect to the basis functions $ \innerprod{\mathds{1}_{\mathbb{C}_i}}{\mathds{1}_{\mathbb{C}_i}}_{\mu}^{-1} \mathds{1}_{\mathbb{C}_i}$,  $i=1,\ldots,d+1$. More general MSMs can be built by Galerkin projections of the transfer operator to spaces spanned by other --- not necessarily piecewise constant --- basis functions \cite{Web06,msm_milestoning, WeberFackeldeySchuette2017,KKS16,KNKWKSN17,tika2013,variational2013}. However, in some of these methods, one also often loses the interpretation of the entries of the matrix $ T^t $ as probabilities.

Ultimately, the best MSM in terms of approximation quality in~\eqref{eq:MSMapprox} is given by the Galerkin projection of $ \mathcal{T}^t $ onto the space spanned by its dominant eigenfunctions $ \varphi_0 $, $\dots $, $\varphi_d $. This space is invariant under $ \mathcal{T}^t $ since $\mathcal{T}^t \varphi_i = \lambda_i^t \varphi_i $ and the dominant eigenvalues (and hence the time scales) are the same for the MSM and for $ \mathcal{T}^t $. Due to the curse of dimensionality, however, the computation of the eigenfunctions $ \varphi_i $ is in general infeasible for high-dimensional problems.
\begin{remark}
There are quantitative results assessing the error in~\eqref{eq:MSMapprox} of the MSM in terms of the projection errors~$\|\Pi_{\psi}^{\perp}\varphi_i\|_{L^2_{\mu}}$, $i=0,\ldots,d$, cf.~\cite[Section 5.3]{SchuetteSarich2013}. One can obtain a weaker, but similar result from our Lemma~\ref{lem:EVapprox} in the next section.
\end{remark}

\subsection{Coordinate projections and effective transfer operators}
\label{ssec:coord proj}

While the MSMs from above successfully reproduce the dominant time scales of the original system, they often discard all other information about the system, such as the transition paths between metastable sets. Minimal coordinates that describe these transitions are called \emph{reaction coordinates} and reducing the dynamics onto these coordinates yields \emph{effective dynamics}~\cite{LeLe10,ZHS16}. The goal of the previous section --- namely to retain the dominant time scales of the original dynamics in a reduced model --- can now be reformulated for this lower-dimensional effective dynamics or, equivalently, for its (effective) transfer operator.

Let $ \xi \colon \X \to \R^k $ be a $ C^1 $ function, where $ k \le n $. Let $ \mathbb{L}_z = \{x \in \X \mid \xi(x) = z\} $ be the $ z $-level set of $ \xi $. The so-called \emph{coarea formula}~\cite[Section 3.2]{Fed69}, which can be considered as a nonlinear variant of Fubini's theorem, splits integrals over~$\X$ into consecutive integrals over level sets of~$\xi$ and then over the range of~$\xi$. For $f\in L^2_{\mu}(\X)$, we have\footnote{The coarea formula holds for~$L^1$ functions, but~$L^2_{\mu}\subset L^1_{\mu}$, since~$\mu$ is a probability measure (i.e., it is finite).}
\begin{equation}
\int_{\X} f(x)\,d\mu(x) = \int_{\xi(\X)} \int_{\mathbb{L}_z} f(x')\varrho(x')\det\left(\nabla\xi(x')^{\intercal} \nabla \xi(x')\right)^{-1/2}\,d\sigma_z(x')\,dz\,,
\label{eq:coarea}
\end{equation}
where~$ z = \xi(x) $ and~$ \sigma_z $ is the surface measure on~$ \mathbb{L}_z $. The \emph{coordinate projection}, defined next, averages a given function along the level sets of a coordinate function~$\xi$.
\begin{definition}[Coordinate projection]
For $ f \in L^2_\mu(\X) $, we define
\begin{eqnarray}
    P_\xi f(x) & = & \int_{\mathbb{L}_z} f(x') \,d\mu_z(x') \label{eq:Pxi1}\\
               & = & \frac{1}{\Gamma(z)}\int_{\mathbb{L}_z}f(x')\varrho(x')\det(\nabla\xi(x')^{\intercal}\nabla\xi(x'))^{-1/2}\,d\sigma_z(x'), \label{eq:Pxi2}
\end{eqnarray}
where~$ \mu_z $ is a probability measure on $ \mathbb{L}_z $ with density $ \frac{\varrho}{\Gamma(z)}\det(\nabla\xi^{\intercal} \nabla \xi)^{-1/2}$ with respect to $ \sigma_z $. Here, $\Gamma(z)$ is just the normalization constant so that $ \mu_z $ becomes a probability measure. The residual projection is given by $ P_\xi^\perp = \mathrm{Id} - P_{\xi} $.
\end{definition}
To get a better feeling for the action of~$P_{\xi}$, note that~$P_{\xi}f(x)$ is the expectation of $ f(\mathbf{x}') $ with respect to $ \mu $ conditional to $ \xi(\mathbf{x}') = \xi(x) $, i.e.
\begin{equation*}
    P_{\xi}f(x) = \mathsf{E}_{\mu}\left[f(\mathbf{x}')\,\big\vert\, \xi(\mathbf{x}') = \xi(x)\right]\,.
\end{equation*}
Or, in other words, $ \mu_z $ is the marginal of $\mu$ conditional to $ \xi(x) = z $. Note, in particular, that~$P_{\xi}f$ is itself a function on~$ \X $, but it is constant on the level sets of~$\xi$, and thus let us set~$\widehat{P_{\xi}f}(\xi(x)) = P_{\xi}f(x)$ for~$x\in\mathbb{L}_{\xi(x)}$. It follows from the coarea formula~\eqref{eq:coarea} and~\eqref{eq:Pxi2} that
\begin{equation}
    \int_{\X} f(x)\,d\mu(x) = \int_{\xi(\X)} \Gamma(z)\widehat{P_{\xi}f}(z)\,dz\,.
    \label{eq:splitintPxi}
\end{equation}
Next, we state some properties of the coordinate projection.

\begin{proposition} \label{prop:PxiProperties}
The coordinate projection has the following properties.
\begin{enumerate}[(a)]
\item $P_{\xi}$ is a linear projection, i.e.~$P_{\xi}^2 = P_{\xi}$.
\item $P_{\xi}$ is self-adjoint with respect to $ \innerprod{\cdot}{\cdot}_\mu $.
\item $P_{\xi} \colon L^2_{\mu}(\X) \to L^2_{\mu}(\X)$ is orthogonal, hence non-expansive, i.e.~$\|P_{\xi}f\|_{L^2_{\mu}} \le \|f\|_{L^2_{\mu}}$.
\end{enumerate}
\end{proposition}
\begin{proof}
See Appendix~\ref{app:Pxi}.
\end{proof}

We use the coordinate projection to describe the dynamics-induced propagation of reduced distributions with respect to the variable~$ \xi $. To this end, we define the \emph{effective transfer operator} $\mathcal{T}_{\xi}^t \colon L^2_{\mu}(\X) \to L^2_{\mu}(\X)$ by
\begin{equation}
    \mathcal{T}_\xi^t = P_\xi\mathcal{T}^t P_\xi.
\end{equation}
We immediately obtain from the self-adjointness of $ \mathcal{T}^t $ (see Section~\ref{sec:Transfer operators and their properties}) and Proposition~\ref{prop:PxiProperties} (b) that~$\mathcal{T}_{\xi}^t$ is a self-adjoint operator on $L^2_{\mu}(\X)$. Moreover, $\|\mathcal{T}^t\|_{L^2_{\mu}} \le 1$ and Proposition~\ref{prop:PxiProperties}~(c) imply that $\|\mathcal{T}_{\xi}^t\|_{L^2_{\mu}}\le 1$. Thus, the spectrum of the effective transfer operator lies in the interval $[-1,1]$, too.

Returning to the purpose of these constructions, we call $ \xi $ a \emph{good reaction coordinate}~if
\begin{equation} \label{eq:EffectiveApprox}
    \sigma_\textrm{dom}(\mathcal{T}^t) \approx \sigma_\textrm{dom}(\mathcal{T}_{\xi}^t).
\end{equation}
While the previously introduced Markov state model $ T^t $ obtained by the Galerkin projection was approximating the dominant spectrum of the original transfer operator by a finite-dimensional operator (i.e.\ a matrix), the effective transfer operator still acts on an infinite-dimensional space. The reduction lies in the fact that~$\mathcal{T}^t$ operates on functions over $ \X \subseteq \R^n $, but the effective transfer operator $\mathcal{T}_\xi^t$ operates \emph{essentially} on functions over $ \xi(\X) \subset \R^k $, although we embed those into~$\X$ through the level sets of~$\xi$.

As mentioned above, a Galerkin projection of the transfer operator onto its dominant eigenfunctions is a perfect MSM. In the same vein, we ask here how we can characterize a good reaction coordinate. We can make use of the following general result.
\begin{lemma} \label{lem:EVapprox}
Let~$\mathbb{H}$ be a Hilbert space with scalar product~$\langle\cdot,\cdot\rangle$, and associated norm~${\|\cdot\|}$, let~$Q:\mathbb{H}\to\mathbb{H}$ be some orthogonal projection on a linear subspace of~$\mathbb{H}$, with~$Q^{\perp} = \mathrm{Id}-Q$. Let~$T:\mathbb{H}\to\mathbb{H}$ be a self-adjoint non-expansive linear operator, and~$u$ with~$\|u\|=1$ its eigenvector, i.e.,~$Tu = \lambda u$ for some~$\lambda\in\mathbb{R}$. If~$\|Q^{\perp}u\|<\varepsilon$, then~$T_Q:=QTQ$ has an eigenvalue~$\lambda_Q\in\mathbb{R}$ with~$|\lambda - \lambda_Q|<\varepsilon/\sqrt{1-\varepsilon^2}$.
\end{lemma}
\begin{proof}
Using $ Q = \mathrm{Id} - Q^{\perp} $, we have
\begin{equation*}
    T_Q Q u
    = Q T \underbrace{QQ}_{=Q} u
    = Q T u - \underbrace{Q T Q^{\perp} u}_{=:-\zeta}
    = \lambda Q u + \zeta,
\end{equation*}
where $\|\zeta\|\le \|Q^{\perp} u\| < \varepsilon$ since~$Q$ and~$T$ are non-expanding. Thus,~$u' := Q u/\|Qu\|$ satisfies~$T_Q u' = \lambda u' + \zeta/\|Qu\|$, and the orthogonality of~$Q$ gives~$\|Qu\|>\sqrt{1-\varepsilon^2}$. Now, any orthogonal projection is self-adjoint, as is shown in the proof of Proposition~\ref{prop:PxiProperties}, hence the operator~$QTQ$ is self-adjoint, too, and thus normal. From the theory of pseudospectra for normal operators~\cite[Theorems 2.1,~2.2, and \S 4]{TrEm05}, we know that if~$ \|T_Q u' - \lambda u'\| < \varepsilon/\sqrt{1-\varepsilon^2} $, then~$ T_Q $ has an eigenvalue~$\lambda_Q\in\mathbb{R}$ in the~$\varepsilon/\sqrt{1-\varepsilon^2}$-neighborhood of~$\lambda$.
\end{proof}
With~$\mathbb{H}=L^2_{\mu}$,~$Q=P_{\xi}$, and~$T = \mathcal{T}^t$ we immediately obtain the following result.
\begin{corollary} \label{cor:characRC}
As before, let $ \lambda_i^t $ and $ \varphi_i $, $ i = 0, \dots, d $, denote the dominant eigenvalues and eigenfunctions of $ \mathcal{T}^t $, respectively. For any given $ i $, if $ \|P_{\xi}^{\perp}\varphi_i\|_{L^2_\mu} < \varepsilon $, then there is an eigenvalue~$\tilde{\lambda}_i^t$ of~$ \mathcal{T}^t_\xi $ with $ |\lambda_i^t - \tilde{\lambda}_i^t| < \varepsilon/\sqrt{1-\varepsilon^2} $.
\end{corollary}

Corollary~\ref{cor:characRC} implies that if the projection error of \emph{all dominant} eigenfunctions is small, then $ \xi $ is a good reaction coordinate in the sense of~\eqref{eq:EffectiveApprox}. Very similar results are available for approximation of the eigenvalues of the infinitesimal generator of the Fokker--Planck equation associated with the transfer operator if the dynamical system under consideration is continuous in time \cite{ZS17}. 

Under which conditions is the projection error small? Let us consider the case where there are $ \tilde{\varphi}_i:\R^k \to \R $, $ i = 1, \dots, d $, such that $ \varphi_i(x) = \tilde{\varphi}_i(\xi(x)) $. We then say that \emph{$\varphi_i$ is a function of~$\xi$} or that \emph{$\xi$ parametrizes~$\varphi_i$}. If $\xi$ parametrizes~$\varphi_i$ perfectly, the projection error obviously vanishes. Thus, trivially, by choosing $ \xi = \varphi = (\varphi_1, \dots, \varphi_d)^{\intercal} $, we obtain a perfect reaction coordinate since with $\tilde{\varphi}_i(z):= z_i $ with $ \varphi_i = \tilde{\varphi}_i \circ \xi $. However, the eigenfunctions are \emph{global} objects, i.e., their computation is prohibitive in high dimensions. Since we are aiming at computing a reaction coordinate, we have to answer the question of whether there is a reaction coordinate $\xi$ that can be evaluated based on local computations only while it parametrizes the dominant eigenfunctions of~$\mathcal{T}^t$ well enough such that it leads to a small projection error. 
We will see next that this question can be answered by utilizing a common property of most metastable systems: The transitions between the metastable sets happen along so-called \emph{reaction pathways}, which imply the existence of \emph{transition manifolds} in the space of transition densities. A \emph{suitable} parametrization of this manifold results in a parametrization of the dominant eigenfunctions with a small error.

\section{Identifying good reaction coordinates}
\label{sec:identifying RC}

The goal is now to find a reaction coordinate $\xi$ that is as low-dimensional as possible and results in a good projected transfer operator in the sense of \eqref{eq:EffectiveApprox}. As we saw in the previous section, the condition $\|P_\xi^\perp \varphi_i\|_{L^2_{\mu}} \approx 0$ is sufficient. Thus, the idea to numerically seek $\xi$ that parametrizes the dominant eigenfunctions of $\mathcal{T}^t$ in the $ \| \cdot \|_{L^2_{\mu}}$-norm seems natural since this would lead to small projection error $ \|P_\xi^\perp\varphi_i\|_{L^2_{\mu}}$. 

In fact, eigenfunctions of transfer operators have been used before to compute reduced dynamics and reaction coordinates: In \cite{FGH14a}, methods to decompose multiscale systems into fast and slow processes and to project the dynamics onto these subprocesses based on eigenfunctions of the Koopman operator $ \mathcal{K}^t $ are proposed. In \cite{MHP17}, the dominant eigenfunctions of the transfer operator $ \mathcal{T}^t $, which due to the assumed reversibility of the system is identical to $ \mathcal{K}^t $, are shown to be good reaction coordinates. Also, commitor functions (introduced in Appendix~\ref{app:RCexist}), which are closely related to the dominant eigenfunctions, have been used as reaction coordinates in~\cite{Du98,LuVE14}.

However, we propose a fundamentally different path in defining and finding reaction coordinates, as working with dominant eigenfunctions has two major disadvantages:
\begin{enumerate}
\item The eigenproblem is \emph{global}. Thus if we wish to learn the value of an eigenfunction $\varphi_i$ at only one location $x\in\mathbb{X}$, we need an approximation of the transfer operator $\mathcal{T}_t$ that has to be accurate on all of $\mathbb{X}$. The computational effort to construct such an approximation grows exponential with $\dim(\mathbb{X})$, this is  the \emph{curse of dimensionality}. There have been attempts to mitigate this \cite{Web06,JuKo09,Web12}, but we aim to circumvent this problem entirely. Given two points $x,y \in \mathbb{X}$, we will decide whether $\xi(x)$ is close to $\xi(y)$ or not by using only local computations around $x$ and $y$ (i.e. samples from the transition densities $p^t(x,\cdot)$ and $p^t(y,\cdot)$ for moderate $t$). 
\item The number of dominant eigenfunctions $ (d + 1) $ equals the number of meta\-stable states, and this number can be much larger than the dimension of the transition manifold. This fact is illustrated in Example~\ref{ex:LemonSlice} below.
\end{enumerate}

\begin{example} \label{ex:LemonSlice}
Let us consider a diffusion process of the form \eqref{eq:overdampedLangevin} with the circular multi-well potential shown in Figure~\ref{fig:minimalreactioncoordinate}. Choosing a temperature that is not high enough for the central potential barrier to be overcome easily, transitions between the wells typically happen in the vicinity of a one-dimensional reaction pathway, the unit circle. The number of dominant eigenfunctions, however, corresponds to the number of wells. Nevertheless, projecting the system onto the unit circle would retain the dominant time scales of the system, cf.\ Section~\ref{sec:Numerical Examples}. \exampleSymbol

\begin{figure}[htb]
    \centering
    \begin{minipage}{0.45\textwidth}
        \centering
        \subfiguretitle{a)}
        \includegraphics[width=0.9\textwidth]{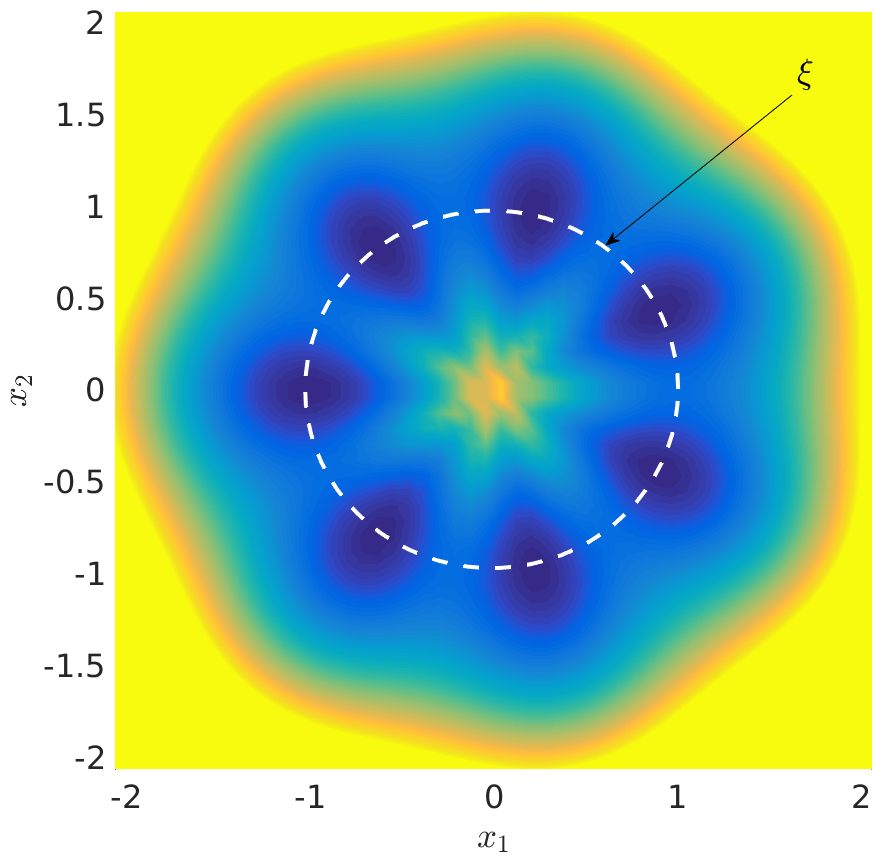}
    \end{minipage}
    \begin{minipage}{0.45\textwidth}
        \centering
        \subfiguretitle{b)}
        \includegraphics[width=0.9\textwidth]{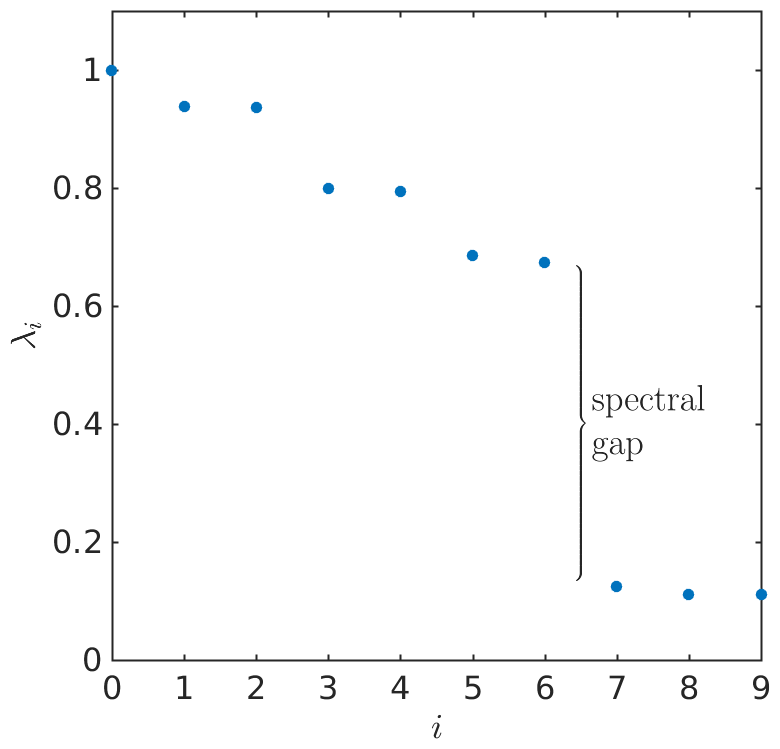}
    \end{minipage}
    \caption{a) Potential with seven wells and thus seven dominant eigenvalues, but only a one-dimensional reaction coordinate. The reaction pathway is marked by a dashed white line. b) Dominant eigenvalues of $ \mathcal{T}^t $ for $ t = 0.1 $ and $\beta=0.5$. The spectral gap is clearly visible.}
    \label{fig:minimalreactioncoordinate}
\end{figure}

\end{example}

\subsection{Parametrization of dominant eigenfunctions}

If the~$ (d+1) $ dominant eigenfunctions do not depend fully on the phase space $ \X $, a lower-dimensional and ultimately easier to find reaction coordinate suffices for keeping the eigenvalue approximation error \eqref{eq:EffectiveApprox} small. It is easy to see that if there exists a function $ \xi \colon \X \rightarrow \R^k $ for some $ k $ so that the eigenfunctions $ \varphi $ are constant on the level sets of $ \xi $, i.e., there exist functions $ \tilde{\varphi}_i \colon \R^k \rightarrow \R $, $ i = 1, \dots, d $ such that $ \varphi_i = \tilde{\varphi}_i \circ \xi $, then the projection error $ \|P_\xi^\perp \varphi_i \|_{L^2_{\mu}} $ is zero. A quantitative generalization of this is the statement that if the eigenfunctions $ \varphi_i $ are \emph{almost constant} on level sets of a~$ \xi $, then the projection error is small.

\begin{lemma}\label{lem:parametrizableeigenfunctions1}
Assume that there exists a function $ \xi \colon \X \rightarrow \R^k $ for some $ k $ and functions $ \tilde{\varphi}_i \colon \R^k \rightarrow \R $, $ i = 1, \dots, d $, with
\begin{equation} \label{eq:almostconstanteigenfunction}
    |\varphi_i(x) - \tilde{\varphi}_i(\xi(x))| \leq \varepsilon \quad \forall~x\in\X.
\end{equation}
Then $ \|P_\xi^\perp\varphi_i \|_{L^2_{\mu}} \leq  2\varepsilon $.
\end{lemma}
\begin{proof}
Assuming~\eqref{eq:almostconstanteigenfunction} holds, there exists a function $ c_i \colon \R \rightarrow \R $ with $ c_i(x)\leq 1~\forall x\in\X $ so that
\begin{equation*}
    \varphi_i(x) = \tilde{\varphi}_i(\xi(x)) + c_i(x)\varepsilon\,.
\end{equation*}
Thus, we have
\begin{align*}
    P_\xi\varphi_i(x) &= \int_{\mathbb{L}_{\xi(x)}} \Big(\tilde{\varphi}_i\big(\xi(x')\big)+c_i(x')\varepsilon\Big) d\mu_{\xi(x)}(x') \\
                      &=\tilde{\varphi}_i\big(\xi(x)\big) + \varepsilon\int_{\mathbb{L}_{\xi(x)}} c_i(x')d\mu_{\xi(x)}(x').
\end{align*}
For the projection error, we then obtain
\begin{align*}
    \|P_\xi \varphi_i - \varphi_i\|_{L^2_{\mu}} &\leq \|P_\xi\varphi_i - \tilde{\varphi_i}\circ\xi\|_{L^2_{\mu}} + \|\tilde{\varphi_i}\circ\xi - \varphi_i\|_{L^2_{\mu}} \\
                                    &\leq 2\varepsilon. \qedhere
\end{align*}
\end{proof}

\begin{remark}
From the proof we see that the pointwise condition (\ref{eq:almostconstanteigenfunction}) can be replaced by the much weaker condition
\[
\int_{\mathbb{L}_{z}} \left|\varphi_i(x') - \tilde{\varphi}_i(\xi(x'))\right| d\mu_{z}(x') \leq \varepsilon,
\]
for all level sets~$\mathbb{L}_z$ of $\xi$.
\end{remark}

From here on, we address the following two central questions:
\begin{itemize}
\item[(Q1)] \emph{In which dynamical situations can we expect to find low-dimensional reaction coordinates?}
\item[(Q2)] \emph{How can we computationally exploit the properties of the dynamics to obtain reaction coordinates?}
\end{itemize}
Let us start with the first question. We will address the second question in Section~\ref{ssec:Whitney} and Section~\ref{sec:NumApprox}. Experience shows~\cite{ERVE02,RVEME05,towards_tpt2006,SchuetteSarich2013} that transitions between metastable states tend to happen along so-called \emph{reaction pathways}, which is the low-dimensional dynamical backbone in the high-dimensional state space, connecting the metastable states via saddle points of the potential $V$~\cite{WF}.

From now on, we observe the system at an intermediate time scale $t_\text{slow} \gg t \gg t_\text{fast}$ (where~$t_\text{slow}$ and $t_\text{fast}$ are the implied time scales $t_d,~t_{d+1}$ from Section \ref{ssec:implied_ts}) and thus assume that the process $\mathbf{X}_t$ has already left the transition region (if it started there), equilibrated to a quasi-stationary distribution inside some metastable wells, but has not had enough time to equilibrate \emph{globally}.
At this time scale, starting in some~$x\in\X$, the transition density~$p^t(x,\cdot)$ is observed to approximately depend only on progress along these reaction paths; see Figure \ref{fig:reactionmanifold} for an illustration. 
This means that the density $p^t(x,\cdot)$ on the fiber perpendicular to the transition pathway is approximately the same as $p^t(x^*,\cdot)$ for some $x^*$ \emph{on} the transition pathway. As this pathway is low-dimensional, this means that the image $\overline{\mathcal{Q}}(\X)$ of the map 
$$
\overline{\mathcal{Q}}(x):= p^t(x,\cdot)
$$
is almost a low-dimensional manifold in $L^1(\mathbb{X})$.

The existence of this low-dimensional structure in the space of probability densities is exactly the assumption we need to ensure that the dominant eigenfunctions are low-dimensionally parametrizable, and thus that a low-dimensional reaction coordinate $\xi$ exists. This assumption is made precise in Definition \ref{def:reducibility}.
To summarize, we will see that~$\xi$ is a good reaction coordinate if~$p^t(x,\cdot) \approx p^t(y,\cdot)$ for~$\xi(x) = \xi(y)$. 

\begin{figure}[htb]
    \centering
    \includegraphics[width=.9\textwidth]{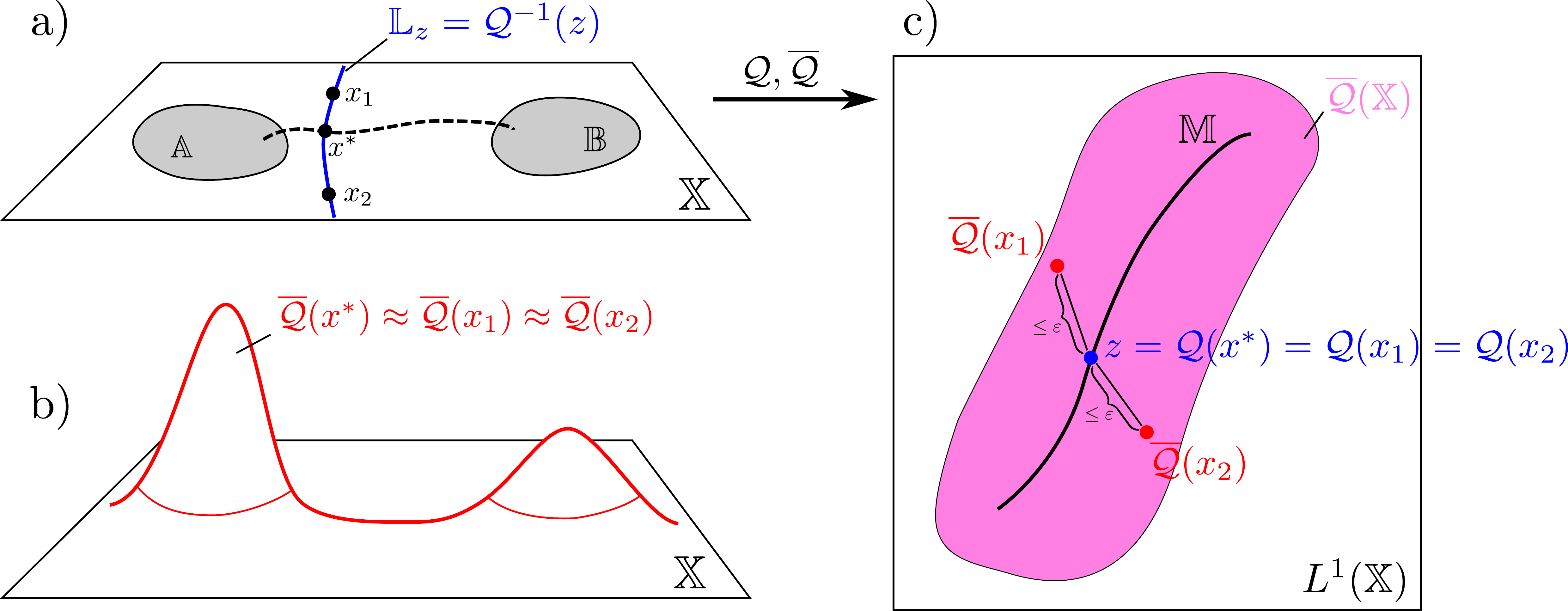}
    \caption{a) and b) The transition densities $\overline{\mathcal{Q}}(x_1)$ and $\overline{\mathcal{Q}}(x_2)$ are ``similar'' to $\overline{\mathcal{Q}}(x^*)$ for some $x^*$ on the transition path (dashed line) that connects the metastable sets $\mathbb{A}$ and~$\mathbb{B}$. c) The mapping~$\mathcal{Q}$ can be thought of as mapping all points that are ``similar'' under~$\overline{\mathcal{Q}}$ to the same point in~$L^1(\mathbb{X})$. The image of~$\mathcal{Q}$ thus forms a $r$-dimensional manifold in~$L^1(\mathbb{X})$.}
    \label{fig:reactionmanifold}
\end{figure}

\begin{definition}\label{def:reducibility}[$(\varepsilon,r)$-reducibility and transition manifold]\label{def:reducibleprocess}
We call the process $\mathbf{X}_t$ \emph{$(\varepsilon,r)$-reducible}, if there exists a smooth closed $r$-dimensional manifold~$ \mathbb{M} \subset L^2_{1/\mu} \subset L^1(\X)$ such that for~$t_\text{fast}\ll t \ll t_\text{slow}$ and all $x\in\X$
\begin{equation} \label{eq:reacMF}
   \min_{f\in\mathbb{M}} \|f - p^t(x,\cdot)\|_{L^2_{1/\mu}} \leq \varepsilon
\end{equation}
holds. We call~$\mathbb{M}$ the \emph{transition manifold} and the map~$\mathcal{Q} \colon \X \to \mathbb{M}$,
\begin{equation}
    \mathcal{Q}(x) := \mathrm{arg}\min_{f\in\mathbb{M}}\|p^t(x,\cdot) - f\|_{L^2_{1/\mu}}
\label{eq:firstDefQ}
\end{equation}
the \emph{mapping onto the transition manifold}. We can set~$\mathbb{M} = \mathrm{cl}(\mathcal{Q}(\X))$, where~$ \mathrm{cl}(\mathbb{Y})$ denotes the closure of the set~$\mathbb{Y}$.\footnote{If it is necessary to break ties in~\eqref{eq:firstDefQ}, we can do so by taking any of the minimizers. The mapping~$x\mapsto p^t(x,\cdot)$ can be shown to be smooth~\cite[Theorem C.1]{BiKoJu15}, hence~$\mathcal{Q}(\X)$ is a smooth manifold satisfying~\eqref{eq:reacMF}.}
\end{definition}

\begin{remark}
While it is natural to motivate $(\varepsilon,r)$-reducibility by the existence of reaction pathways in phase space, it is not strictly necessary. There exist stochastic systems without low-dimensional reaction pathways whose densities still quickly converge to a transition manifold in $L^1$. Future work includes the  identification of necessary and sufficient conditions for the existence of transition manifolds (see the first point in the conclusions).
We also further elaborate on the connection between reaction pathways and transition manifolds in Appendix~\ref{app:RCexist}. 
\end{remark}

\begin{remark}
We recall from Section~\ref{sec:Transfer operators and their properties} that the Perron--Frobenius operator~$\mathcal{P}^t$ is also naturally defined on the space~$L^2_{1/\mu}$~\cite{SchCa92}. Further, with the Dirac distribution centered in~$x\in\X$, denoted by~$\delta_x$, we formally have~$p^t(x,\cdot) = \mathcal{P}^t\delta_x$. Hence, the choice of norm in Definition~\ref{def:reducibleprocess} is natural.
It should also be noted that since~$\mu$ is a probability measure, the H\"older inequality yields~$\|f\|_{L^1_{\mu}} \le \|f\|_{L^2_{\mu}}$. Using this we have
\[
\|f\|_{L^1} = \|f/\varrho\|_{L^1_{\mu}} \le \|f/\varrho\|_{L^2_{\mu}} = \|f\|_{L^2_{1/\mu}}\,,
\]
which shows that if~$p^t(x,\cdot)$ and~$p^t(y,\cdot)$ are close in the~$L^2_{1/\mu}$ norm, they are also close in the~$L^1$ norm. We require the closeness of the respective~$p^t(x,\cdot)$ in the~$L^2_{1/\mu}$ norm for our theoretical considerations below, but otherwise we will think of them as functions in~$L^1$.
\end{remark}
Note that we only need to evolve the system at hand for a moderate time~$t\ll t_\text{slow}$, which has to be merely sufficiently large to damp out the fast fluctuations in the metastable states. This will be an important point later, allowing for numerical tractability.

Next, we show that~$(\varepsilon,r)$-reducibility implies that dominant eigenfunctions are almost constant on the level sets of~$\mathcal{Q}$.

\begin{lemma}\label{lem:eigenfunctionsalmostconstant}
If $\mathbf{X}_t$ is $(\varepsilon,r)$-reducible, then for an eigenfunction~$\varphi_i$ of~$\mathcal{T}^t$ with~$\|\varphi_i\|_{L^2_{\mu}}=1$ and points~$x,y\in\X$ with $\mathcal{Q}(x) = \mathcal{Q}(y)$ we have
\begin{equation*}
    |\varphi_i(x) - \varphi_i(y)| \leq \frac{2\varepsilon}{|\lambda_i|}\,.
\end{equation*}
\end{lemma}
\begin{proof}
First note that for the transition densities $p^t(x,\cdot),~p^t(y,\cdot)$ it holds that
\begin{equation} \label{eq:transdensitydifference}
    \begin{split}
    \|p^t(x,\cdot) - p^t(y,\cdot)\|_{L^2_{1/\mu}} &\leq \|p^t(x,\cdot) - \mathcal{Q}(x)\|_{L^2_{1/\mu}} + \|\mathcal{Q}(x) - p^t(y,\cdot)\|_{L^2_{1/\mu}}\\
                                    & = \|p^t(x,\cdot) - \mathcal{Q}(x)\|_{L^2_{1/\mu}} + \|\mathcal{Q}(y) - p^t(y,\cdot)\|_{L^2_{1/\mu}} \leq 2\varepsilon~.
    \end{split}
\end{equation}
With this we can show the assertion:
\begin{align*}
    \lambda_i\varphi_i(x) &= \mathcal{T}^t \varphi_i(x)
                           = \mathcal{K}^t \varphi_i(x)
                           = \int_\X\varphi_i(x')p^t(x,x') \, dx'.
    \intertext{Applying \eqref{eq:transdensitydifference}, for some function $e\in L^2_{1/\mu}(\X)$ with $\|e\|_{L^2_{1/\mu}}\leq 2\varepsilon$, we get}
\lambda_i\varphi_i(x)    &=\int_\X\varphi_i(x')\big(p^t(y,x')+e(x')\big)~dx'\\
    &=\int_\X\varphi_i(x')p^t(y,x') dx' + \int_\X\varphi_i(x')\frac{e(x')}{\varrho(x')}~d\mu(x')\\
    &=\lambda_i\varphi_i(y) + \int_\X\varphi_i(x')\frac{e(x')}{\varrho(x')} \, d\mu(x'),
\end{align*}
where in the last equation, we again used that due to reversibility $ \mathcal{K}^t = \mathcal{T}^t $ and that $\varphi_i$ is an eigenfunction. Thus for the difference, we have
\begin{align*}
    |\varphi(x) - \varphi(y)| &=\frac{1}{|\lambda_i|}\Big|\int_\X\varphi_i(x')\frac{e(x')}{\varrho(x')}~d\mu(x')\Big| \\
                              &\leq \frac{1}{|\lambda_i|}\underbrace{\|\varphi_i\|_{L^2_{\mu}}}_{=1} \underbrace{\|e/\varrho\|_{L^2_{\mu}}}_{ = \|e\|_{L^2_{1/\mu}}}
                               \leq \frac{2\varepsilon}{|\lambda_i|}\,. \qedhere
\end{align*}
\end{proof}

Assuming that the eigenfunctions are normalized (which we do from now on), i.e., $\|\varphi_i\|_{L^2_{\mu}}=1$, and that~$\varepsilon$ is sufficiently small, Lemma~\ref{lem:eigenfunctionsalmostconstant} implies that the dominant eigenfunctions (i.e.,~$|\lambda_i|\approx 1$) are almost constant on the level sets of~$\mathcal{Q}$. This can now be used to show that the $\varphi_i$ are not fully dependent on $\X$, but only on the level sets of~$\mathcal{Q}$ (up to a small error), in a sense similar to Lemma~\ref{lem:parametrizableeigenfunctions1}.

\begin{corollary}\label{cor:parametrizableeigenfunctions2}
Let $\mathbf{X}_t$ be $(\varepsilon,r)$-reducible. Then there exists a function $\tilde{\varphi}_i \colon \mathbb{M} \rightarrow \R $ such that
\begin{equation*}
    \left|\varphi_i(x) - \tilde{\varphi}_i\big(\mathcal{Q}(x)\big)\right| \leq \frac{\varepsilon}{|\lambda_i|}\,.
\end{equation*}
\end{corollary}
\begin{proof}
Fix~$x\in\X$, and let~$z = \mathcal{Q}(x)$.
Define the function  $\tilde{\varphi}_i$ by
\begin{equation*}
    \tilde{\varphi}_i(\mathcal{Q}(x)) := \frac12 \left( \inf_{\mathcal{Q}(y)=z} \varphi_i(y) + \sup_{\mathcal{Q}(y)=z} \varphi_i(y) \right).
\end{equation*}
Since by Lemma~\ref{lem:eigenfunctionsalmostconstant} it holds that~$|\varphi_i(x) - \varphi_i(y)| \le \tfrac{2\varepsilon}{|\lambda_i|}$ if~$\mathcal{Q}(x) = \mathcal{Q}(y)$, we have that
\[
\left| \sup_{\mathcal{Q}(y)=z} \varphi_i(y) - \inf_{\mathcal{Q}(y)=z} \varphi_i(y) \right| \le \frac{2\varepsilon}{|\lambda_i|}\,,
\]
thus our choice of~$\tilde{\varphi}_i$ gives
\[
\left| \varphi_i(x) - \tilde{\varphi}_i(\mathcal{Q}(x)) \right| \le \frac{\varepsilon}{|\lambda_i|}\,.
\]
\end{proof}

\subsection{Embedding the transition manifold} \label{ssec:Whitney}

In light of Corollary~\ref{cor:parametrizableeigenfunctions2}, one could say that $\mathcal{Q}$ is an ``$\mathbb{M}$-valued reaction coordinate''.
However, as we have no access to~$\mathbb{M}$ so far, and a $\mathbb{R}^k$-valued reaction coordinate is more intuitive, we aim to obtain a more useful representation of the transition manifold through~\emph{embedding} it into a finite, possibly low-dimensional Euclidean space.

We will see that we are very free in the choice of the embedding mapping, even though the manifold $\mathbb{M}$ is not known explicitly (we only assumed that it exists). 
To achieve this, we will use an infinite-dimensional variant of the \emph{weak Whitney embedding theorem}~\cite{SaYoCa91,Whi36}, which, roughly speaking, states that ``almost every bounded linear map from~$L^1(\mathbb{X})$ to $\mathbb{R}^{2r+1}$ will be one-to-one on $\mathbb{M}$ and its image''. We first specify what we mean by ``almost every'' in the context of bounded linear maps, following the notions of Sauer et al.~\cite{SaYoCa91}.

\begin{definition}[Prevalence]
A Borel subset $\mathbb{S}$ of a normed linear space $\mathbb{V}$ is called \emph{prevalent} if there is a finite-dimensional subspace $\mathbb{E}$ of $\mathbb{V}$ such that for each $v\in\mathbb{V}$, $v+e$ belongs to $\mathbb{S}$ for (Lebesgue) almost every $e$ in $\mathbb{E}$.
\end{definition}

As the infinite-dimensional embedding theorem from Hunt et al.~\cite{HuKa99} is applicable not only to smooth manifolds, but to arbitrary subsets $\mathbb{A}\subset\mathbb{V}$ of fractal dimension, it uses the concepts of \emph{box covering dimension} $\dim_B(\mathbb{A})$ and \emph{thickness exponent} $\tau(\mathbb{A})$ from fractal geometry. Intuitively, $\dim_B(\mathbb{A})$ describes the exponent of the growth rate in the number of boxes of decreasing side length that are needed to cover $\mathbb{A}$, and~$\tau(\mathbb{A})$ describes how well $\mathbb{A}$ can be approximated using only finite-dimensional linear subspaces of~$\mathbb{V}$. As these concepts coincide with the traditional measure of dimensionality in our setting, we will not go into detail here and point to~\cite{HuKa99} for a precise definition.

The general infinite-dimensional embedding theorem reads:
\begin{theorem}[{\cite[Theorem~3.9]{HuKa99}}] \label{thm:embeddingThm}
Let $\mathbb{V}$ be a Banach space and $\mathbb{A}\subset\mathbb{V}$ be a compact set with box-counting dimension $d$ and thickness exponent $\tau$. Let $k>2d$ be an integer, and let $\alpha$ be a real number with
$$
    0<\alpha <\frac{k-2d}{k(1+\tau)}\,.
$$
Then for almost every (in the sense of prevalence) bounded linear function $\mathcal{E}:\mathbb{V}\rightarrow\mathbb{R}^k$ there exists $C>0$ such that for all $x,y\in\mathbb{A}$,
\begin{equation}\label{eq:EmbeddingHoelder}
C\|\mathcal{E}(x)-\mathcal{E}(y)\|_2^\alpha \geq \|x-y\|_2\,,
\end{equation}
where~$\|\cdot\|_2$ denotes the Euclidean $2$-norm.
\end{theorem}

Note that (\ref{eq:EmbeddingHoelder}) implies H\"older continuity of $\mathcal{E}^{-1}$ on $\mathcal{E}(\mathbb{A})$ and in particular that $\mathcal{E}$ is one-to-one on $\mathbb{A}$ and its image. Using that the box counting dimension of a smooth $r$-dimensional manifold $\mathbb{K}$ is simply $r$ and that the thickness exponent is bounded from above by the box-counting dimension, thus $0\leq \tau(\mathbb{K}) \leq r$, see \cite{HuKa99}, we get the following infinite-dimensional embedding theorem for smooth manifolds.

\begin{corollary}\label{thm:whitney}
Let $\mathbb{V}$ be a Banach space, let $\mathbb{K}\subset \mathbb{V}$ be a smooth manifold of dimension~$r$ and let $k>2r$. 
Then almost every (in the sense of prevalence) bounded linear function~$\mathcal{E}:\mathbb{V}\rightarrow\mathbb{R}^k$ is one-to-one on $\mathbb{K}$ and its image in $\mathbb{R}^k$.
\end{corollary}

Thus, since the transition manifold~$\mathbb{M}$ is assumed to be a smooth $r$-dimensional manifold in $L^1(\X)$, an arbitrarily chosen bounded linear map $\mathcal{E} \colon L^1(\X)\rightarrow \R^{2r+1}$ can be assumed to be one-to-one on $\mathbb{M}$ and its image. In particular, $\mathcal{E}(\mathbb{M})$ is again an $r$-dimensional manifold (although not necessarily smooth). With this insight, we can now construct a reaction coordinate in Euclidean space:

\begin{corollary}\label{cor:xihaterrorsmall}
Let $\mathbf{X}_t$ be $(\varepsilon,r)$-reducible and let $\mathcal{E} \colon L^1(\X) \rightarrow \R^{2r+1}$ be one-to-one on $\mathbb{M}$ and its image. Define $\xi \colon \R^n \rightarrow \R^{2r+1}$ by
\begin{equation}\label{eq:xidefinition}
    \xi(x) := \mathcal{E}\big(\mathcal{Q}(x)\big)\,.
\end{equation}
Then there exists a function $\hat{\varphi}_i \colon \R^{2r+1} \rightarrow \R $ so that
\begin{equation}\label{eq:phidifference}
    |\varphi_i(x) - \hat{\varphi}_i(\xi(x))| \leq \frac{\varepsilon}{|\lambda_i|}\,.
\end{equation}
\end{corollary}

\begin{proof}
As $\mathcal{E}$ is one-to-one on $\mathbb{M}$ and its image, it is invertible on $\mathcal{E}(\mathbb{M})$. With $\tilde{\varphi}_i$ chosen as in the proof of Corollary~\ref{cor:parametrizableeigenfunctions2}, define~$\hat{\varphi}_i \colon \mathcal{E}(\mathbb{M}) \to \R$ by
\begin{equation}\label{eq:phihatdefinition}
    \hat{\varphi_i}(\hat{z}) := \tilde{\varphi}_i\big(\mathcal{E}^{-1}(\hat{z})\big)\,.
\end{equation}
Then
\begin{equation*}
    |\varphi_i(x) - \hat{\varphi}_i(\xi(x))| 
        = |\varphi_i(x) - \tilde{\varphi}_i(\mathcal{Q}(x))| 
        \overset{\text{Cor.~\ref{cor:parametrizableeigenfunctions2}}}{\leq} \frac{\varepsilon}{|\lambda_i|} \,. \qedhere
\end{equation*}
\end{proof}
Since~$\hat{\mathbb{M}} := \mathcal{E}(\mathbb{M})$ is an~$r$-dimensional manifold,~$\xi$ is effectively an~$r$-dimensional reaction coordinate.
Thus, if the right-hand side of \eqref{eq:phidifference} is small, the $\varphi_i$ are ``almost parametrizable'' by the $r$-dimensional reaction coordinate $\xi$. Using Lemma \ref{lem:parametrizableeigenfunctions1}, we immediately see that this results in a small projection error $\|P_\xi^\perp\varphi_i\|$, and due to Corollary~\ref{cor:characRC} in a good transfer operator approximation; hence~$\xi$ is a good reaction coordinate.

The reaction coordinate~$\xi$ remains an ``ideal'' case, because we have no access to the map~$\mathcal{Q}$ and hence to~$\mathbb{M}$, only to~$\overline{\mathcal{Q}}(x) = p^t(x,\cdot) \approx \mathcal{Q}(x)$.
We summarize the construction of the ideal reaction coordinate~$\xi$ in Figure~\ref{fig:RC_ideal}.

\begin{figure}[htb]
    \centering
    \includegraphics[width=\textwidth]{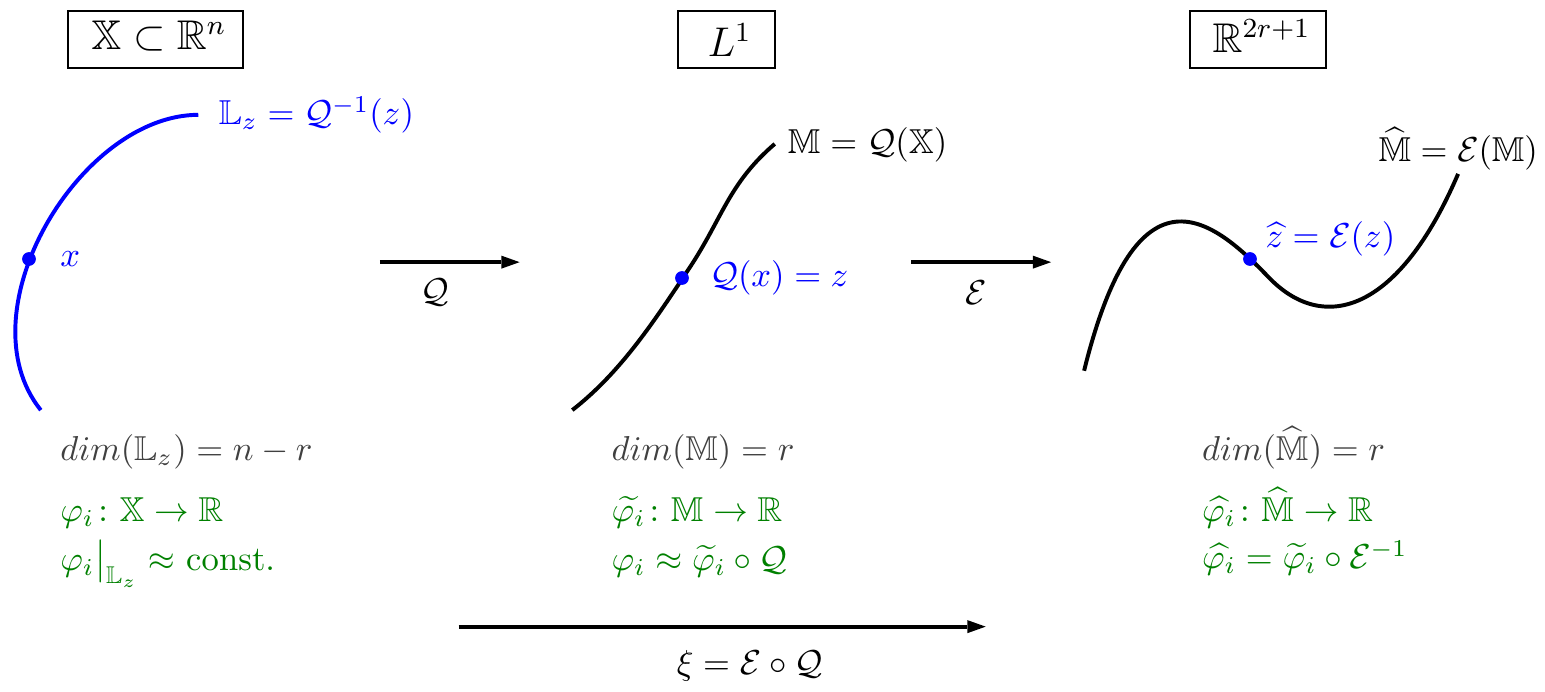}
    \caption{Summary of the construction of the ideal reaction coordinate~$\xi$.}
    \label{fig:RC_ideal}
\end{figure}

\begin{remark}
The recent work of Dellnitz et al.~\cite{DHZ16} uses similar embedding techniques to identify finite-dimensional objects in the state space of infinite-dimensional dynamical systems. They utilize the infinite-dimensional delay-embedding theorem of Robinson~\cite{Ro05}, a generalization of the well-known \emph{Takens embedding theorem}~\cite{Ta81}, to compute finite-dimensional attractors of delay differential equations by established subdivision techniques~\cite{DeHo97}.
\end{remark}

\subsection{Numerical approximation of the reaction coordinate}\label{sec:NumApprox}

\paragraph{Approximate embedding of the transition manifold.}

We now elaborate how to construct a good reaction coordinate~$\overline{\xi}$ numerically. To use the central definition \eqref{eq:xidefinition} in practice, two points have to be addressed:
\begin{enumerate}
    \item How to choose the embedding $\mathcal{E}$?
    \item How to deal with the fact that we do not know~$\mathcal{Q}$?
\end{enumerate}
For the choice of $\mathcal{E}$, we restrict ourselves to linear maps of the form
\begin{equation} \label{eq:gammadefinition}
    \mathcal{E}(f) :=
    \begin{pmatrix}
        \innerprod{f}{\eta_1} \\
        \vdots \\
        \innerprod{f}{\eta_{2r+1}}
    \end{pmatrix}\,,
\end{equation}
with arbitrarily chosen linearly independent functions~$ \eta_i\in L^{\infty}(\X) $, where~$\langle f,\eta_i\rangle = \int f\eta_i$. In practice, we will choose the~$\eta_i:\X\to\mathbb{R}$ as linear functions themselves, i.e.~$\eta_i(x) = a_i^{\intercal}x$ for some, usually randomly drawn,~$a_i\in\mathbb{R}^n$. Note that then~$\eta_i\notin L^{\infty}$, but this is not a problem because we will embed the functions~$f = p^t(x,\cdot)$, and~$p^t(x,y)$ can be shown to decay exponentially as~$\|y\|_2\to\infty$, cf.~\cite[Theorem C.1]{BiKoJu15}. Thus,~$\innerprod{f}{\eta_i}$ will exist.
For linearly independent~$\eta_i$, these maps are still generic in the sense of the Whitney embedding theorem, and thus still embed the transition manifold~$\mathbb{M}$.

A natural choice for the approximation of the unknown map~$\mathcal{Q}$ is the mapping to the transition probability density,
\begin{equation}
    \overline{\mathcal{Q}}: x \mapsto p^t(x,\cdot)\,,
\end{equation} 
as~$\|\mathcal{Q}(x) - p^t(x,\cdot)\|_{L^2_{1/\mu}} \leq \varepsilon$.
With this, we consider
\begin{equation} \label{eq:xidefinition2}
    \mathcal{E}\big(\overline{\mathcal{Q}}(x)\big) = \mathcal{E}\big(p^t(x,\cdot)\big) = 
    \begin{pmatrix}
        \innerprod{p^t(x, \cdot)}{\eta_1} \\
        \vdots \\
        \innerprod{p^t(x, \cdot)}{\eta_{2r+1}}
    \end{pmatrix} \stackrel{\eqref{eq:Koopman operator}}{=}
    \begin{pmatrix}
        \mathcal{K}^t \eta_1(x) \\
        \vdots\\
        \mathcal{K}^t \eta_{2r+1}(x)
    \end{pmatrix}.
\end{equation}
The values on the right-hand side can in turn be approximated by a Monte Carlo quadrature, using only short-time trajectories of the original dynamics:
\begin{equation} \label{eq:Monte Carlo approximation}
    \mathcal{K}^t \eta_i(x) = \mathsf{E}\big[\eta_i(\mathbf{X}_t) \mid \mathbf{X}_0 = x \big] \approx \frac{1}{M} \sum_{m=1}^M\eta_i\big(\mathbf{\Phi}_t^{(m)}(x)\big),
\end{equation}
where the $\mathbf{\Phi}_t^{(m)}(x)$ are independent realizations of $\mathbf{X}_t$ with starting point $\mathbf{X}_0=x$, in practice realized by a stochastic integrator (e.g. Euler--Maruyama).

\paragraph{The computationally infeasible reaction coordinate~$\boldsymbol{\xi}$.}

Note that~$\mathcal{E}\circ\overline{\mathcal{Q}}$ is not yet an~$r$-dimensional reaction coordinate, since~$\overline{\mathcal{Q}}(\X)$ is only approximately an $r$-dimensional manifold; more precisely, it lies in the~$\varepsilon$-neighborhood of an~$r$-dimensional submanifold~$\mathbb{M}$ of~$L^1$. Hence, also~$\mathcal{E}(\overline{\mathcal{Q}}(\X))$ is only approximately an~$r$-dimensional manifold; see the magenta regions in Figure~\ref{fig:RC_dontknowhow}.

\begin{figure}[htb]
    \centering
    \includegraphics[width=\textwidth]{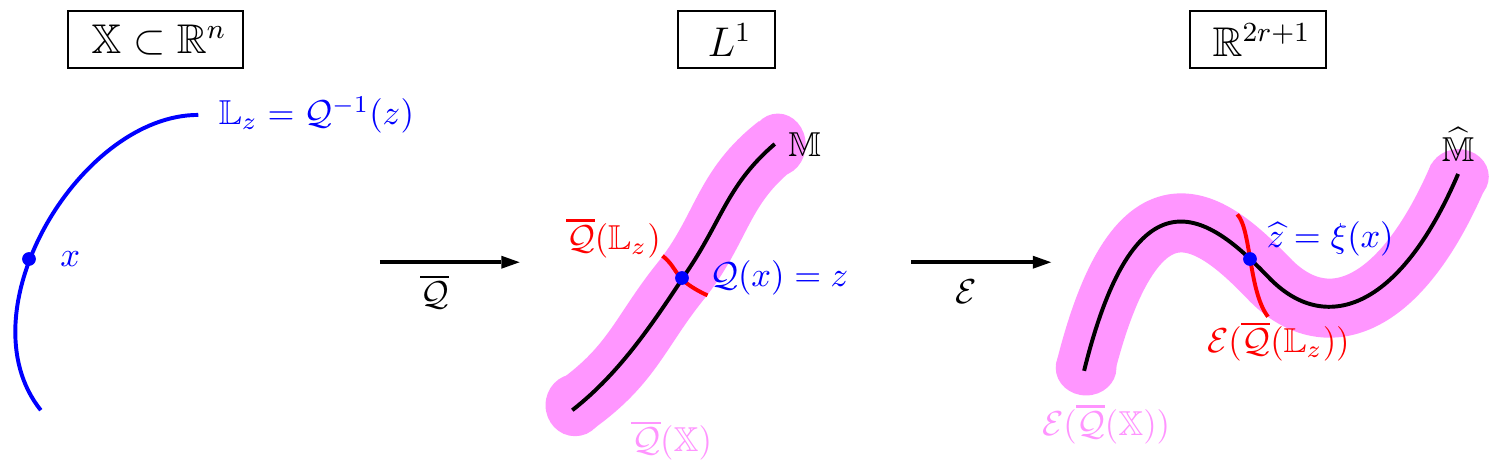}
    \caption{How to make a good $r$-dimensional reaction coordinate out of~$\mathcal{E}\circ\overline{\mathcal{Q}}$? Given the analysis from the previous section, we would like to make the level sets $\mathbb{L}_z$ of $\mathcal{Q}$ also the level sets of $\overline{\xi}$ (red line segment). Unfortunately, we have no access to these.}
    \label{fig:RC_dontknowhow}
\end{figure}

The question now is how we can reduce~$\mathcal{E}\circ\overline{\mathcal{Q}}$ to an $r$-dimensional \emph{good} reaction coordinate. Since we know from above that~$\xi = \mathcal{E}\circ \mathcal{Q}$ is a good reaction coordinate, let us see what would be needed to construct it.

The property of~$\xi$ that we want is that it is constant along level sets~$\mathbb{L}_z$ of~$\mathcal{Q}$, i.e., $\xi\vert_{\mathbb{L}_z} = \mathrm{const}$ (because this implies that it is a good reaction coordinate, cf.~Corollary~\ref{cor:xihaterrorsmall}).
Hence, if we could identify~$\hat{\mathbb{M}}$ as an $r$-dimensional manifold in~$\R^{2r+1}$, we would project~$\mathcal{E}(\overline{\mathcal{Q}}(x))$ along~$\mathcal{E}(\overline{\mathcal{Q}}(\mathbb{L}_z))$ onto~$\hat{\mathbb{M}}$ --- assuming that~$\hat{\mathbb{M}}$ and~$\mathcal{E}(\overline{\mathcal{Q}}(\mathbb{L}_z))$ intersect in~$\R^{2r+1}$ --- to obtain~$\xi(x)$ as the resulting point (see Figure~\ref{fig:RC_dontknowhow}, where we would project along the red line on the right).
Unfortunately, we have no access to~$\mathcal{Q}$ (not to mention that~$\hat{\mathbb{M}}$ and~$\mathcal{E}(\overline{\mathcal{Q}}(\mathbb{L}_z))$ need not intersect in~$\R^{2r+1}$) and hence to its level sets $\mathbb{L}_z$. Thus, this strategy seems infeasible.

\paragraph{A computationally feasible reaction coordinate.}

What helps us at this point is that there is a certain amount of arbitrariness in the definition of~$\mathcal{Q}$. Recalling Definition~\ref{def:reducibleprocess}, what we are given is~$\overline{\mathcal{Q}}$, and we construct~$\mathcal{Q}(x)$ as a projection of~$\overline{\mathcal{Q}}(x)$ onto the~$r$-dimensional manifold~$\mathbb{M}$ by the closest-point projection~$\mathcal{Q}'$; i.e.,~$\mathcal{Q} = \mathcal{Q}'\circ\overline{\mathcal{Q}}$. This choice of~$\mathcal{Q}'$ is convenient, because we can show
\begin{equation}
    \|\overline{\mathcal{Q}}(x) - \overline{\mathcal{Q}}(y)\|_{L^2_{1/\mu}}\le 2\varepsilon\quad \text{for every }\mathcal{Q}(x) = \mathcal{Q}(y)~\text{(i.e., on level sets of $\mathcal{Q}'$),}
    \label{eq:QbarClose}
\end{equation}
which is used in Lemma~\ref{lem:eigenfunctionsalmostconstant}. 
Other choices of~$\mathcal{Q}'$ could, however, yield a similarly practicable~$\mathcal{O}(\varepsilon)$-bound in~\eqref{eq:QbarClose}. Our strategy will be to choose a specific~$r$-dimensional reaction coordinate~$\overline{\xi}$ and to show that in general it can be expected to be a good reaction coordinate.

Let us recall that, by assumption, the set~$\overline{\mathcal{Q}}(\X)$ is contained in the~$\varepsilon$-neighborhood of an unknown smooth~$r$-dimensional manifold~$\mathbb{M}\subset L^1(\X)$. Thus, a generic smooth map~$\mathcal{E} \colon L^1(\X) \to \R^{2r+1}$ will embed~$\mathbb{M}$ into~$\R^{2r+1}$, forming a diffeomorphism from~$\mathbb{M}$ to~$\hat{\mathbb{M}}$. Thus,~$\mathcal{E}$ is going to map~$\overline{\mathcal{Q}}(\X)$ to an~$\mathcal{O}(\varepsilon)$-neighborhood of~$\hat{\mathbb{M}}$. This means, the~$r$-dimensional manifold structure of~$\hat{\mathbb{M}}$ should still be detectable and can be identified with standard manifold learning tools. We use the diffusion maps algorithm (see Section~\ref{sec:diffmaps} below), which gives us a map $\Psi : \R^{2r+1} \rightarrow \R^r$ (the diffusion map). Then we define $\overline\xi$ as
\begin{equation}
    \overline{\xi} := \Psi \circ \mathcal{E}\circ \overline{\mathcal{Q}}.
    \label{eq:xibar}
\end{equation}
This is depicted on the right-hand side of Figure~\ref{fig:RC_realized}, where the red dashed line shows the level set~$\hat{\mathbb{L}}_{\hat{z}} = \{ z\in \R^{2r+1}: \Psi(z) = \Psi(\hat z)\}$.

\begin{figure}[htb]
    \centering
    \includegraphics[width = 1.0\textwidth]{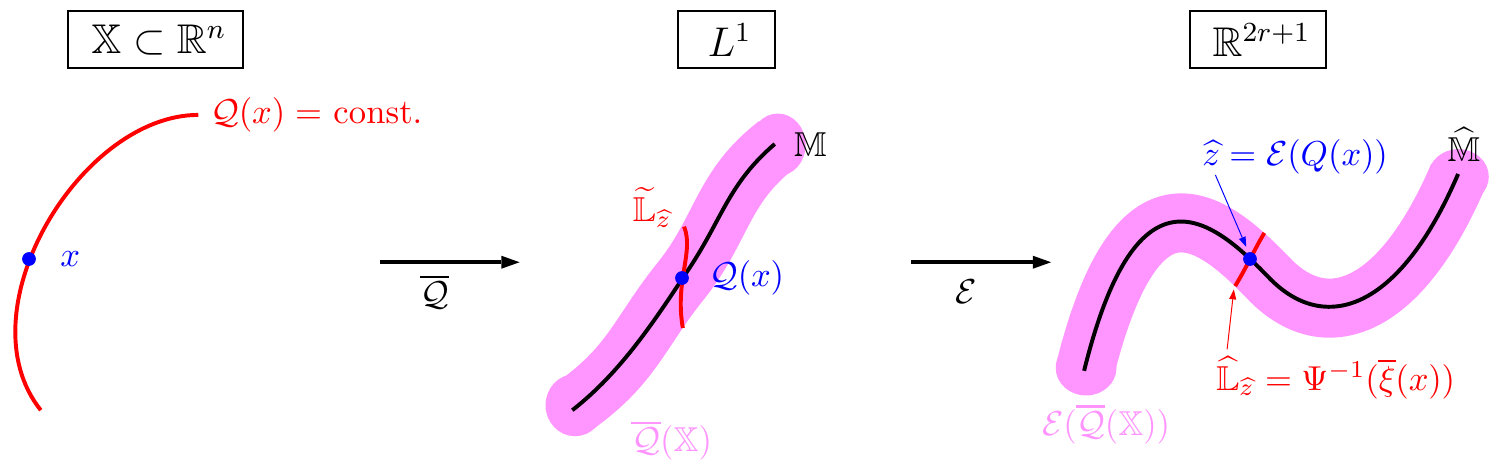}
    \caption{The realized reaction coordinate~$\overline{\xi}$.}
    \label{fig:RC_realized}
\end{figure}

Next, we consider the set~$\tilde{\mathbb{L}}_{\hat{z}} := \mathcal{E}^{-1}(\hat{\mathbb{L}}_{\hat{z}}) \cap \overline{\mathcal{Q}}(\X)$. It holds that $\tilde{\mathbb{L}}_{\hat{z}} = \left\{\overline{\mathcal{Q}}(x)\,\big\vert\,\overline{\xi}(x) = \Psi(\hat{z})\right\}$. Recall that~$\mathcal{E} \colon \mathbb{M} \to \hat{\mathbb{M}}$ is one-to-one, thus~$\tilde{\mathbb{L}}_{\hat{z}}$ intersects~$\mathbb{M}$ in exactly one point. We define this one point as~$\mathcal{Q}(x)$, and thus~$\mathcal{Q}'$ is the projection onto~$\mathbb{M}$ along~$\tilde{\mathbb{L}}_{\hat{z}}$. We see that~$\mathcal{Q}$ is well-defined and that~$\mathcal{Q}(x)=\mathcal{Q}(y) \Leftrightarrow \overline{\xi}(x) = \overline{\xi}(y)$.

At this point we assume that~$\mathcal{E}^{-1}$ is sufficiently well-behaved in a neighborhood of~$\hat{\mathbb{M}}$, it does not ``distort transversality'' of intersections, such that the diameter of~$\tilde{\mathbb{L}}_{\hat{z}}$ is~$\mathcal{O}(\varepsilon)$ with a moderate constant in~$\mathcal{O}(\cdot)$.
We will investigate a formal justification of this fact in a future work, here we assume it holds true, and we will see in the numerical experiments that the assumption is justified.
This assumption implies that~$\|\overline{\mathcal{Q}}(x) - \overline{\mathcal{Q}}(y)\|_{L^2_{1/\mu}} = \mathcal{O}(\varepsilon)$ for~$\mathcal{Q}(x) = \mathcal{Q}(y)$, i.e.\ for~$\overline{\xi}(x) = \overline{\xi}(y)$. Now, however, Lemma~\ref{lem:eigenfunctionsalmostconstant} implies that~$\varphi_i$ is almost constant (up to an error~$\mathcal{O}(\varepsilon)$) on level sets of~$\overline{\xi}$, which, in turn, by Lemma~\ref{lem:parametrizableeigenfunctions1} and Corollary~\ref{cor:characRC} shows that~$\overline{\xi}$ is a good reaction coordinate.

\subsection{Identification of $\hat{\mathbb{M}}$ through Manifold Learning}\label{sec:diffmaps}

In this section, we describe how to identify~$\hat{\mathbb{M}}$ numerically. The task is as follows: Given that we have computed~$\mathcal{E}(\overline{\mathcal{Q}}(x_i)) = \hat z_i \in \R^{2r+1}$ for a number of sample points~$\{x_i\}_{i=1}^{\ell} \subset \mathbb{X}$, we would like to identify the $r$-dimensional manifold $\hat{\mathbb{M}}$, noting the points~$\mathcal{E}(\overline{\mathcal{Q}}(x_i))$ are in a~$\mathcal{O}(\varepsilon)$-neighborhood of~$\hat{\mathbb{M}}$ (see Section \ref{sec:NumApprox}). Additionally, we would like an $r$-dimensional coordinate function~$\Psi:\R^{2r+1} \rightarrow \R^r$ that parameterizes~$\hat{\mathbb{M}}$ (so that the level sets of~$\Psi$ are transversal to~$\hat{\mathbb{M}}$).

This is a default setting for which manifold learning algorithms can be applied. Many standard methods exist; we name multidimensional scaling \cite{kruskal1964, kruskal1966}, Isomap \cite{tenenbaum2000}, and diffusion maps \cite{coifman2006} as a few of the most prominent examples. Because of its favorable properties, we choose the diffusion maps algorithm here and summarize it briefly for our setting in what follows. For details, the reader is referred to \cite{coifman2006, NLCK06, CKLMN08, SEKC09}.

Given sample points $\{\hat z_i\}_{i=1}^{\ell} \subset \R^{2r+1}$, diffusion maps proceeds by constructing a similarity matrix~$W\in\mathbb{R}^{\ell\times\ell}$ with
\[
W_{ij} = h\left(\frac{\|\hat z_i - \hat z_j\|_2^2}{\sigma}\right)\,,
\]
where $\|\cdot \|_2$ is the Euclidean norm in $\R^{2r+1}$, $\sigma > 0$ is a scale factor, and $h : \R \rightarrow \R_+$ is a kernel function which is most commonly chosen as $h(x) = \exp(-x) 1_{x\leq R}$ with a suitably chosen cutoff $R$ that sparsifies $W$ and ensures that only local distances enter the construction. With~$D$ being the diagonal matrix containing the row sums of~$W$, the similarity matrix is then normalized to give~$\tilde W = D^{-1}WD^{-1}$. Finally, the stochastic matrix~$P = \tilde D^{-1}\tilde W$ is constructed, where~$\tilde D$ is the diagonal matrix containing the row sums of~$\tilde W$. $P$ is similar to the symmetric matrix $\tilde D^{-1/2}\tilde W\tilde D^{-1/2}$, thus it has an orthonormal basis of eigenvectors $\{\psi_i\}_{i=0}^{\ell-1}$ with real eigenvalues $\gamma_i$. Since~$P$ is also stochastic,~$|\gamma_i| \leq 1$. The diffusion map is then given by
\begin{equation}
\Psi: \R^{2r+1} \rightarrow \R^r, \quad \Psi(\hat z) = \left(\gamma_1 \psi_1(\hat z),\ldots, \gamma_r \psi_r(\hat z)\right)^{\intercal}.
\label{eq:realXiBar}
\end{equation}
Using properties of the Laplacian eigenproblem on~$\hat{\mathbb{M}}$, one can show that $\Psi$ indeed parameterizes the $r$-dimensional manifold~$\hat{\mathbb{M}}$ for suitably chosen~$\sigma$ \cite{coifman2006}.

\begin{remark}\label{rem:Diffusion Maps Applicability}

The diffusion maps algorithm will only reliably identify $\hat{\mathbb{M}}$ based on the neighborhood relations between the embedded sample points $z_i$, if the points cover all parts of $\hat{\mathbb{M}}$ sufficiently well. In particular, as $p^t(x,\cdot)$ and thus $\big(\mathcal{E}\circ\overline{\mathcal{Q}}\big)(x)$ vary strongly with $x$ traversing the transition regions, a good coverage of those regions is required.

For the various low-dimensional academical examples Section~\ref{sec:Numerical Examples}, this is ensured by choosing the $x_i$ to be a dense grid of points in $\mathbb{X}$. For the high-dimensional example in Section \ref{sec:Lemon Slice}, the evaluation points are generated as a subsample from a long equilibrated trajectory, essentially sampling $\mu$. 
Both of these ad-hoc methods are likely to be unapplicable in realistic high-dimensional systems with very long equilibration times. However, as we mentioned in the introduction, there exist multiple statistical and dynamical approaches to this common problem of quickly sampling the relevant parts of phase space, including the transition regions. Each of these sampling methods can be easily integrated into our proposed algorithm as a pre-processing step.

Fundamentally though, the central idea of our method does not depend crucially on the applicability of diffusion maps. Rather, the latter can be considered an optional post-processing step. Using the $2r+1$-dimensional reaction coordinate 
$$
\overline{\overline{{\xi}}} := \mathcal{E}\circ \overline{\mathcal{Q}}~,
$$
i.e.\ \eqref{eq:xibar} without the manifold learning step, may in practice already represent a sufficient dimensionality reduction. 

In addition, situations may occur where the a priori generation of evaluation points is not possible or desired. One of the final goals and currently work in progress is the construction of an accelerated integration scheme that generates significant evaluation points and their reaction coordinate value ``on the fly''. This is related to the effective dynamics mentioned in fifth point of the conclusion. However, this also requires us to be able to evaluate the reaction coordinate at isolated points, independent of each other, and thus also necessitates the use of the above $\overline{\overline{{\xi}}}$ instead of $\overline{\xi}$.
\end{remark}

\section{Numerical Examples}
\label{sec:Numerical Examples}

Based on the results from the previous sections, we propose the following algorithm to compute reaction coordinates numerically:
\begin{enumerate}
\item Let $ x_i $, $ i = 1, \dots, \ell $, be the points for which we would like to evaluate $ \overline{\xi} $. Here, we assume the points satisfy the requirements addressed in Remark \ref{rem:Diffusion Maps Applicability}.
\item Choose linearly independent functions $ \eta_j \in L^\infty(\X) $, $ j = 1, \dots, 2r+1 $. The essential boundedness of the~$ \eta_j $ is not necessary, but~$ |\eta_j(x)| $ should not grow faster than a polynomial as~$ \|x\|_2\to\infty $.
\item In each point~$ x_i $, start~$ M $ simulations of length~$ t $ and estimate~$\mathcal{E}_j\big(\overline{\mathcal{Q}}(x_i)\big)$ using~\eqref{eq:xidefinition2} and~\eqref{eq:Monte Carlo approximation}, to obtain the point~$\hat{z}_i\in\mathbb{R}^{2r+1}$. We discuss the appropriate choice of $M$ and $t$ in Section \ref{sec:bananapot}.
\item Apply the diffusion maps technique from Section~\ref{sec:diffmaps} for the point cloud~$ \{\hat{z}_i\}_{i=1}^{\ell} $, and obtain~$\Psi:\mathbb{R}^{2r+1}\to\mathbb{R}^r$, a parametrization of the point in its~$r$ essential directions of variation.
\item By~\eqref{eq:realXiBar}, we define the reaction coordinate as
\begin{equation}
\overline{\xi}:\,x_i \mapsto \Psi(\hat{z}_i)\,.
\end{equation}
\end{enumerate}

The numerical effort of this algorithm depends strongly on the third step. Given $\ell$ evaluation points, and a choice of $M$ trajectories per point, the cost is mainly given by $M\cdot \ell\cdot c(t)$, where $c(t)$ is the effort of a single numerical realization of the dynamics up to time $t$. The high-dimensional phase space only enters the algorithm as the domain of the observables $\eta_j$. The cost of evaluating those typically very simple functions\footnote{In our examples, we used linear functions with great success.} at the $M\cdot \ell$ end points of the trajectory is negligible. The cost of the method is thus essentially independent of $n$.

In order to demonstrate the efficacy of our method, we compute the reaction coordinates for three representative problems, namely a simple curved double-well potential, a multi-well potential defined on a circle, both in low and high dimensions, and two slightly different quadruple-well potentials stressing the difference between a one- and a two-dimensional reaction coordinate.

\subsection{Curved double-well potential}
\label{sec:bananapot}

As a first verification, we consider a system with an analytically known reaction coordinate that is then used for comparison. Consider the two-dimensional drift-diffusion process \eqref{eq:overdampedLangevin} with potential
\begin{equation*}
    V(x_1,x_2) = (x_1^2-1)^2 + 2(x_1^2+x_2-1)^2
\end{equation*}
and inverse temperature $\beta = 0.5$. This potential already served as a motivational example for the nature of reaction coordinates in the introduction and is shown in Figure~\ref{fig:bananapot}. The system possesses two metastable sets around the minima $ (-1,0)^{\intercal} $ and $ (1,0)^{\intercal} $, which are connected by the transition path $ \{x\in\mathbb{R}^2 \mid x_2 = 1-x_1^2\} $. The implied time scales, defined in~\eqref{eq:implied time scales}, can be computed from the eigenvalues using a standard Ulam-type Galerkin discretization~\cite{KKS16,KNKWKSN17} of the transfer operator~$ \mathcal{T}^t $ and are shown in Figure~\ref{fig:bananatime scales}a\footnote{In realistic, high-dimensional systems, the computation of the dominant eigenvalues using grid-based methods is likely infeasible. In these situations, the implied time scales have to be estimated, for example using standard Markov State Model techniques \cite{A19-1}.}. We observe a significant gap between $t_1$ and $t_2$ and thus identify $t_1$ as the last slow and $t_2$ as the first fast time scale. Choosing the lag time $t = 2$ then satisfies $t_\text{slow} > t > t_\text{fast}$. A visual inspection of a typical trajectory of length $ t $ starting in one of the two metastable sets as shown in Figure~\ref{fig:bananatime scales}b confirms that the respective set is sampled, yet a transition to the other set is a rare event.

\begin{figure}[htbp]
    \begin{minipage}{0.49\textwidth}
            \centering
            \subfiguretitle{a)}
    \begin{tabular}{ c | c}
    $t_0$ & $\infty$ \\
    $t_1$ & 6.1823 \\
    $t_2$ & 0.9066 \\
    $t_3$ & 0.6098 \\
    $t_4$ & 0.3976 
    \end{tabular}
        \end{minipage}
        \begin{minipage}{0.49\textwidth}
            \centering
            \subfiguretitle{b)}
            \includegraphics[width=0.9\textwidth]{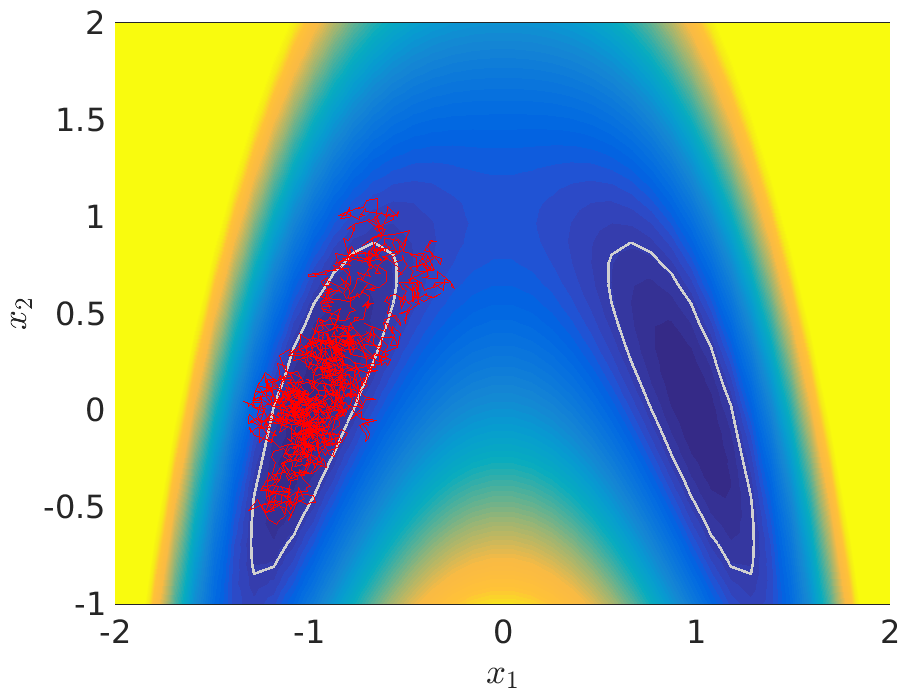}
        \end{minipage}
    \caption{a) Implied time scales of the double-well system. b) Trajectory of length~$t=1$.}
    \label{fig:bananatime scales}
\end{figure}

The low dimension of the system allows us to compute the reaction coordinate on a full regular grid over the phase space. We choose a $40 \times 30$ grid in the rectangular region $[-2,2] \times [-1,2]$ and denote the set of grid points by $\overline{\X}$. For this system, we expect a one-dimensional transition path and thus a one-dimensional reaction coordinate $\xi$. That is, $r=1$ and $2r+1=3$. Thus, we choose three linear observables in our embedding function~\eqref{eq:gammadefinition}, e.g.,
\begin{equation} \label{eq:bananaobservables}
    \begin{aligned}
    \eta_1(x_1,x_2) &=           -0.2630 \, x_1 - 0.3186 \, x_2, \\
    \eta_2(x_1,x_2) &=           -0.2246 \, x_1 + 0.0969 \, x_2, \\
    \eta_3(x_1,x_2) &= \phantom{-}0.1564 \, x_1 + 0.0783 \, x_2, \\
    \end{aligned}
\end{equation}
whose coefficients were drawn uniformly from $[-1,1]$. The expectation value in~\eqref{eq:xidefinition2} is approximated by a Monte Carlo quadrature using $M = 10^5$ sample trajectories for each grid point, cf.~\eqref{eq:Monte Carlo approximation}. The parameter $M$ was chosen such that the error in \eqref{eq:Monte Carlo approximation}, commonly defined as the variance of the Monte Carlo sum, is sufficiently low. The resulting embedding of the grid points $ x $ into $ \R^3 $ is shown in Figure~\ref{fig:bananaembedding}. The transition path seems to be already para\-metrized well by the individual components of $ \mathcal{E} \circ \overline{Q} $.

\begin{figure}[htbp]
    \centering
    \begin{minipage}{0.32\textwidth}
        \centering
        \subfiguretitle{$\big(\mathcal{E}_1\circ\overline{\mathcal{Q}}\big)(x)$}
        \includegraphics[width=\textwidth]{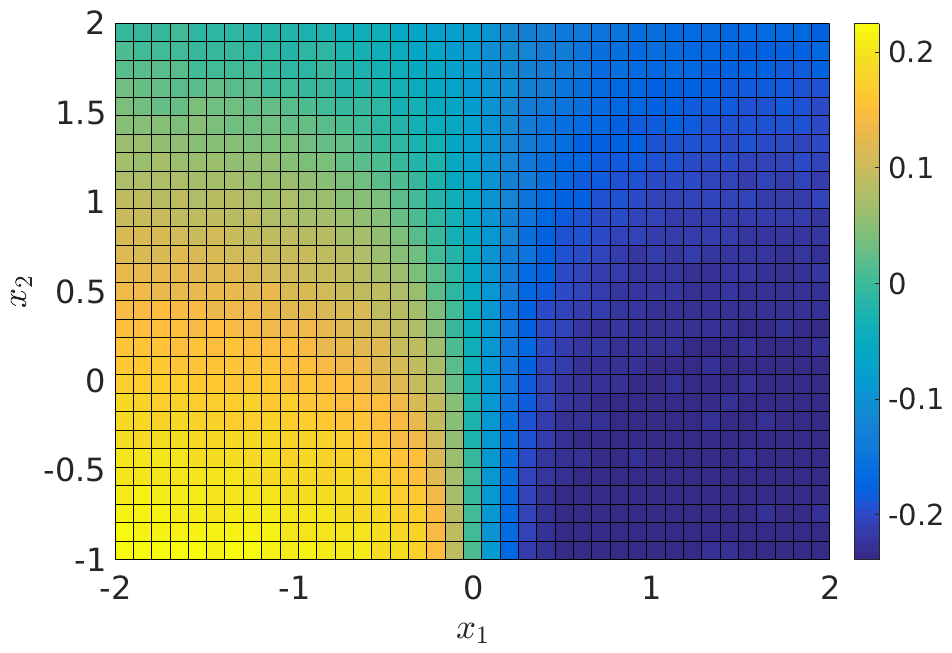}
    \end{minipage}
    \begin{minipage}{0.32\textwidth}
        \centering
        \subfiguretitle{$\big(\mathcal{E}_2\circ\overline{\mathcal{Q}}\big)(x)$}
        \includegraphics[width=\textwidth]{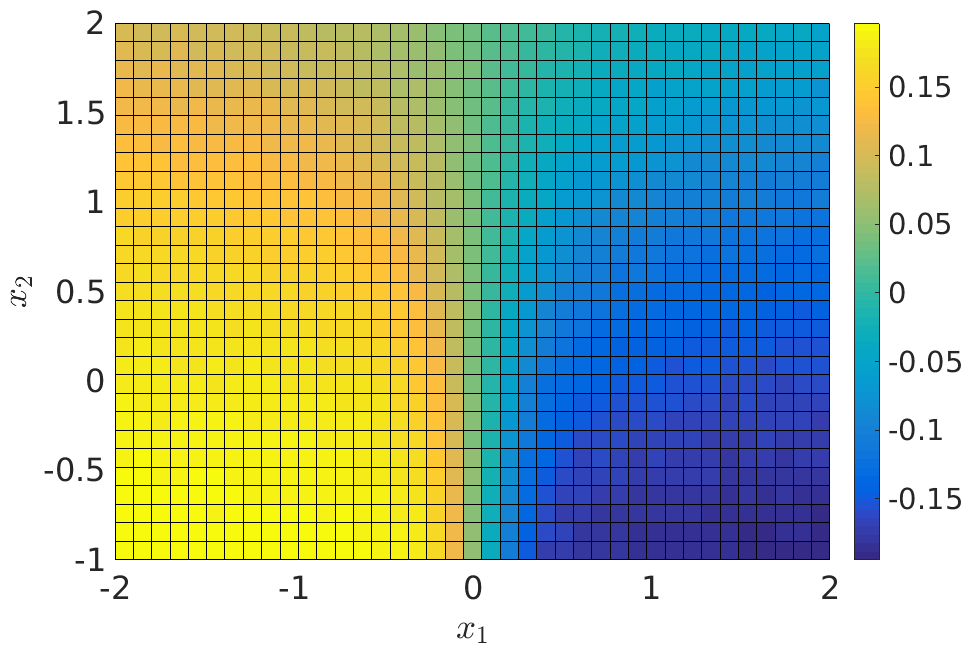}
    \end{minipage}
    \begin{minipage}{0.32\textwidth}
        \centering
        \subfiguretitle{$\big(\mathcal{E}_3\circ\overline{\mathcal{Q}}\big)(x)$}
        \includegraphics[width=\textwidth]{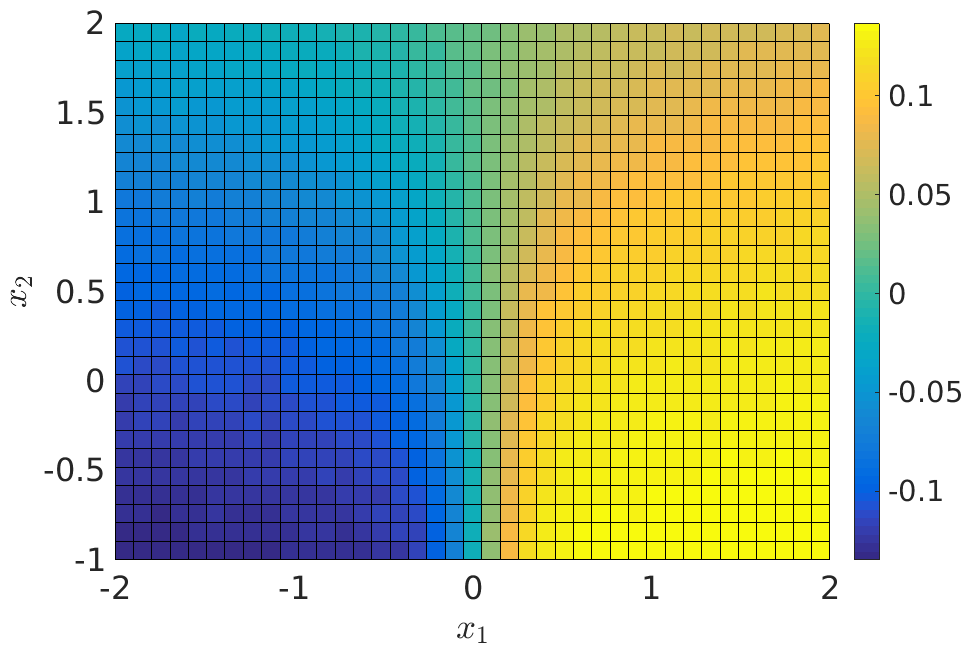}
    \end{minipage}    
    \caption{The individual components of the embedding $\mathcal{E}\circ\overline{\mathcal{Q}}$ on the grid points $x\in\overline{\mathbb{X}}$.}
    \label{fig:bananaembedding}
\end{figure}

For this example, the image of $\X$ under $ \mathcal{E} \circ \overline{\mathcal{Q}} $ should form a compact neighborhood of the one-dimensional manifold $\mathcal{E}(\mathbb{M})$, as described in Section~\ref{sec:NumApprox}. The one-dimensional structure in $\mathcal{E}\big(\overline{\mathcal{Q}}(\overline{\X})\big)$ is clearly visible, see Figure~\ref{fig:bananadiffusionmaps}a. To identify the one-dimensional coordinate along this set the diffusion map algorithm is used. Let $\Psi_1:\big(\mathcal{E}\circ\overline{\mathcal{Q}})(\overline{\X}) \rightarrow \R$ denote the first diffusion map coordinate on the embedded grid points, also visualized in Figure~\ref{fig:bananadiffusionmaps}a. The final reaction coordinate, shown in Figure~\ref{fig:bananadiffusionmaps}b, is then given by
\begin{equation*}
    \overline{\xi}(x) := \Psi_1\big(\big(\mathcal{E}\circ\overline{\mathcal{Q}}\big)(x)\big),\quad x\in\overline{\X}.
\end{equation*}

\begin{figure}[htbp]
    \begin{minipage}{0.49\textwidth}
        \centering
        \subfiguretitle{a)}
        \includegraphics[width=\textwidth]{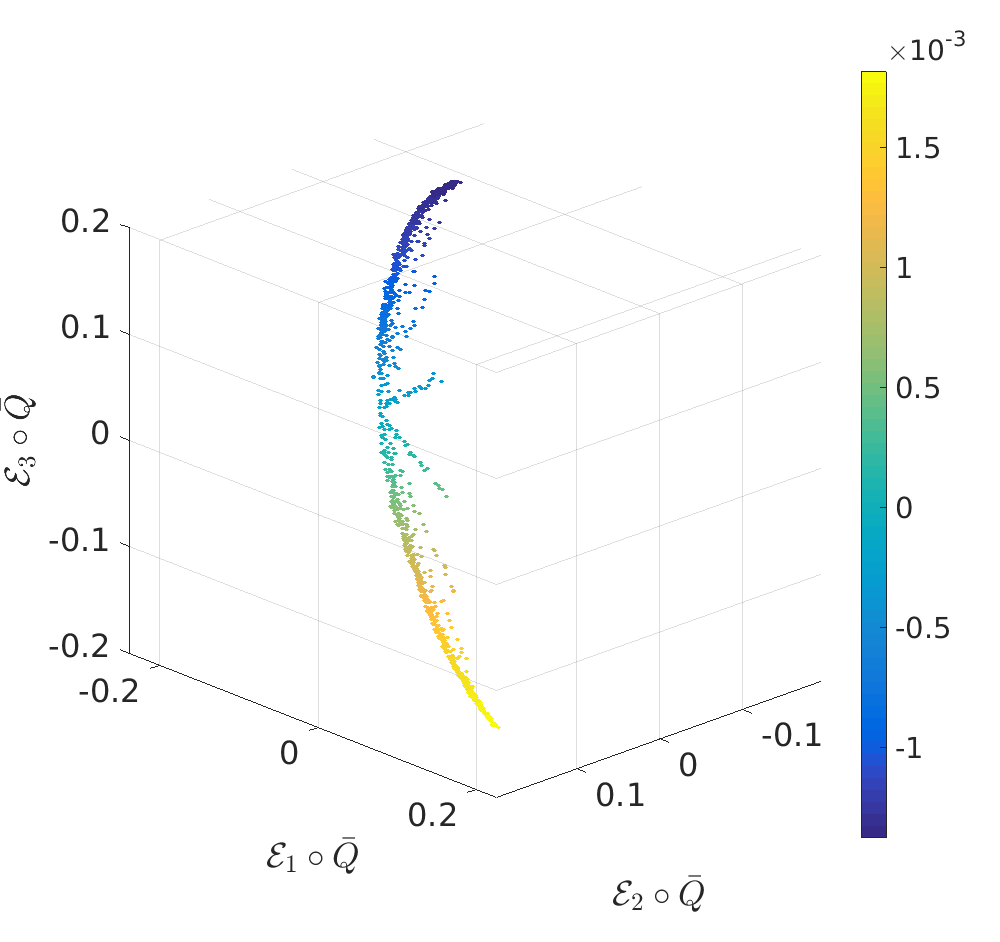}
    \end{minipage}
    \begin{minipage}{0.49\textwidth}
        \centering
        \subfiguretitle{b)}
        \includegraphics[width=\textwidth]{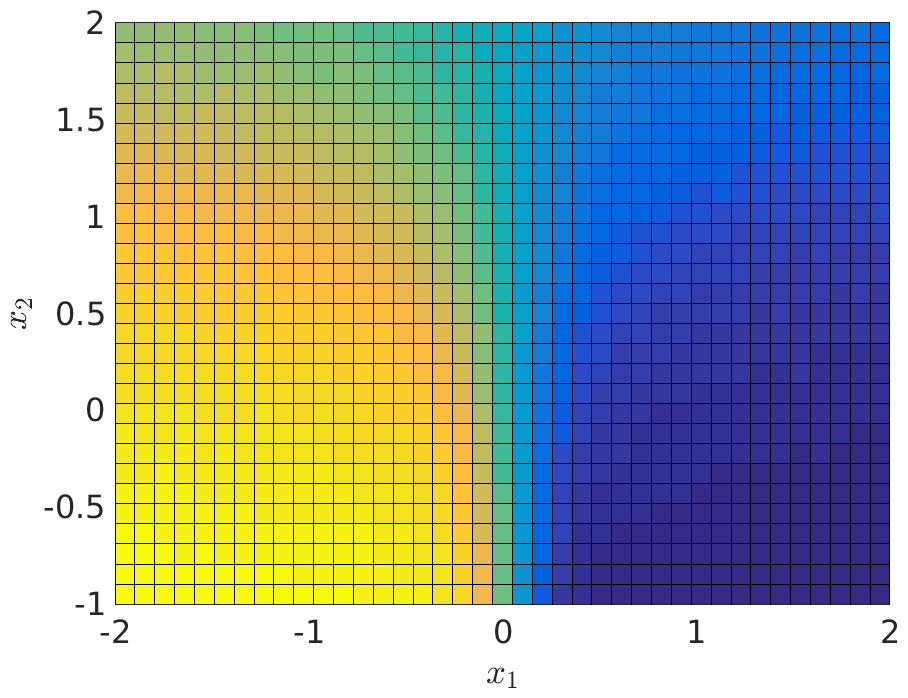}
    \end{minipage}
    \caption{a) The embedded grid points colored according to the first diffusion map coordinate. b) The final reaction coordinate $\overline{\xi}$.}
    \label{fig:bananadiffusionmaps}
\end{figure}

Legoll and Leli\`evre~\cite{LeLe10} show that the effective dynamics based on the reaction coordinate
\begin{equation*}
    \xi^*(x) = x_1\exp(-2x_2)
\end{equation*}
accurately reproduces the long-time dynamics of the full process --- although they do not use dominant eigenvalues of the transfer operator in their argumentation. It is easy to verify that the level sets of $\xi^*$ traverse the transition path orthogonally. Figure~\ref{fig:bananaLevelsets} compares the level sets of $\overline{\xi}$ and $\xi^*$. While the two reaction coordinates have different absolute values, their contour lines coincide well. As the projection operator $P_\xi$ only depends on the level sets of $\xi$, the projected transfer operators $\mathcal{T}^t_{\overline{\xi}}$ and $\mathcal{T}^t_{\xi^*}$ should be similar as well.

\begin{figure}[htbp]
    \centering
    \begin{minipage}{0.49\textwidth}
        \centering
        \subfiguretitle{$\overline{\xi}$}
        \includegraphics[width=\textwidth]{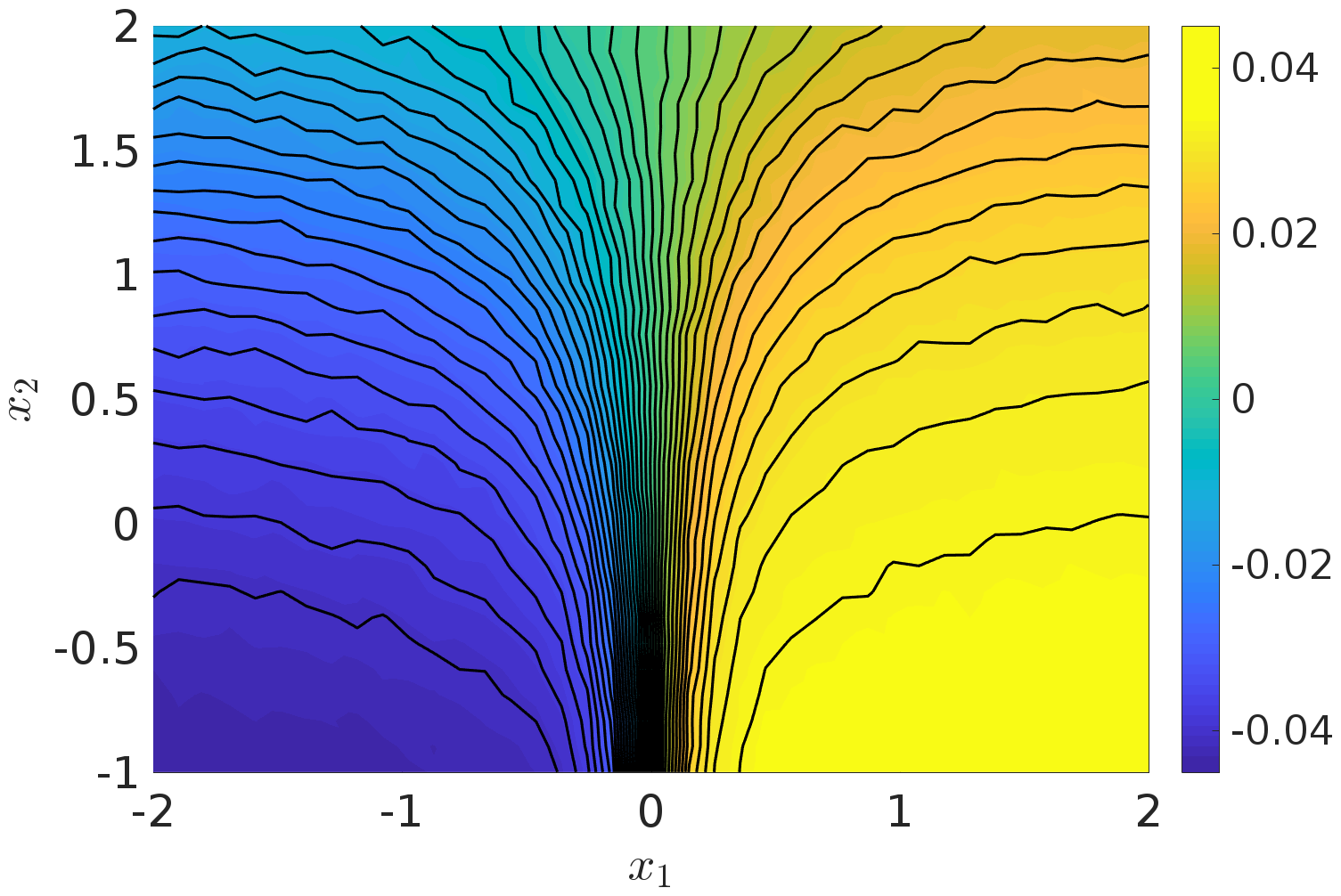}
    \end{minipage}
    \begin{minipage}{0.49\textwidth}
        \centering
        \subfiguretitle{$\xi^*$}
        \includegraphics[width=\textwidth]{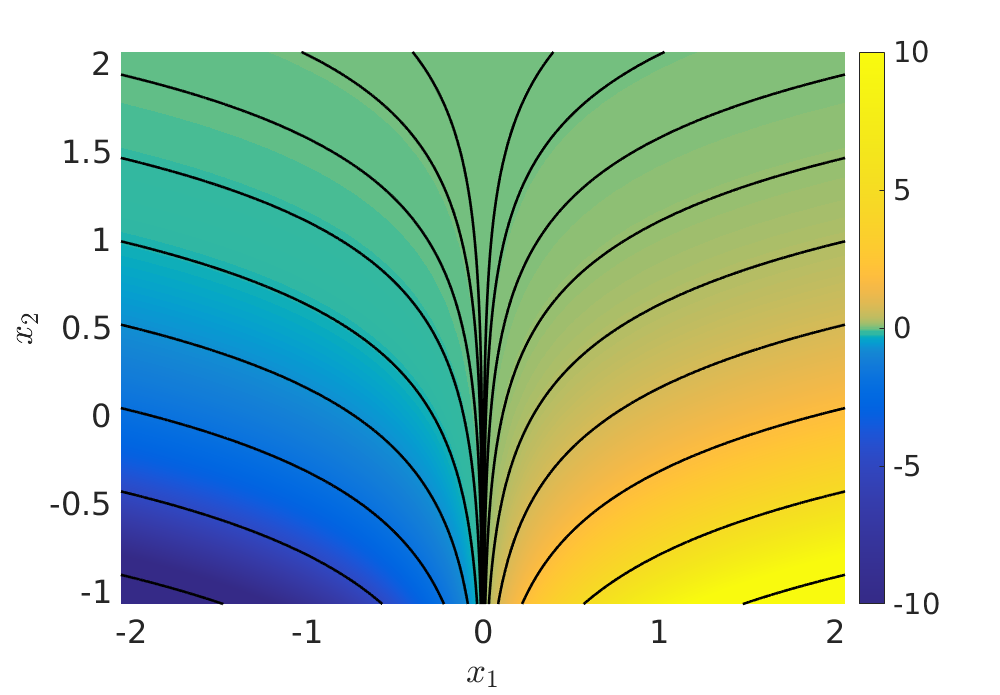}
    \end{minipage}
    \caption{Selected contour lines (black) of the newly identified reaction coordinate $\overline{\xi}$ and the reference reaction coordinate $\xi^*$.}
    \label{fig:bananaLevelsets}
\end{figure}

\paragraph{Projected eigenvalue error.}
To conclude this example, we compute the dominant spectrum of the projected transfer operator and compare it to the spectrum of the full transfer operator. To discretize $\mathcal{T}^t_{\overline{\xi}}$, we use a simple Ulam-type discretization scheme based on a long equilibrated trajectory of the full dynamics. Recall from Section~\ref{ssec:coord proj} that, although~$\mathcal{T}^t_\xi$ formally acts as an operator on functions over~$\X$, it is constant along level sets of~$\overline{\xi}$, and thus can be treated as an operator on functions over~$\mathbb{R}^r$. For completeness, we state the rough outline of an algorithm that we used to approximate~$\mathcal{T}^t_{\overline{\xi}}$.
An introduction to Ulam- and other Galerkin-type discretization schemes for transfer operators can be found, e.g., in~\cite{KKS16}.

\begin{enumerate}
\item Compute points~$\overline{\X} := \{\mathbf{\Phi}_{(k\tau)}x_0~|~k=1,\ldots,N\}$, a discrete trajectory with step size $\tau$ of the full phase space dynamics that adequately samples the invariant density~$\varrho$.
\item Compute the reaction coordinate~$\overline{\xi}$ on the points~$\overline{\X}$.
\item Divide the neighborhood of $\overline{\xi}(\overline{\X})$ into boxes or other suitable discretization elements $\{\mathbb{A}_1,\ldots,\mathbb{A}_N\}$ and sample the boxes from the trajectory, i.e. compute
$$
\overline{\X}_i := \{x\in\overline{\mathbb{X}}~|~\bar{\xi}(x) \in \mathbb{A}_i\}~.
$$
\item Count the time-$t$-transitions within $\overline{\X}$ between the boxes (where $t$ is a multiple of $\tau$), i.e.~compute the matrix
$$
\big(T^t_{\overline{\xi}}\big)_{ij} := \#\big\{x\in \overline{\X}_i~|~\mathbf{\Phi}_tx \in \overline{\X}_j\big\}~.
$$
\item After row-normalization, the eigenvalues of $T^t_{\overline{\xi}}$ approximate the point spectrum of~$\mathcal{T}^t_{\overline{\xi}}$.
\end{enumerate}

\begin{remark}
Note that the equilibrated trajectory $\overline{\mathbb{X}}$ is typically unavailable for more complex systems. In practice, one would replace steps 1 and 2 by directly computing a reduced trajectory $\overline{\mathbb{Z}}=\{z_1,\ldots,z_N\}\subset\mathbb{R}^r$ whose statistics approximate that of $\xi\big(\overline{\mathbb{X}}\big)$. The formulation of a reduced numerical integration scheme to realize this is currently work in progress (see the fifth point in the conclusions).
\end{remark}

For our example system, we compute $\overline{\X}$ as a $N=10^6$ step trajectory with step size $\tau=10^{-2}$ using the Euler-Maruyama scheme. However, to reduce the numerical effort, $\overline{\xi}$ is computed only on a subsample of $\overline{\mathbb{X}}$ ($10^4$ points) and extended to $\overline{\mathbb{X}}$ by nearest-neighbor interpolation. On $\overline{\X}$, the image of the $\overline{\xi}$ is contained in the interval $[-0.04,0.04]$, which we discretize into $M=40$ subintervals of equal length. The spectrum of the full transfer operator $\mathcal{T}^t$ was computed using the standard Ulam method over a $40\times 30$ uniform box discretization of the domain $[-2,2]\times[-1,2]$. With the choice $t=1$ for the lag time, the spectral gap is clearly visible.

We observe in Figure \ref{fig:BananaReducedEigs} that the eigenvalues of $\mathcal{T}^t_{\overline{\xi}}$ and $\mathcal{T}^t$ are in excellent agreement. Not only the dominant eigenvalues $\lambda_0,\lambda_1$ are approximated well (as predicted by Lemma \ref{lem:EVapprox}), but also the further subdominant eigenvalues that are not covered by our theory. In particular, the reaction coordinate $\overline{\xi}$ provides a better approximation to the spectrum of $\mathcal{T}^t$ than other, manually chosen reaction coordinates: Figure \ref{fig:BananaReducedEigs} also shows the eigenvalues of the projected transfer operator associated with the reaction coordinates 
$$
\zeta_1(x):=x_1\quad\text{and}\quad \zeta_2(x) := x_1+x_2.
$$
We see that these are consistently outperformed by the computed reaction coordinate $\overline{\xi}$ (although it appears that $\zeta_1$ already is quite a good reaction coordinate).

\begin{figure}[htbp]
    \centering
    \begin{minipage}{0.59\textwidth}
        \centering
        \subfiguretitle{a)}
            \includegraphics[width=\textwidth]{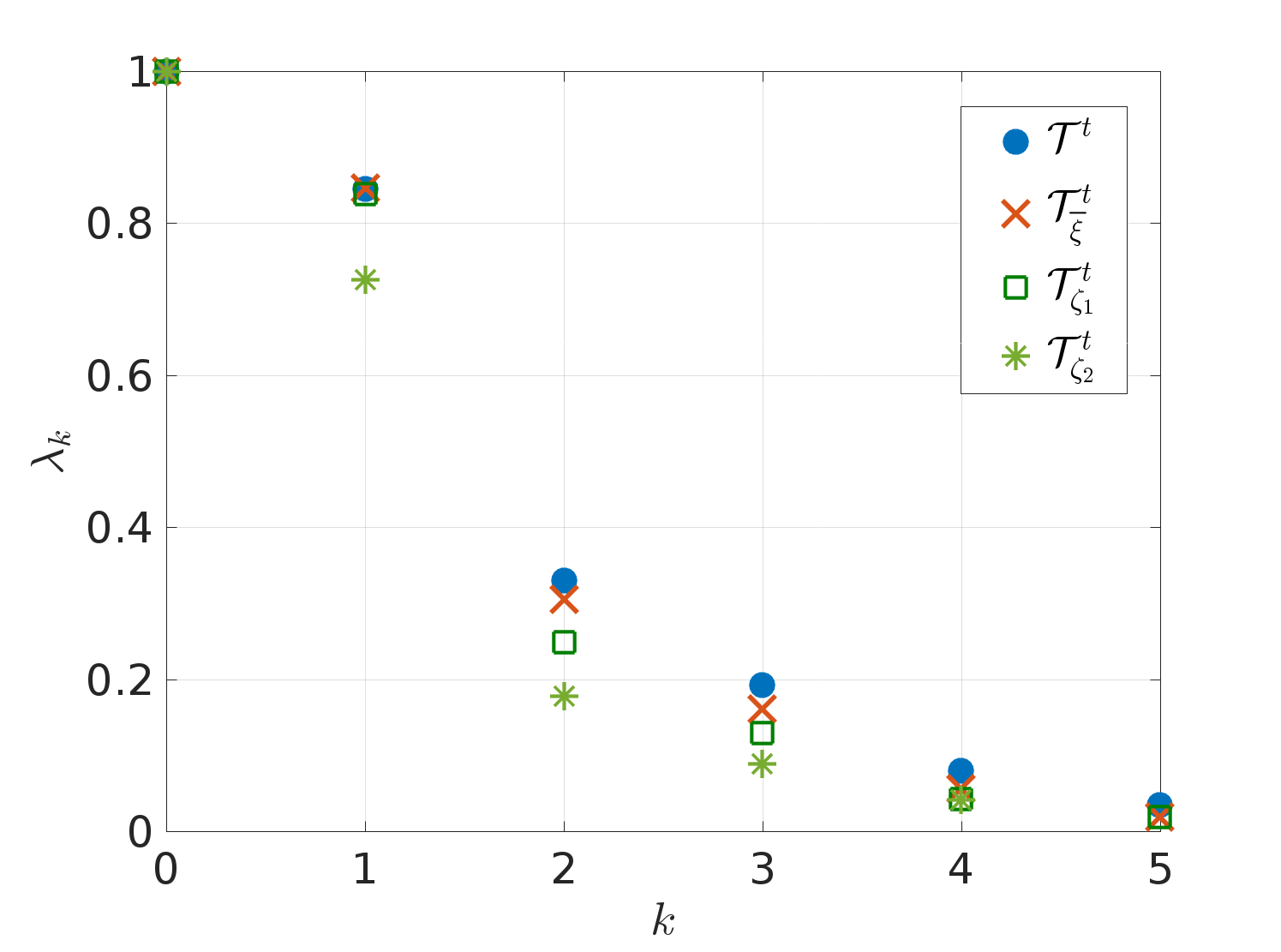}
    \end{minipage}             
    \begin{minipage}{0.39\textwidth}
    \centering
    \subfiguretitle{b)}
    {\renewcommand{\arraystretch}{1.3}
    \begin{tabular}{ c | c}
    & $\lambda_1$\\
    \hline
    $\mathcal{T}^t$ & 0.8503 \\
    $\mathcal{T}^t_{\overline{\xi}}$ & 0.8461 \\
    $\mathcal{T}^t_{\zeta_1}$ & 0.8377 \\
    $\mathcal{T}^t_{\zeta_2}$ & 0.7252 
    \end{tabular}
    }
    \end{minipage}
    \caption{a) Comparison of the two dominant and first four non-dominant eigenvalues of the full transfer operator $\mathcal{T}^t$ and the projected transfer operators $\mathcal{T}^t_{\overline{\xi}},\mathcal{T}^t_{\zeta_1},\mathcal{T}^t_{\zeta_2}$. b) Detailed comparison of the second eigenvalue of the various transfer operators.}
    \label{fig:BananaReducedEigs}
\end{figure}

\subsection{Circular potential}
\label{sec:Lemon Slice}

Let us now compute the reaction coordinates for the multi-well diffusion process described in Example~\ref{ex:LemonSlice}. The corresponding $ k $-well potential is defined as
\begin{equation*}
     V(x) = \cos\left(k \, \arctan(x_2, x_1)\right) + 10 \left(\sqrt{x_1^2 + x_2^2} - 1\right)^2.
\end{equation*}
We use $ k = 7 $, for which the potential is shown in Figure~\ref{fig:minimalreactioncoordinate}a. The potential as well as the dominant eigenvalues of the corresponding transfer operator clearly indicate the existence of seven metastable sets, yet a typical longtime trajectory, shown in Figure~\ref{fig:lemontrajectory}a, suggests a one-dimensional transition path, the unit circle $\mathbb{B}_1$. We demonstrate that with our method, a reaction coordinate of minimal dimension can be computed.

\begin{figure}
    \centering
    \begin{minipage}{0.45\textwidth}
            \centering
            \subfiguretitle{a)}
            \includegraphics[width=\textwidth]{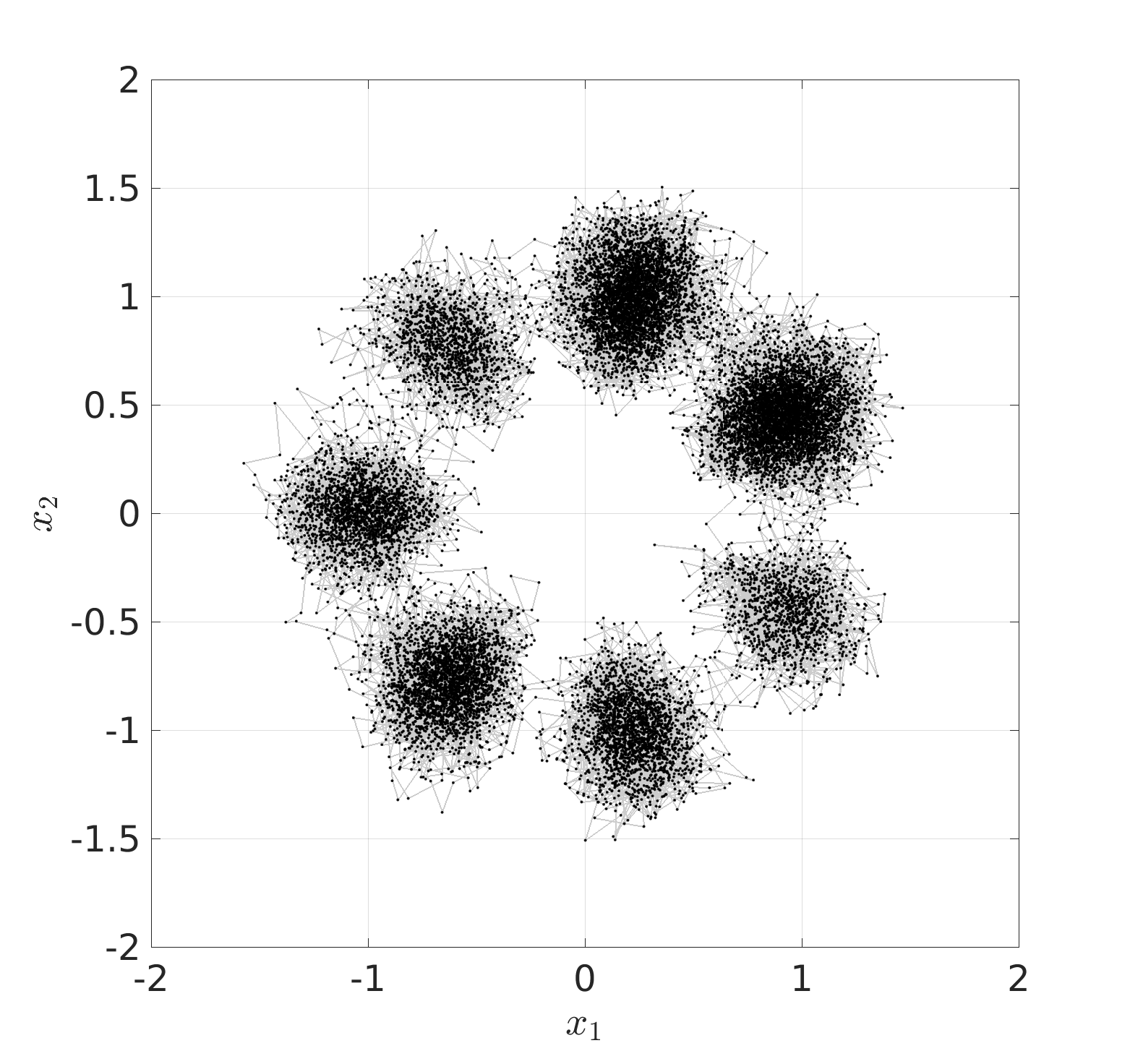}
    \end{minipage}
    \begin{minipage}{0.45\textwidth}
            \centering
            \subfiguretitle{b)}
            \includegraphics[width=\textwidth]{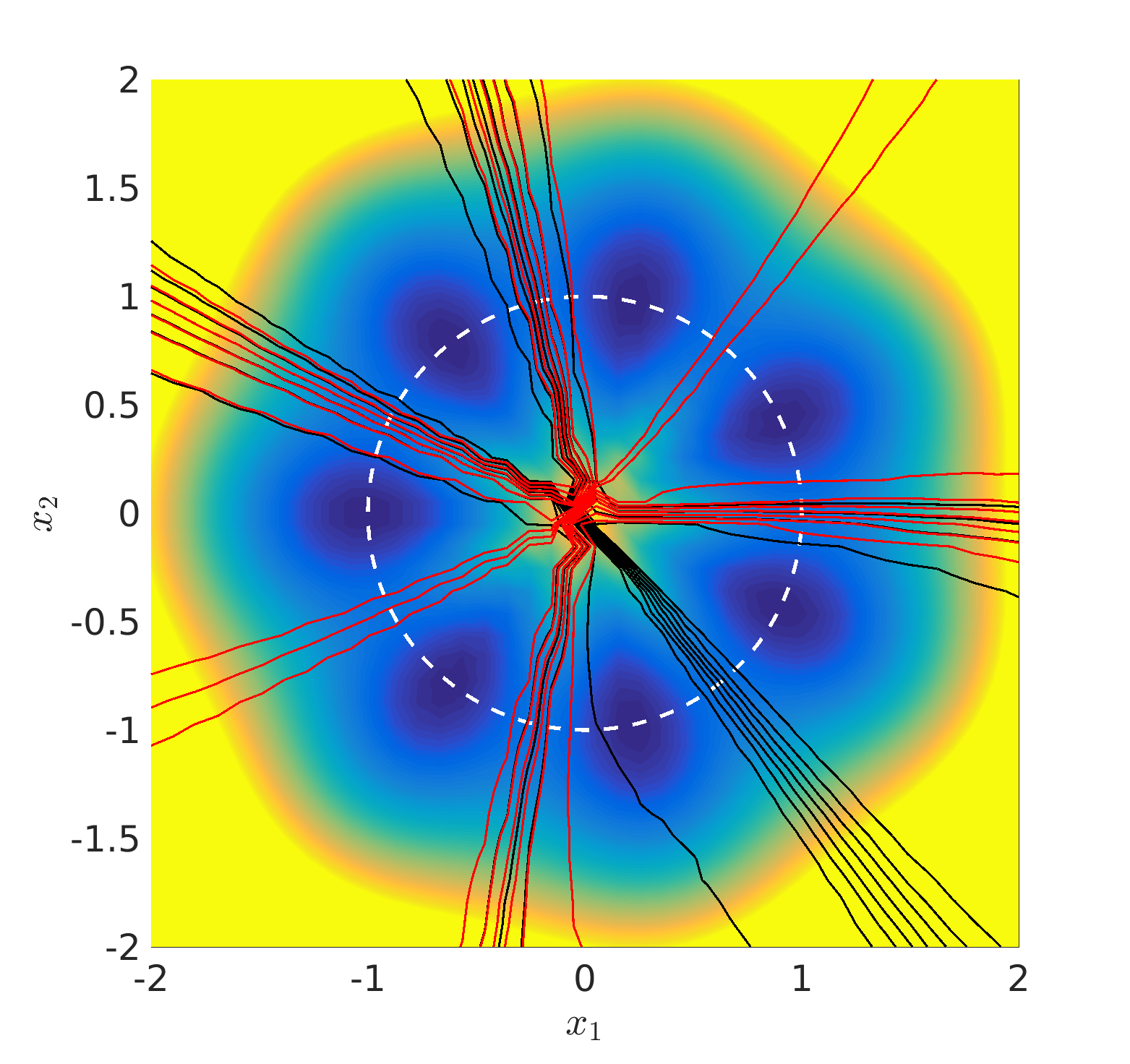}
    \end{minipage}    
    \caption{a) Longtime trajectory of the diffusion process with the circular seven-well potential. b) The contour lines of $\overline{\xi}_1$ (black) and $\overline{\xi}_2$ (red) show that $\overline{\xi}$ is almost constant on the metastable sets, but resolves the transition regions well.}
    \label{fig:lemontrajectory}
\end{figure}

We again choose the inverse temperature $\beta=0.5$ and perform the same analysis as in the previous subsection. For this system, a time scale gap between $t_6 \approx 1.53$ and $t_7 \approx 0.05$ can be found. We thus choose the intermediate time scale $ t = 0.1 $. Since we again expect a one-dimensional transition path, the three observables~\eqref{eq:bananaobservables} are used for the embedding of $\mathbb{M}$. We use the grid points of a $40 \times 40$ grid, denoted again by $\overline{\X}$, over the region $[-2,2]\times [-2,2]$ as our test points.

The individual components of the embedding $\mathcal{E}\circ\overline{\mathcal{Q}}$ are shown in Figure~\ref{fig:lemonembedding}. The embedded grid points, seen as the individual points in Figure~\ref{fig:lemondiffusionmaps}a, seem to concentrate around a one-dimensional circular manifold and thus reveal the one-dimensional nature of the reaction coordinate. Although slightly unintuitive, the diffusion maps algorithm now identifies \emph{two} significant diffusion map components, as shown in Figure~\ref{fig:lemondiffusionmaps}a. The reason is that the circular manifold cannot be embedded into $\R^1$, so that a two-component coordinate is necessary to parametrize it. Figure \ref{fig:lemontrajectory}b shows some contour lines (of equidistant values) of the two components of $\overline{\xi}$. We see that $\overline{\xi}$ is almost constant on the seven metastable sets, but resolves the transition regions well.

\begin{figure}[htbp]
    \centering
    \begin{minipage}{0.32\textwidth}
        \centering
        \subfiguretitle{$\big(\mathcal{E}_1\circ\overline{\mathcal{Q}}\big)(x)$}
        \includegraphics[width=\textwidth]{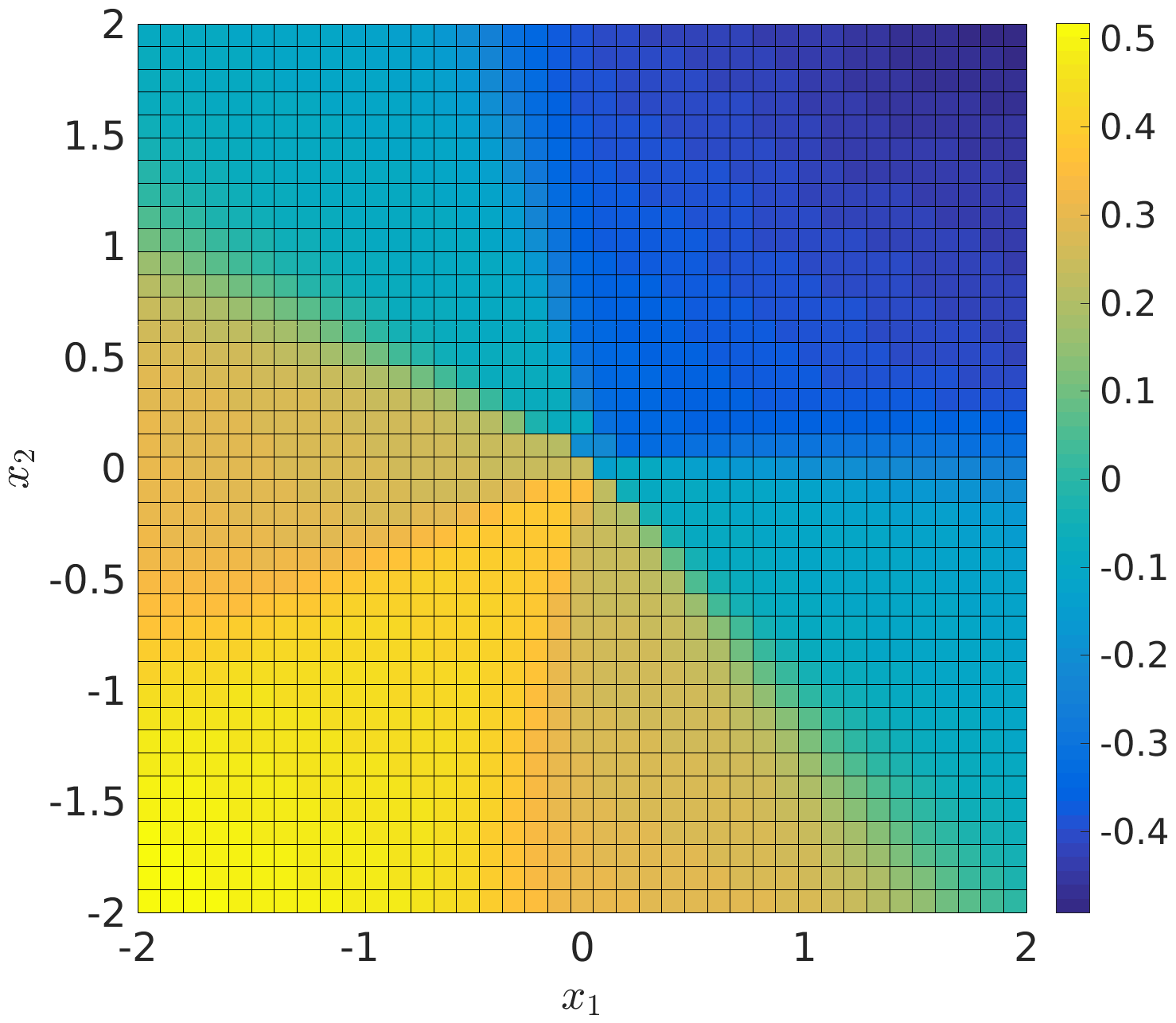}
    \end{minipage}
    \begin{minipage}{0.32\textwidth}
        \centering
        \subfiguretitle{$\big(\mathcal{E}_2\circ\overline{\mathcal{Q}}\big)(x)$}
        \includegraphics[width=\textwidth]{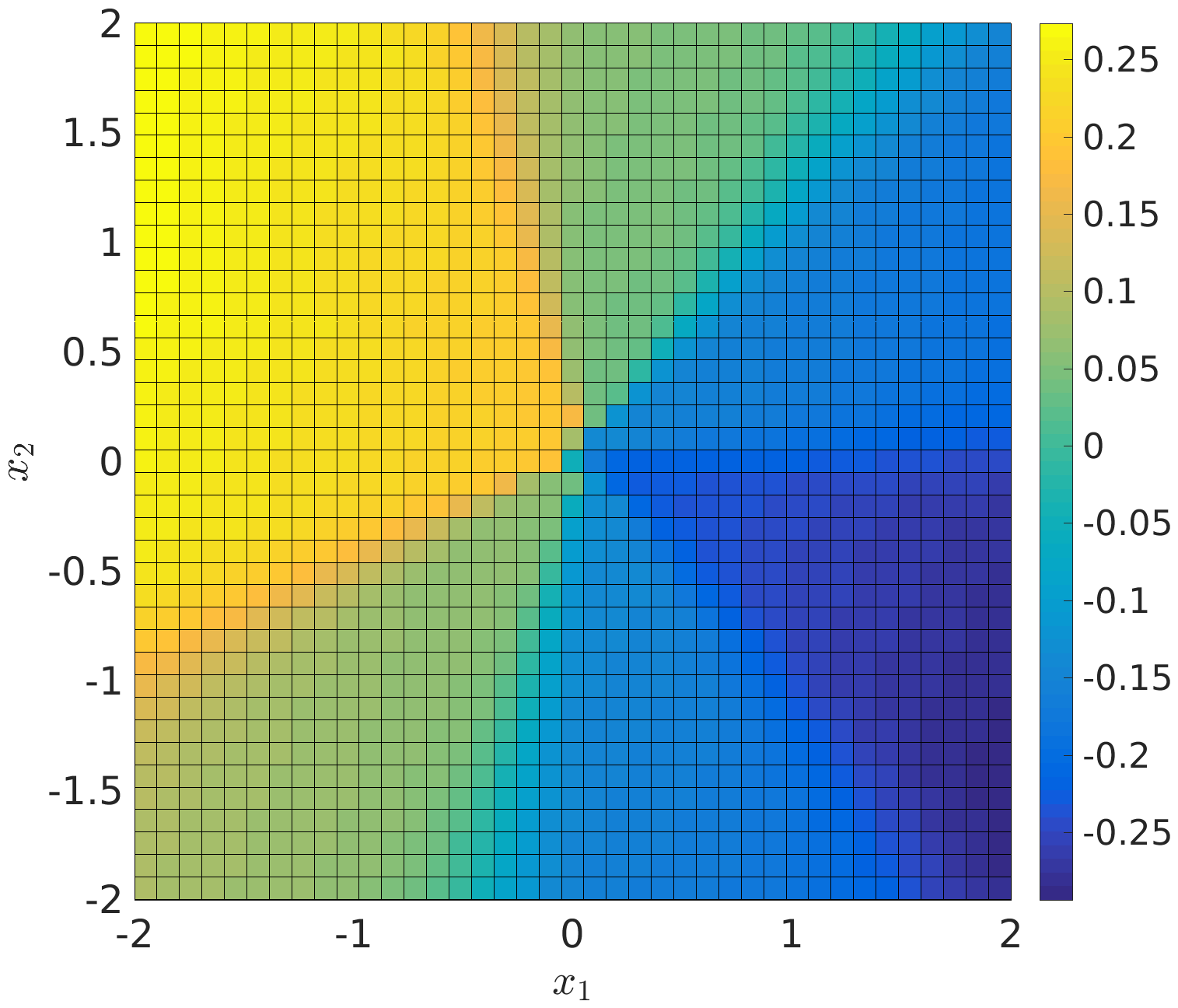}
    \end{minipage}
    \begin{minipage}{0.32\textwidth}
        \centering
        \subfiguretitle{$\big(\mathcal{E}_3\circ\overline{\mathcal{Q}}\big)(x)$}
        \includegraphics[width=\textwidth]{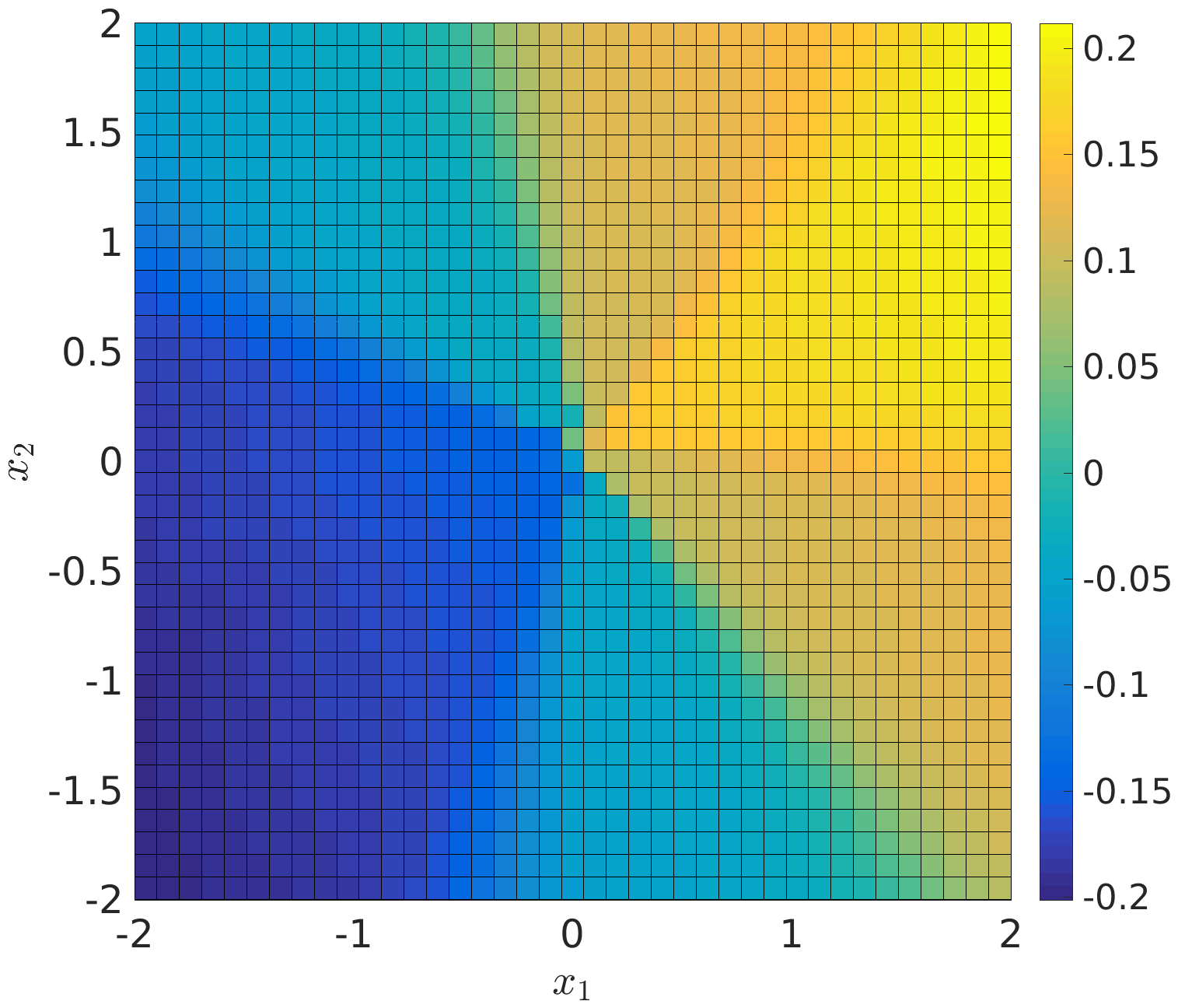}
    \end{minipage}    
    \caption{The individual components of the embedding $\mathcal{E}\circ\overline{\mathcal{Q}}$ on the grid points $x\in\overline{\mathbb{X}}$. }
    \label{fig:lemonembedding}
\end{figure}

\begin{figure}[htb]
    \centering
    \begin{minipage}{0.4\textwidth}
        \centering
        \subfiguretitle{a)}
        \includegraphics[width=\textwidth]{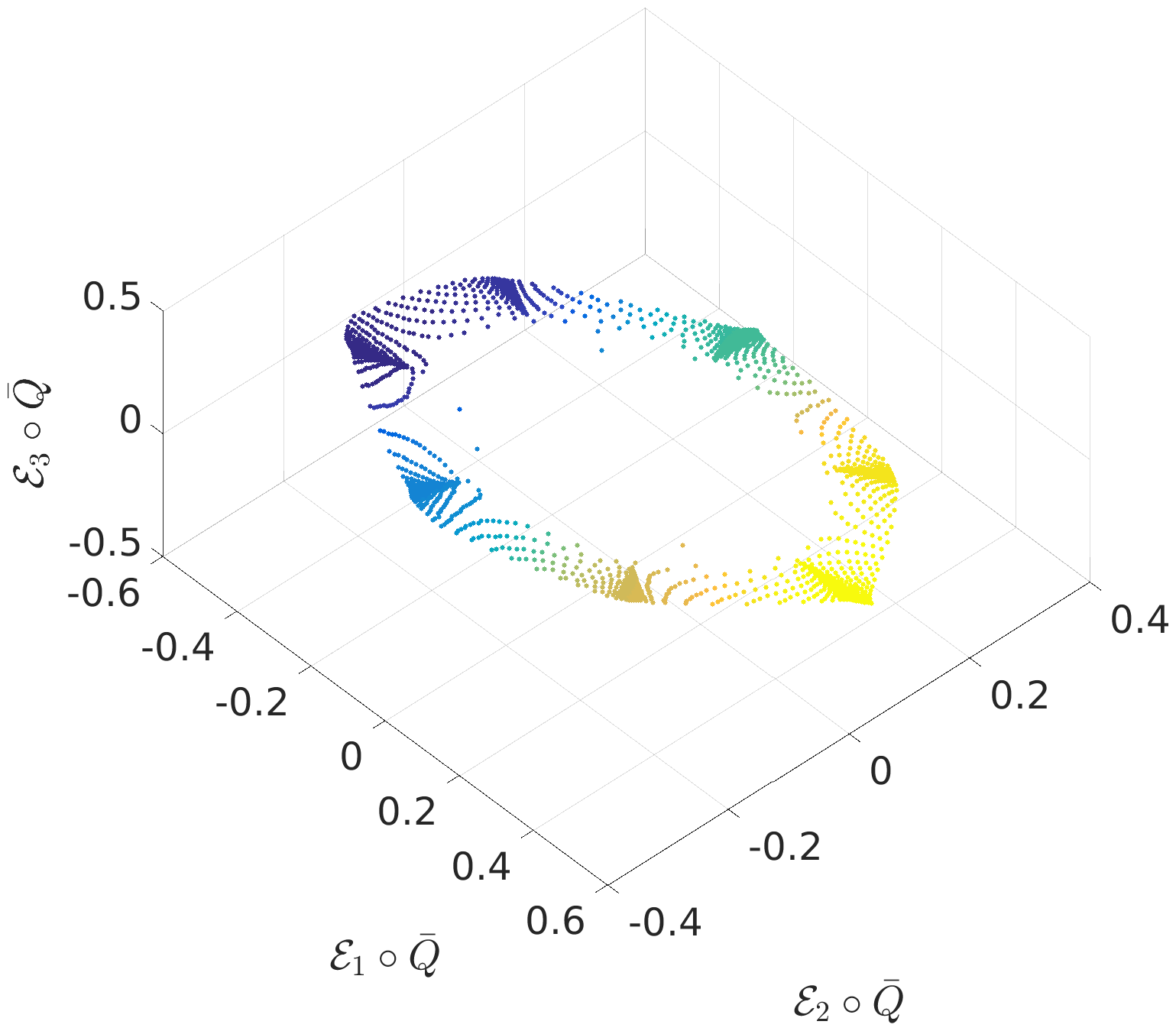}
    \end{minipage}
    \begin{minipage}{0.4\textwidth}
        \centering
        \subfiguretitle{b)}
        \includegraphics[width=\textwidth]{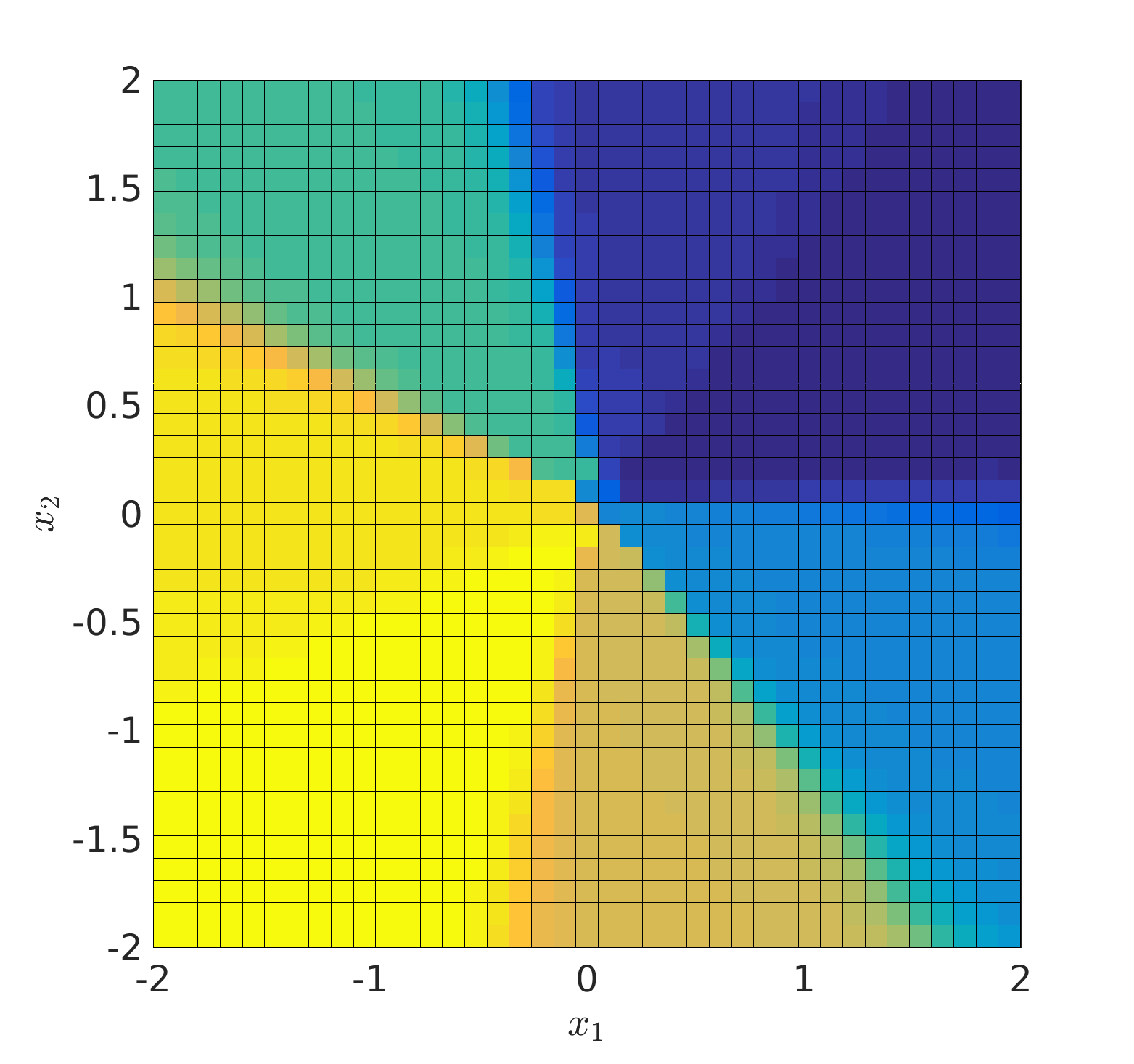}
    \end{minipage}
    
	\begin{minipage}{0.4\textwidth}
        \centering
        \subfiguretitle{c)}
        \includegraphics[width=\textwidth]{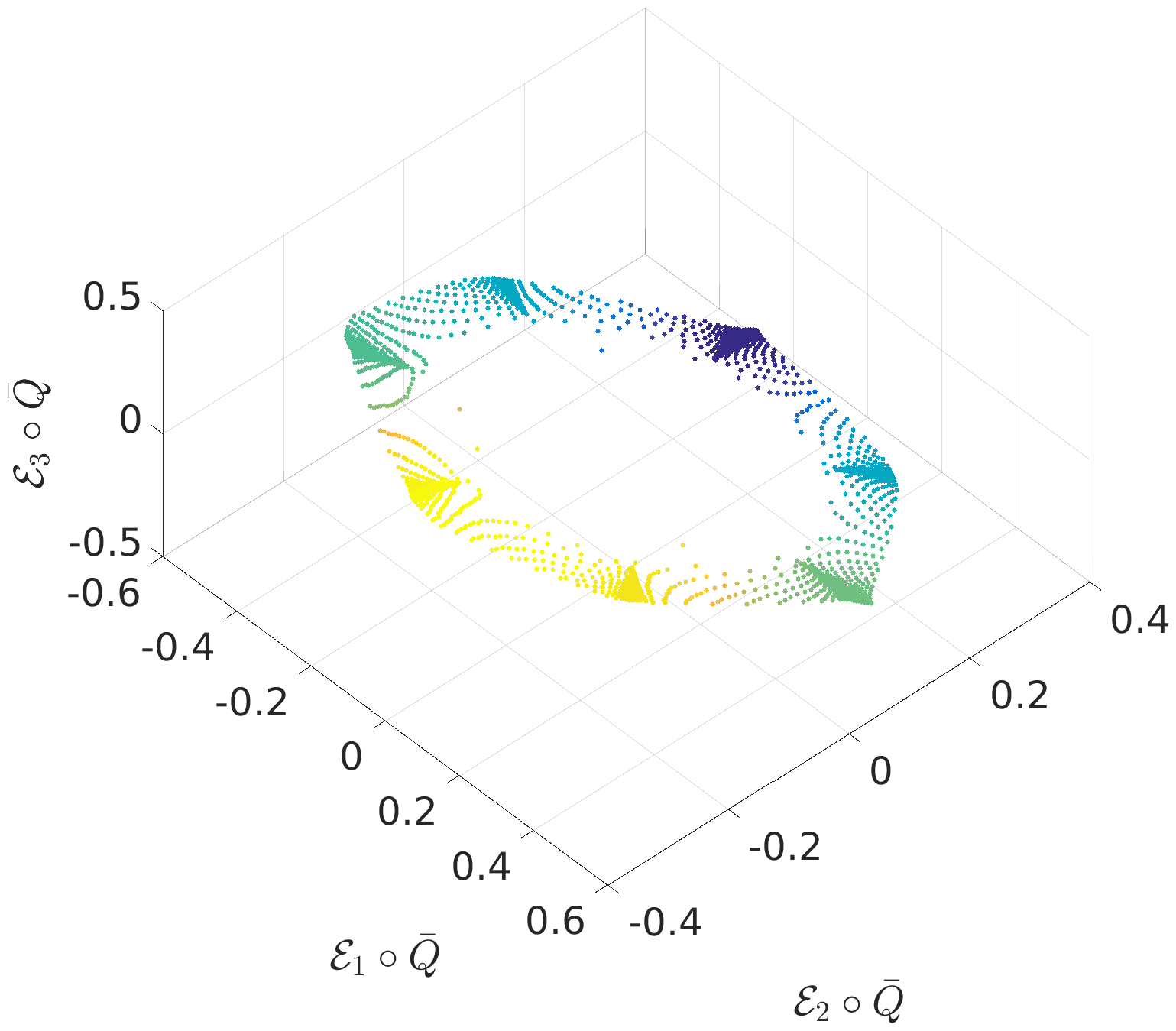}
    \end{minipage}
    \begin{minipage}{0.4\textwidth}
        \centering
        \subfiguretitle{d)}
        \includegraphics[width=\textwidth]{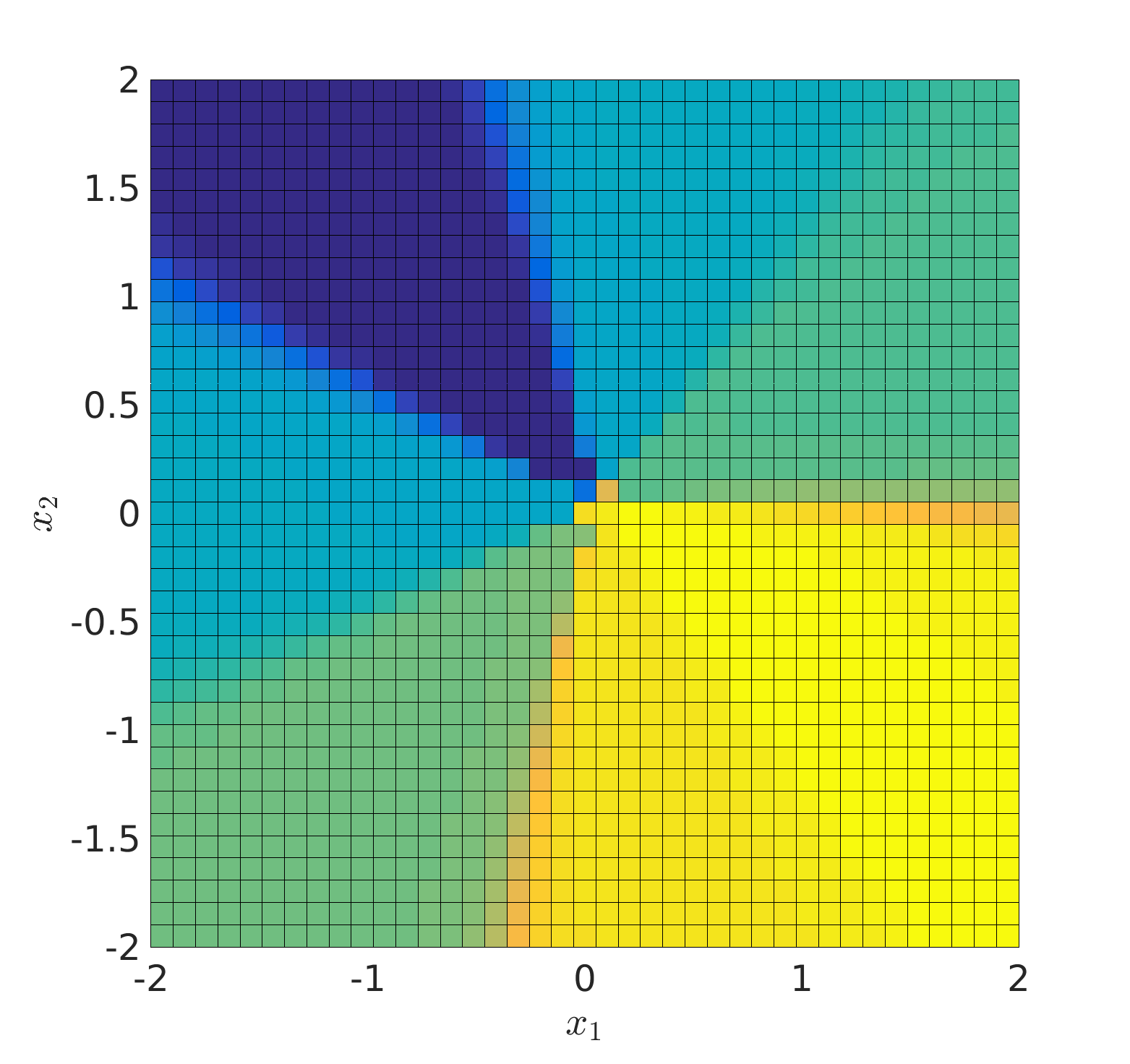}
    \end{minipage}
\caption{Left column: The embedded grid points $\mathcal{E}(\overline{\mathcal{Q}}(\overline{\mathbb{X}}))$. The coloring shows the a) first and c) second significant diffusion map on the points. Right column: The b) first and d) second components of the final reaction coordinate $\overline{\xi}$.}
\label{fig:lemondiffusionmaps}
\end{figure}

\paragraph{Parametrization of the dominant eigenfunctions.}
Next, we experimentally investigate how well the dominant eigenfunctions $\varphi_i$ of $\mathcal{T}^t$ can be parametrized by the numerically computed reaction coordinate $\overline{\xi}$. If the eigenfunctions are almost functions of~$\overline{\xi}$, then by Lemma~\ref{lem:parametrizableeigenfunctions1} and Corollary~\ref{cor:characRC} the reaction coordinate is suitable to reproduce \emph{all the dominant time scales}. To this end, we compute the dominant eigenfunctions~$\varphi_j$, $j=0,\ldots,d$ by the Ulam-type Galerkin method (as in the previous example), and plot~$\varphi_j(x_i)$ against~$\overline{\xi}(x_i)$. Note that due to the reasons discussed above, the range of~$\overline{\xi}$ is a one-dimensional manifold in~$\mathbb{R}^2$.
If $\varphi_j$ can be parametrized by $\bar{\xi}$, we expect that~$\varphi_j(x_{i_1}) \approx \varphi_j(x_{i_2})$, whenever~$\bar{\xi}(x_{i_1}) \approx \bar{\xi}(x_{i_2})$.
 The result is shown in Figure~\ref{fig:functionaldependency}. We clearly see the functional dependency of the first seven (i.e., the dominant) eigenfunctions on the reaction coordinate.

\begin{figure}
	\begin{minipage}{0.32\textwidth}
        \centering
        \subfiguretitle{$\varphi_0 \text{ vs. }\bar{\xi} $}
        \includegraphics[width=\textwidth]{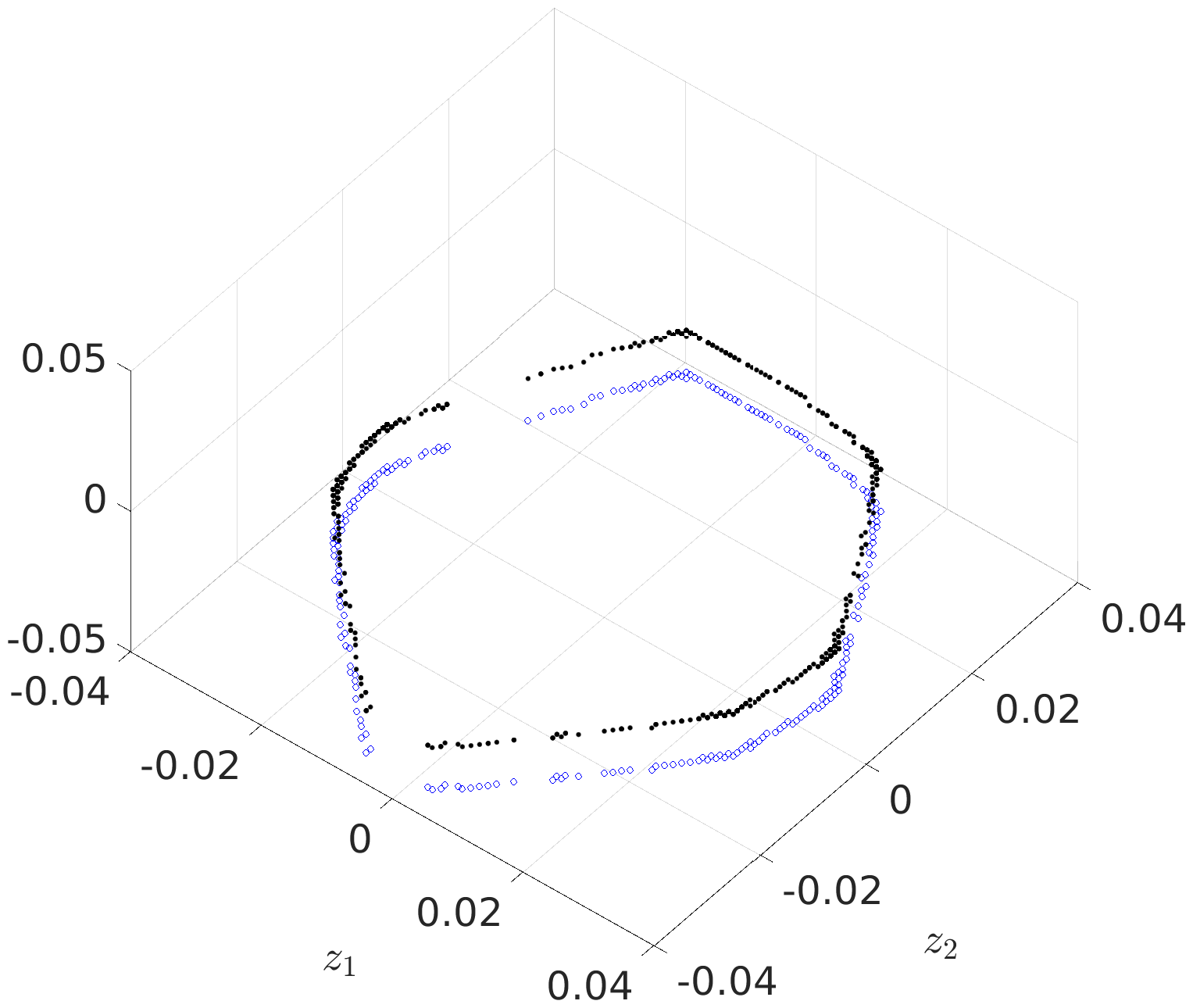}
    \end{minipage}
    \begin{minipage}{0.32\textwidth}
        \centering
        \subfiguretitle{$\varphi_1 \text{ vs. }\bar{\xi} $}
        \includegraphics[width=\textwidth]{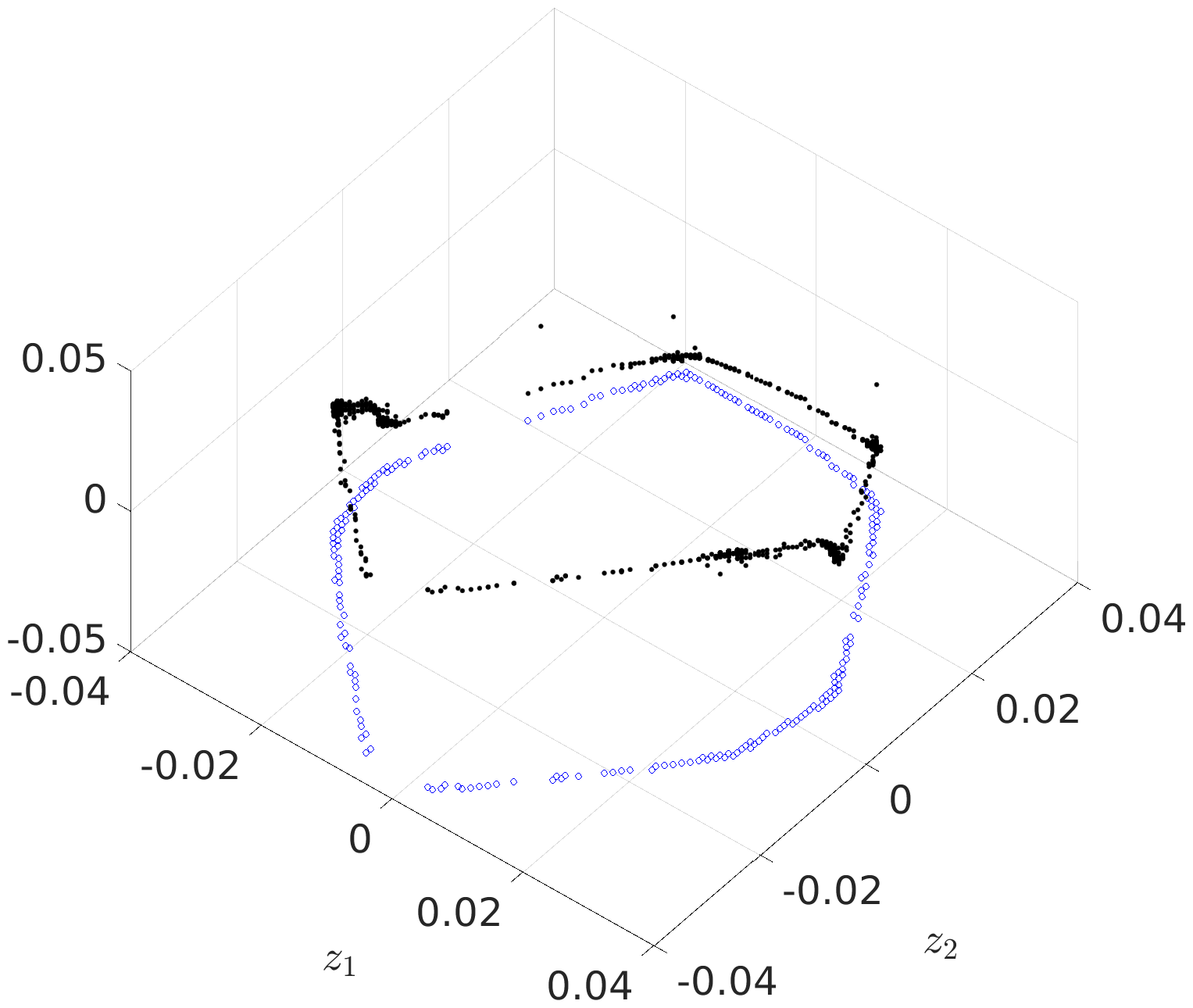}
    \end{minipage}
    \begin{minipage}{0.32\textwidth}
        \centering
        \subfiguretitle{$\varphi_2 \text{ vs. }\bar{\xi} $}
        \includegraphics[width=\textwidth]{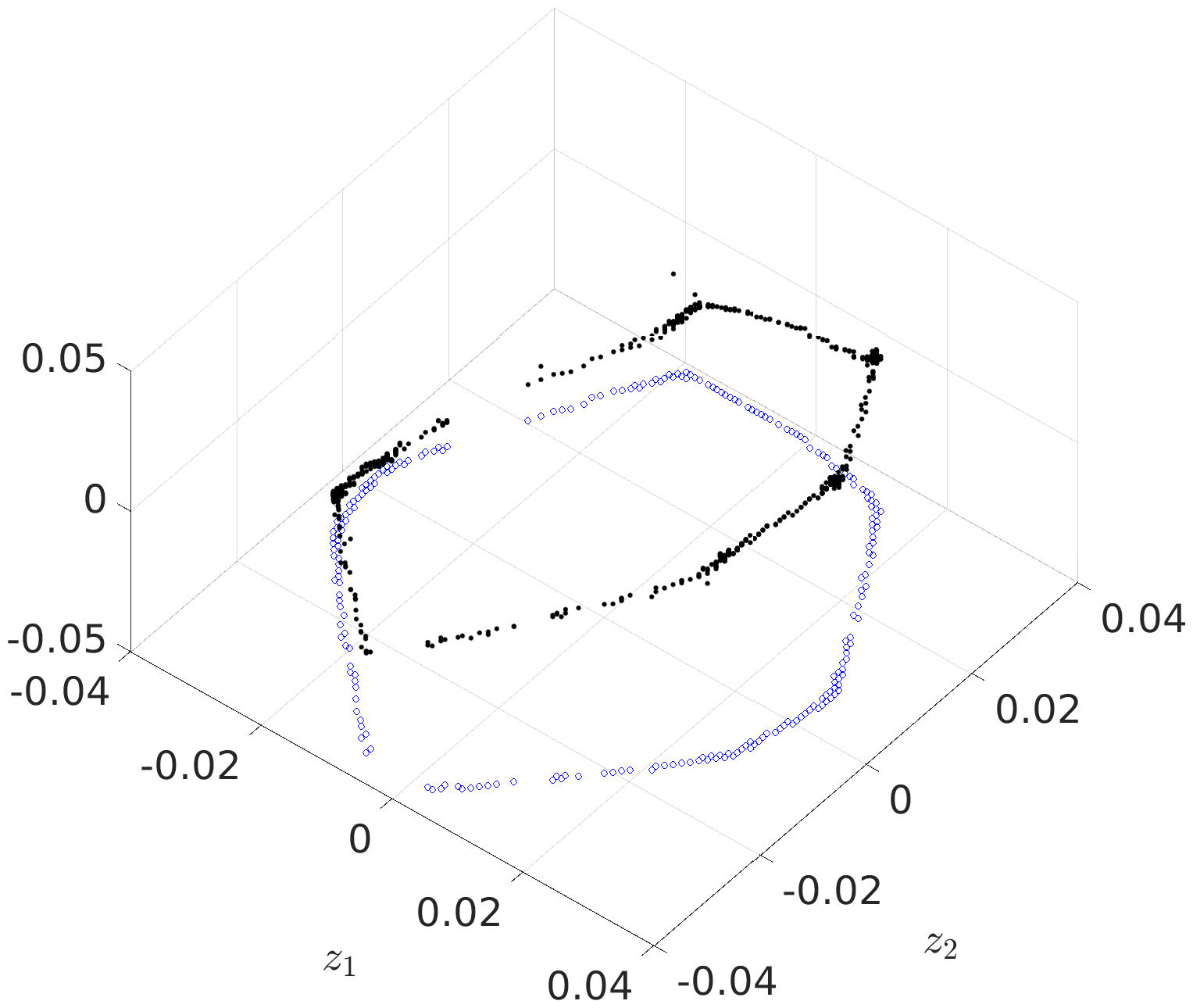}
    \end{minipage}
    \vspace{0.7cm}
    
	\begin{minipage}{0.32\textwidth}
        \centering
        \subfiguretitle{$\varphi_3 \text{ vs. }\bar{\xi} $}
        \includegraphics[width=\textwidth]{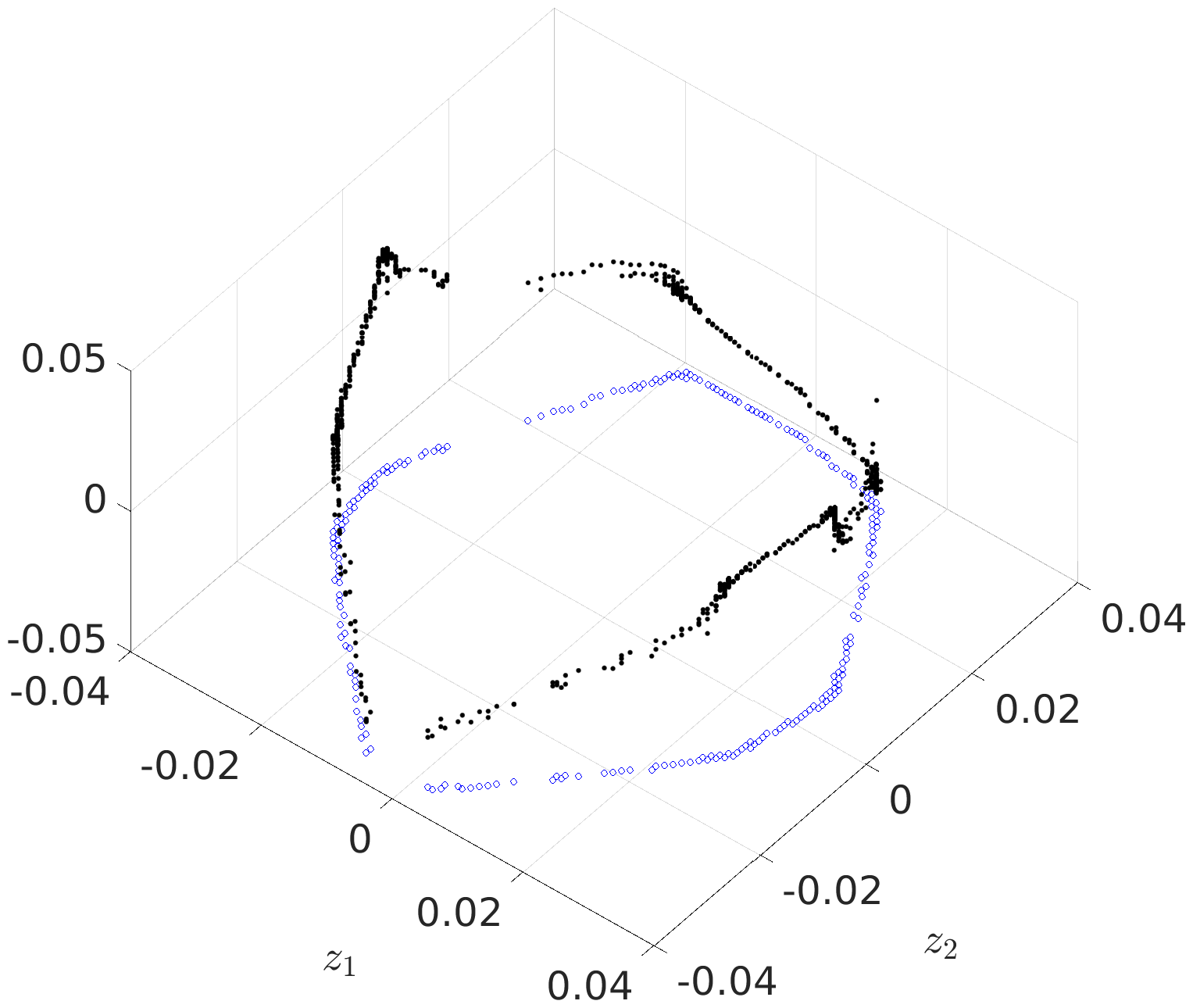}
    \end{minipage}
    \begin{minipage}{0.32\textwidth}
        \centering
        \subfiguretitle{$\varphi_4 \text{ vs. }\bar{\xi} $}
        \includegraphics[width=\textwidth]{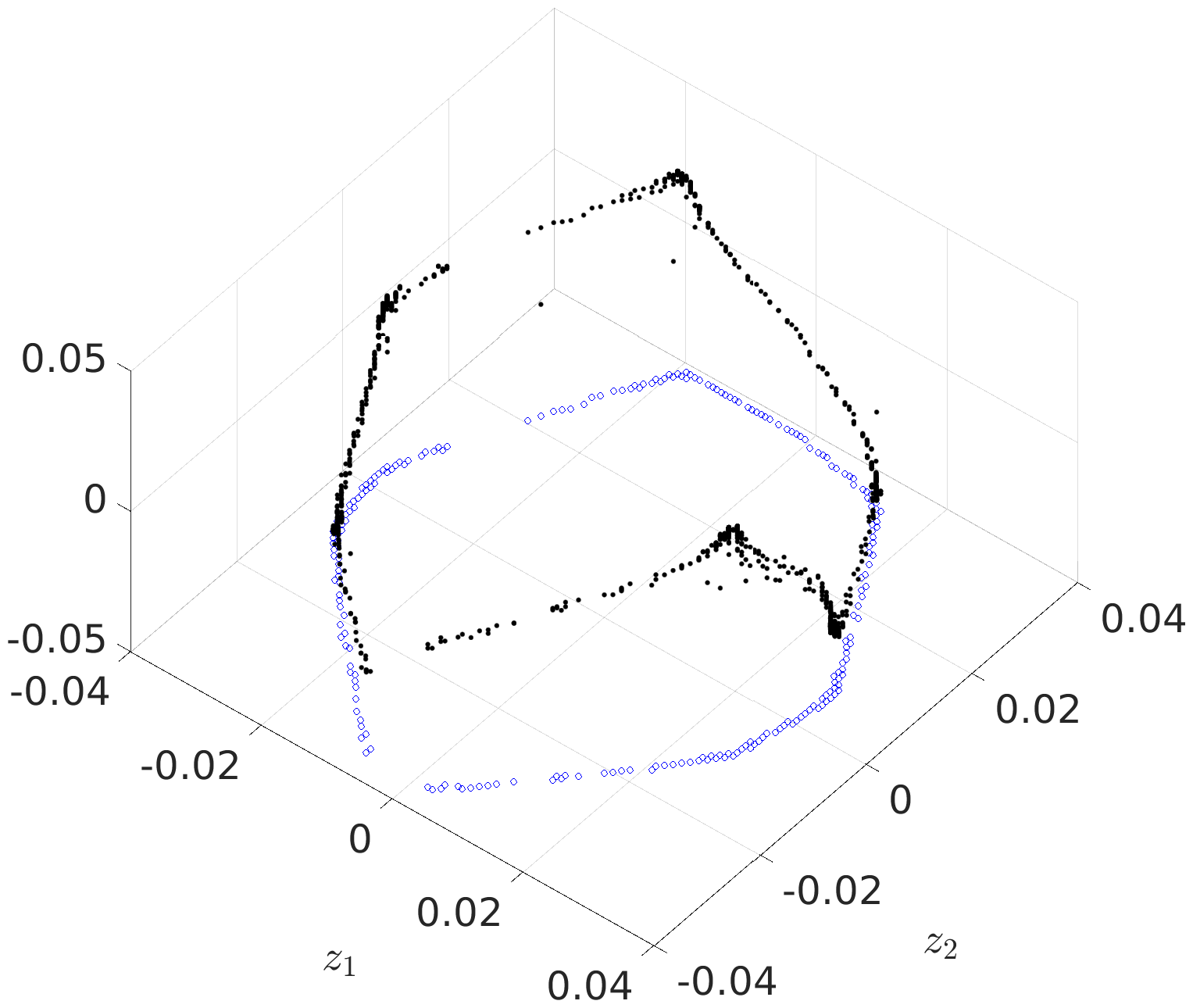}
    \end{minipage}
    \begin{minipage}{0.32\textwidth}
        \centering
        \subfiguretitle{$\varphi_5 \text{ vs. }\bar{\xi} $}
        \includegraphics[width=\textwidth]{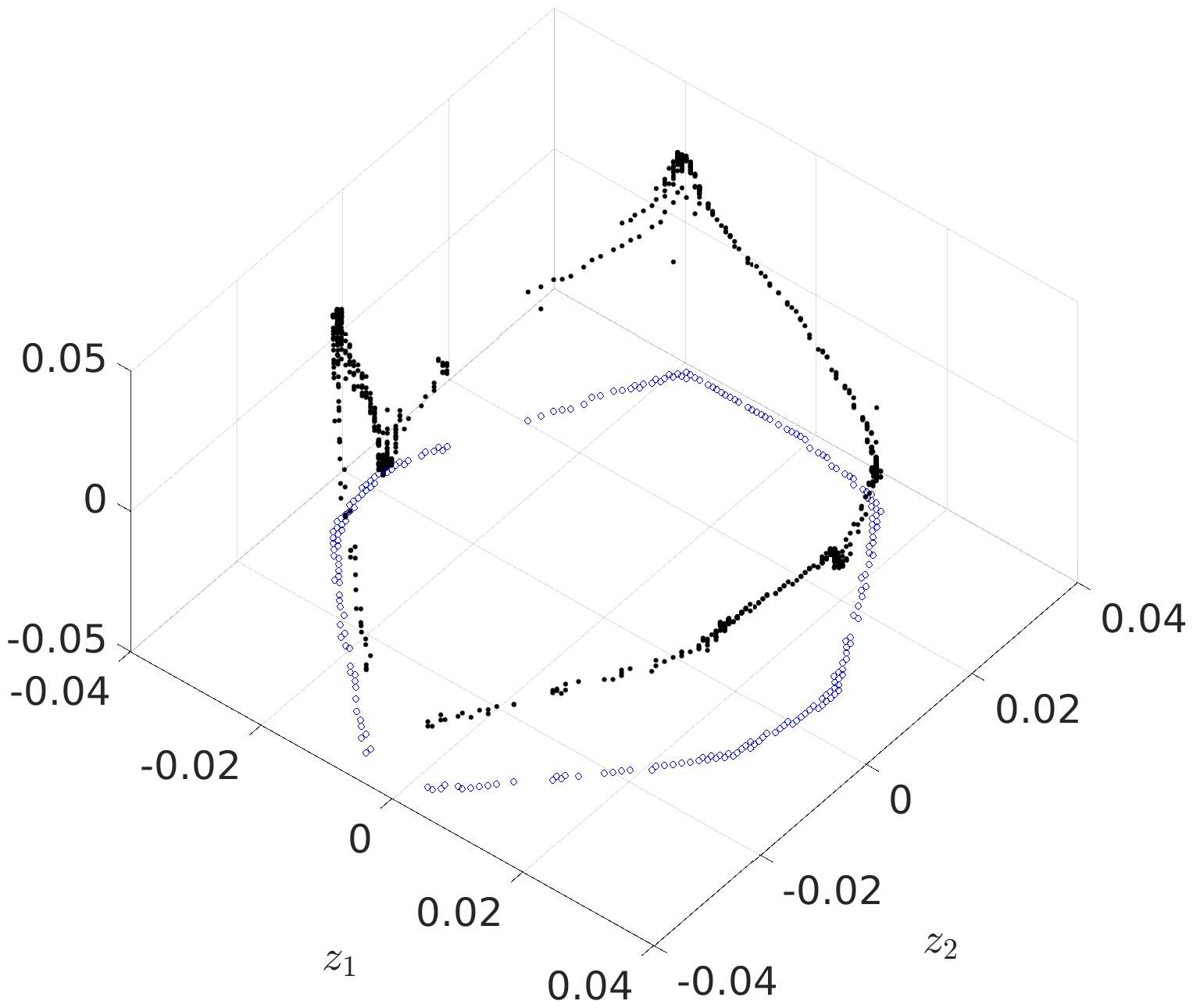}
    \end{minipage}   
    \vspace{0.7cm}

	\begin{minipage}{0.32\textwidth}
        \centering
        \subfiguretitle{$\varphi_6 \text{ vs. }\bar{\xi} $}
        \includegraphics[width=\textwidth]{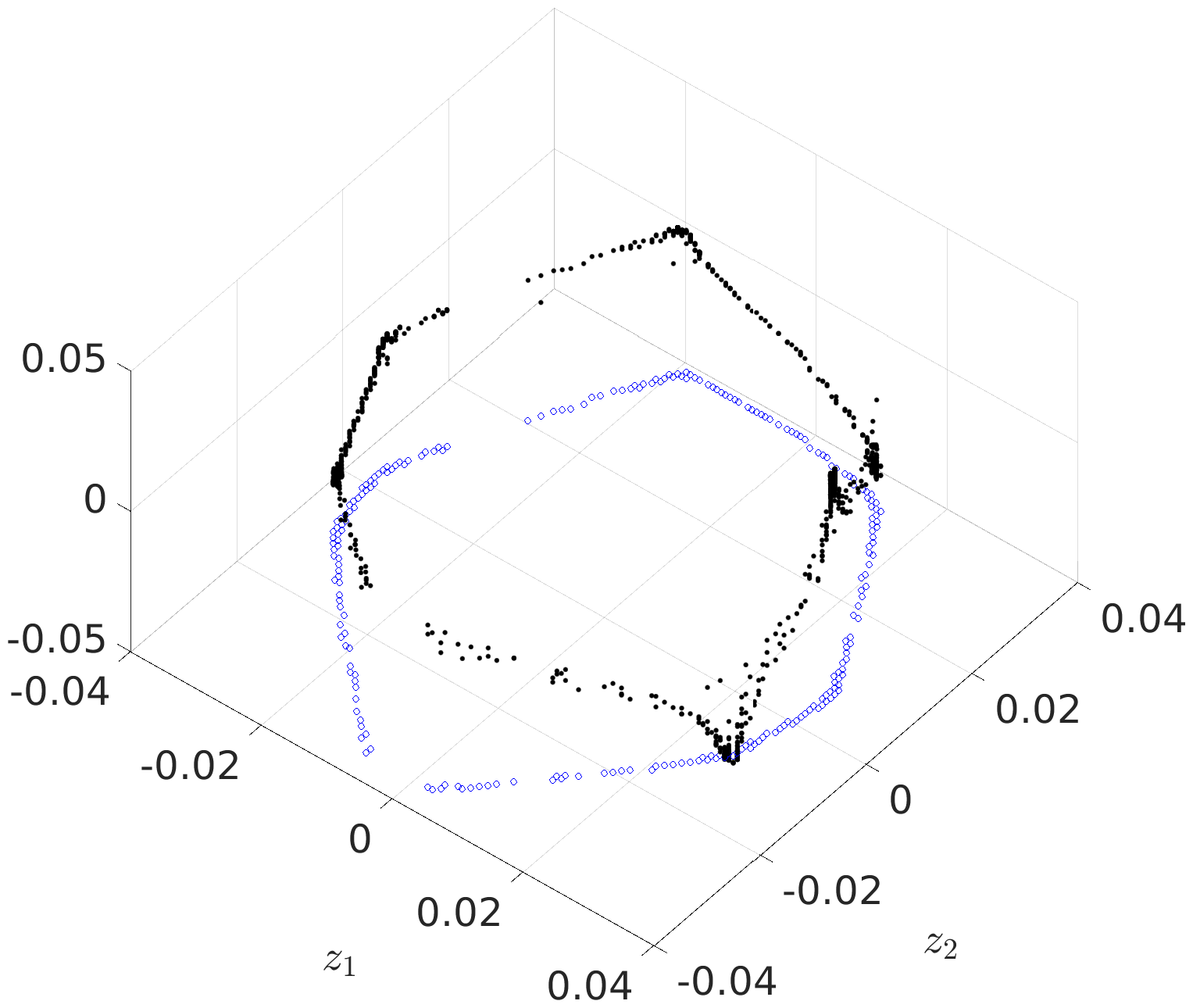}
    \end{minipage}
    \begin{minipage}{0.32\textwidth}
        \centering
        \subfiguretitle{$\varphi_7 \text{ vs. }\bar{\xi} $}
        \includegraphics[width=\textwidth]{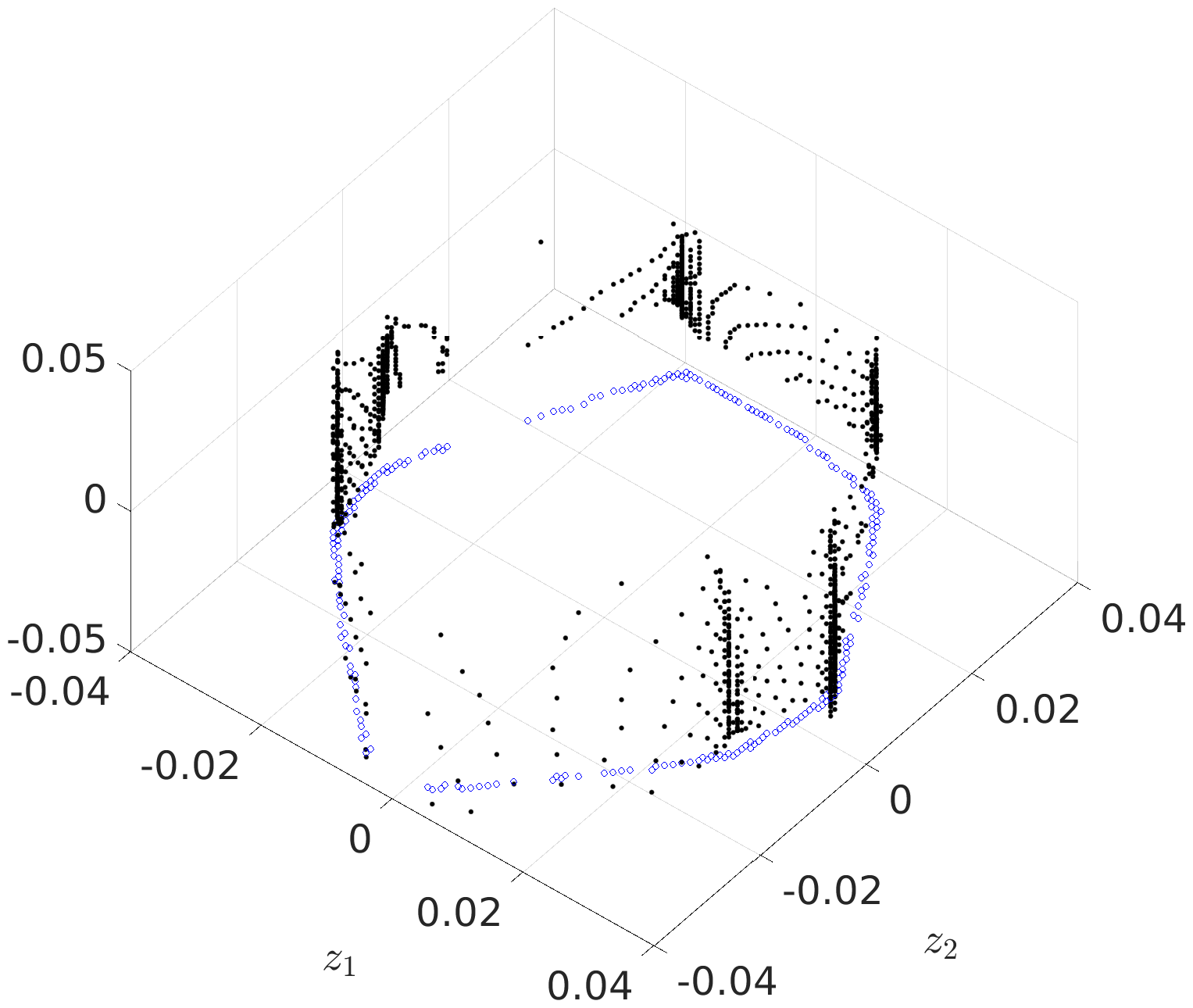}
    \end{minipage}
    \begin{minipage}{0.32\textwidth}
        \centering
        \subfiguretitle{$\varphi_8 \text{ vs. }\bar{\xi} $}
        \includegraphics[width=\textwidth]{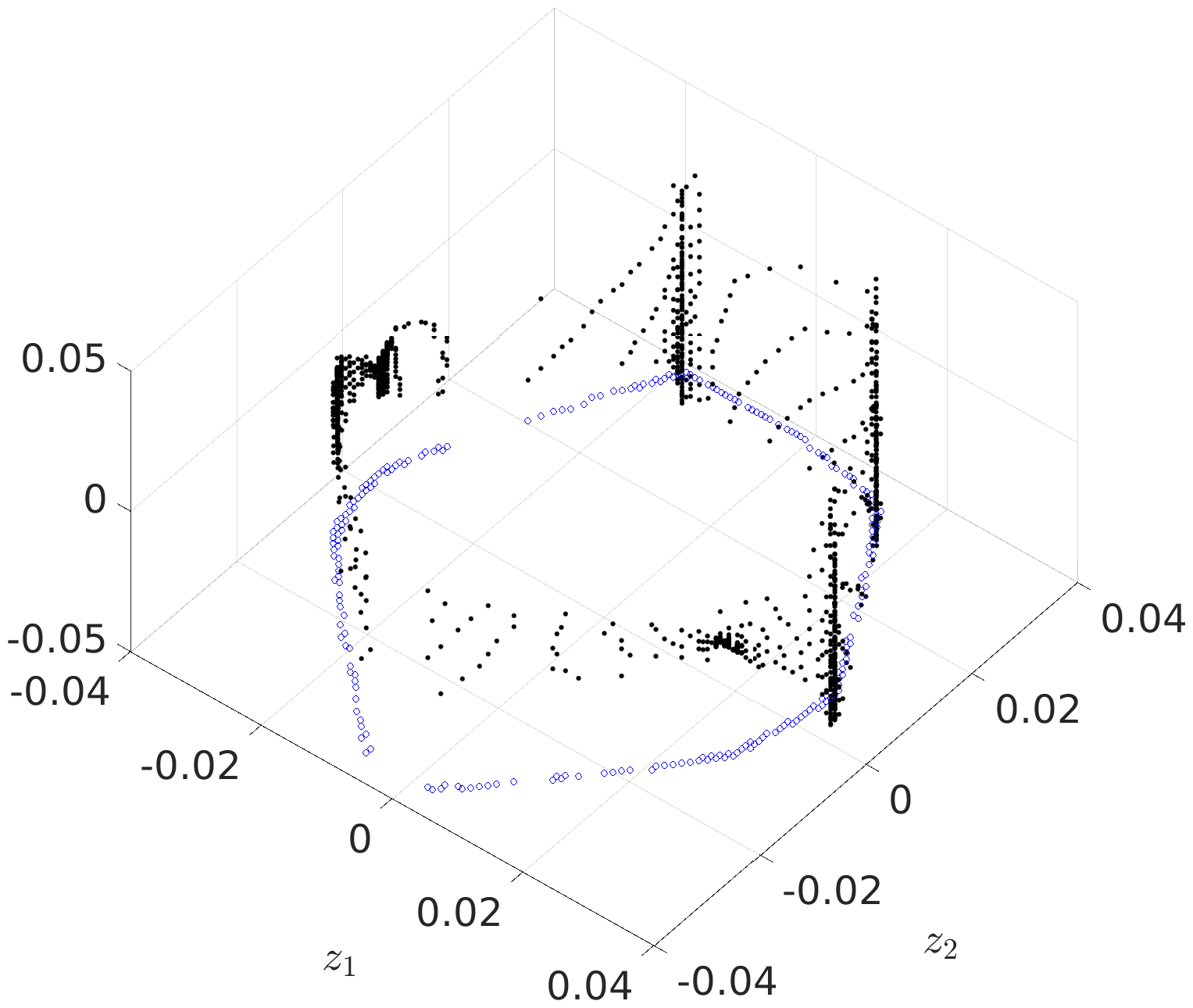}    
    \end{minipage}   
    \caption{Black dots: The values of the first nine eigenfunctions of $\mathcal{T}^t$ plotted against~$\overline{\xi}(x_i)$, $x_i\in\overline{\X}$. The blue markers indicate the~$\overline{\xi}(x_i)$ in the bottom plane. The seven dominant eigenfunctions ($\varphi_0$ to $\varphi_6$) seem to have a smooth dependency on~$\overline{\xi}$. In contrast, the values of the non-dominant $\varphi_7$ and $\varphi_8$ vary substantially over individual level sets of~$\overline{\xi}$.}
    \label{fig:functionaldependency}
\end{figure}

\paragraph{Circular potential in higher dimensions.} 

The identification of reaction coordinates is not limited to two dimensions. To show that our method can effectively find the reaction coordinates in high-dimensional systems, we extend the $7$-well potential to ten dimensions by adding a quadratic term in $x_3,\ldots,x_{10}$:

\begin{equation*}
     V(x) = \cos\left(7 \, \arctan(x_2, x_1)\right) + 10 \left(\sqrt{x_1^2 + x_2^2} - 1\right)^2 + 10\sum_{j=3}^{10}x_j^2\,.
\end{equation*}

We expect the one-dimensional circle $\{x\in\mathbb{R}^{10}~|~x_1^2+x_2^2=1,~x_j=0,~j=3,\ldots,10\}$ to be the transition path and accordingly choose a three-dimensional linear observable $\eta(x) = A\cdot x,~A\in\mathbb{R}^{3\times 10}$, where the coefficients $A_{ij}$ were again drawn uniformly from $[-1,1]$.

In ten dimensions, the computation of the reaction coordinate on all points of a regular grid is no longer possible due to the curse of dimensionality, and neither is the visualization of this grid.
Instead, we compute $\overline{\xi}$ on $10^5$ points sampled from the invariant measure and plot only the first three coordinates. Let this point cloud be called $\overline{\mathbb{X}}$.

Performing the standard procedure, i.e. embedding $\overline{\mathbb{X}}$ into $\mathbb{R}^3$ and identifying the one-dimensional core using diffusion maps, a two-component reaction coordinate is identified. Coloring the first three dimensions of $\overline{\mathbb{X}}$ by $\overline{\xi}$ (Figure \ref{fig:LemonSlice10D}a,b), we see that the expected reaction pathway is indeed parametrized. This pathway as well as the seven metastable states can also be recognized in a plot of the components of $\overline{\xi}(\overline{\X})$ plotted against each other, indicating that the information about the dominant eigenfunctions, thus the long-time jump process, is indeed retained by~$\overline{\xi}$.

\begin{figure}[htb]
    \centering
    \begin{minipage}{0.32\textwidth}
        \centering
        \subfiguretitle{a)}
        \includegraphics[width=\textwidth]{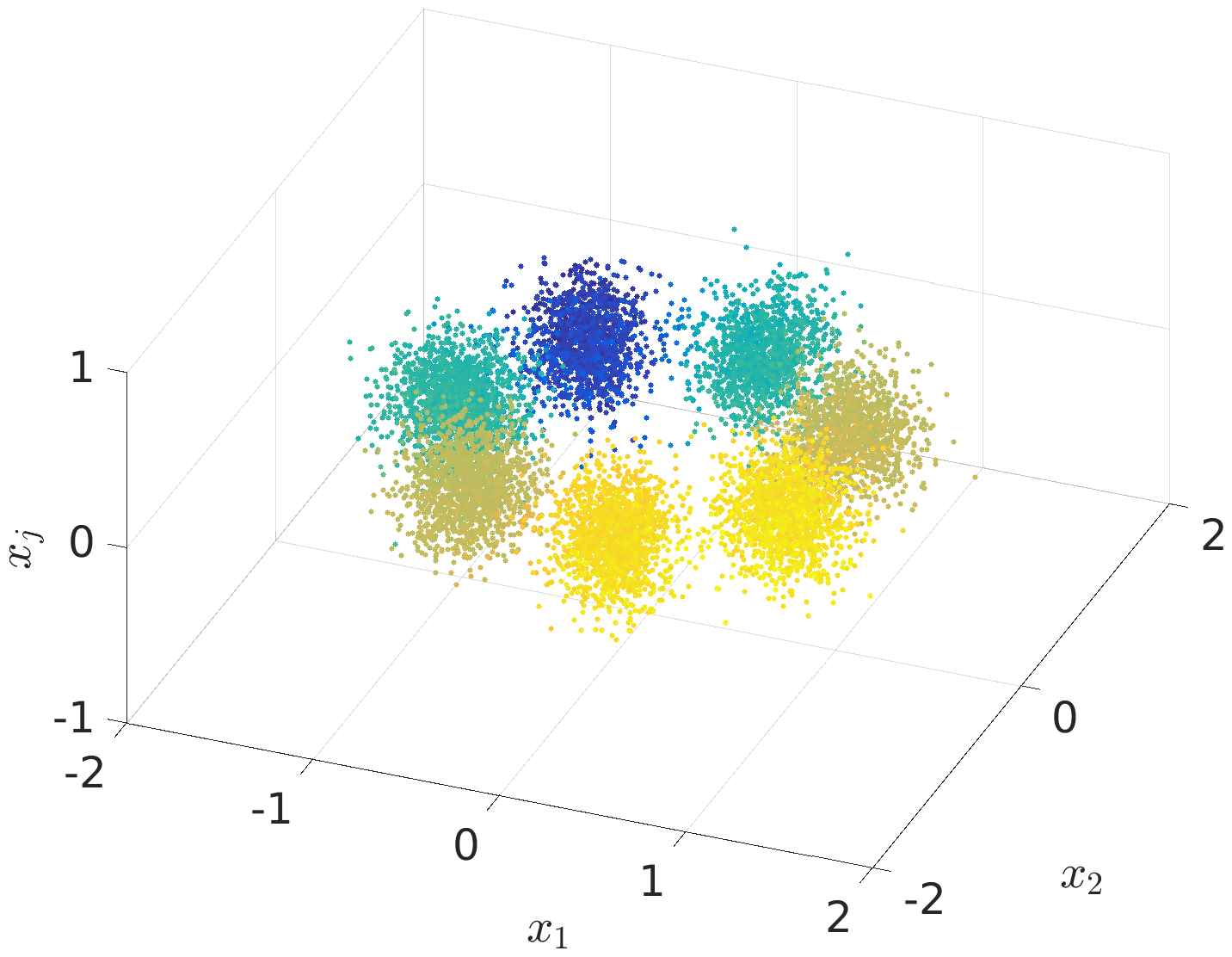}
    \end{minipage}
    \begin{minipage}{0.32\textwidth}
        \centering
        \subfiguretitle{b)}
        \includegraphics[width=\textwidth]{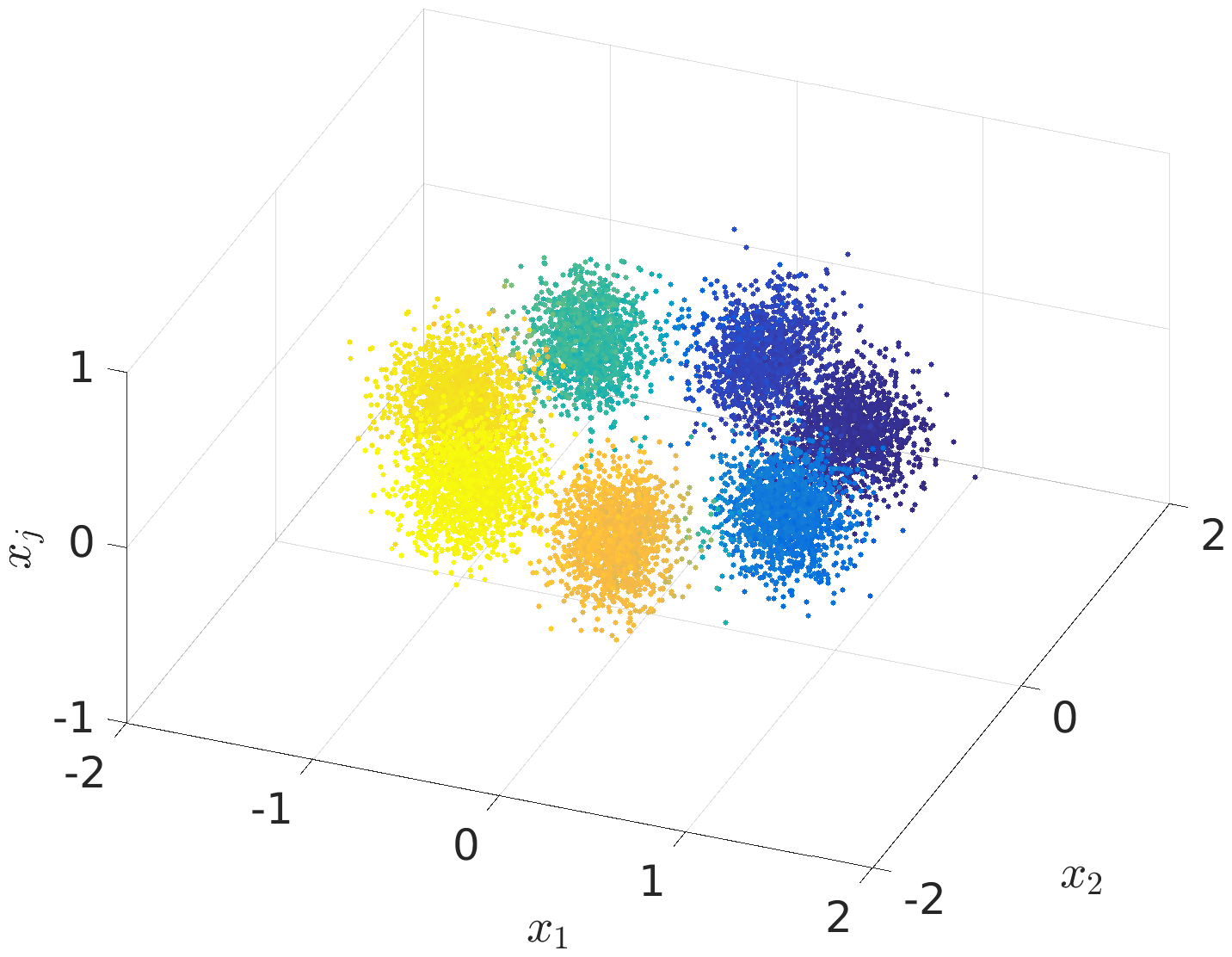}
    \end{minipage}
    \begin{minipage}{0.32\textwidth}
        \centering
        \subfiguretitle{c)}
        \includegraphics[width=\textwidth]{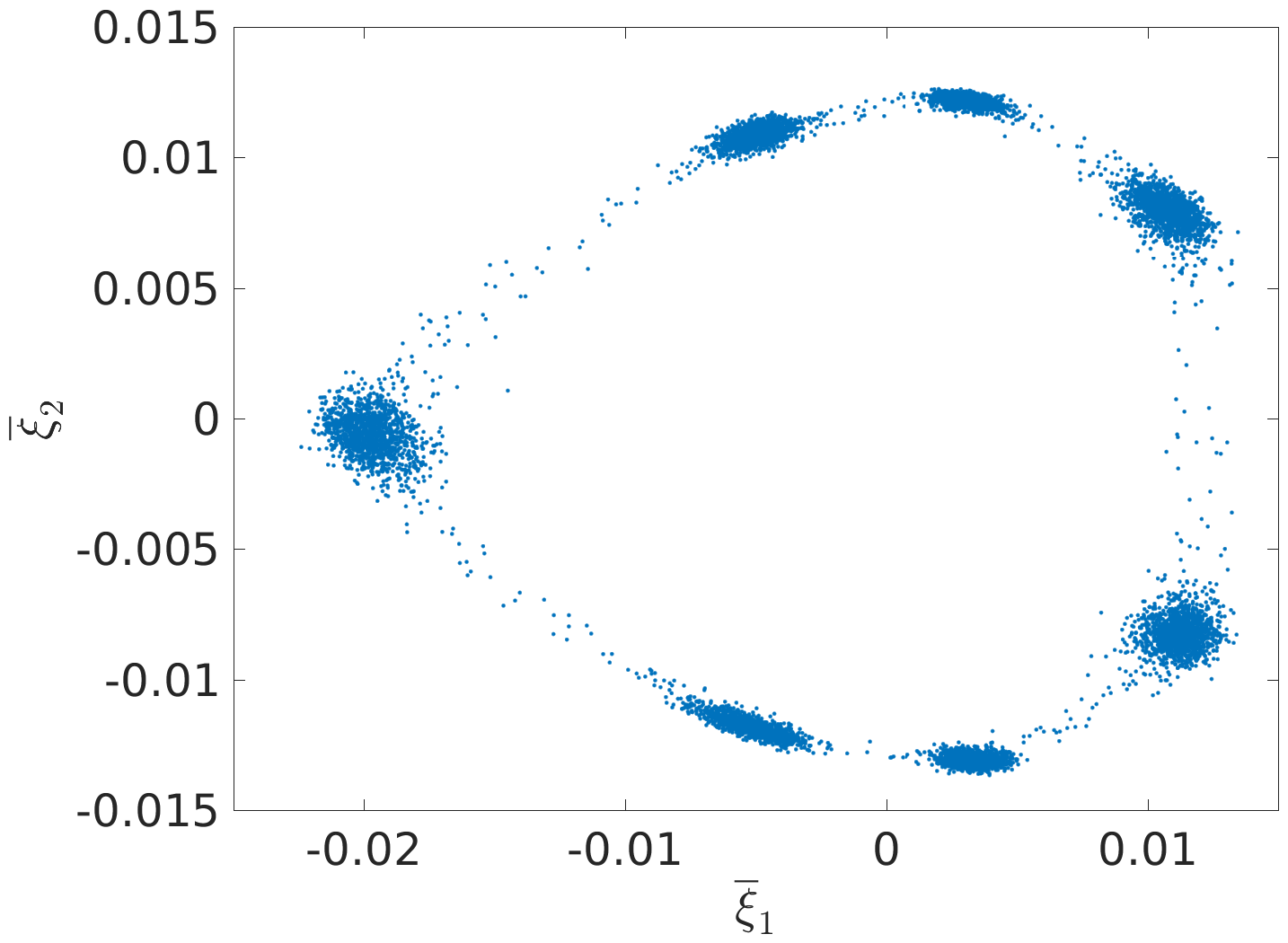}
    \end{minipage}    
    \caption{a) \& b) The two components $\overline{\xi}_1$ and $\overline{\xi}_2$ on the sampling points $\overline{\mathbb{X}}$. The picture shows the first three dimensions of $ x $, but is qualitatively the same when replacing $x_3$ by $x_j~,j=4,\ldots,10$. c) The values of $\overline{\xi}_1$ and $\overline{\xi}_2$ on $\overline{\mathbb{X}}$ plotted against each other.}
    \label{fig:LemonSlice10D}
\end{figure}

\subsection{Two quadruple well potentials}
\label{sec:Hilly flat}

Our theory is based on the existence of an $r$-dimensional transition manifold $\mathbb{M}$ in $L^1(\mathbb{X})$ around which the transition probability functions concentrate. In Appendix~\ref{app:RCexist}, we argued that the existence of an $r$-dimensional transition path suffices to ensure the existence of $\mathbb{M}$. Here we illustrate how the existence of the transition path is reflected in the embedding procedure.

For this we consider the ``hilly'' and ``flat'' quadruple well potentials
$$
V_1(x) =  (x_1^2 - 1)^2 + (x_2^2 - 1)^2 + 5\exp(-5\big(x_1^2 + x_2^2)\big)
$$
and
\begin{alignat*}{3}
V_2(x) &= 1 &- \exp\big(-10\big((x_1 - 1)^2 + (x_2 - 1)^2\big)^2\big) &- \exp\big(-10\big((x_1 - 1)^2 + (x_2 + 1)^2\big)^2\big)\\
&&-\exp\big(-10\big((x_1 + 1)^2 + (x_2 + 1)^2\big)^2\big)&-\exp\big(-10\big((x_1 + 1)^2 + (x_2 - 1)^2\big)^2\big)\,.
\end{alignat*}

\begin{figure}[htb]
    \centering
    \begin{minipage}{0.45\textwidth}
        \centering
        \subfiguretitle{$V_1$}
        \includegraphics[width=\textwidth]{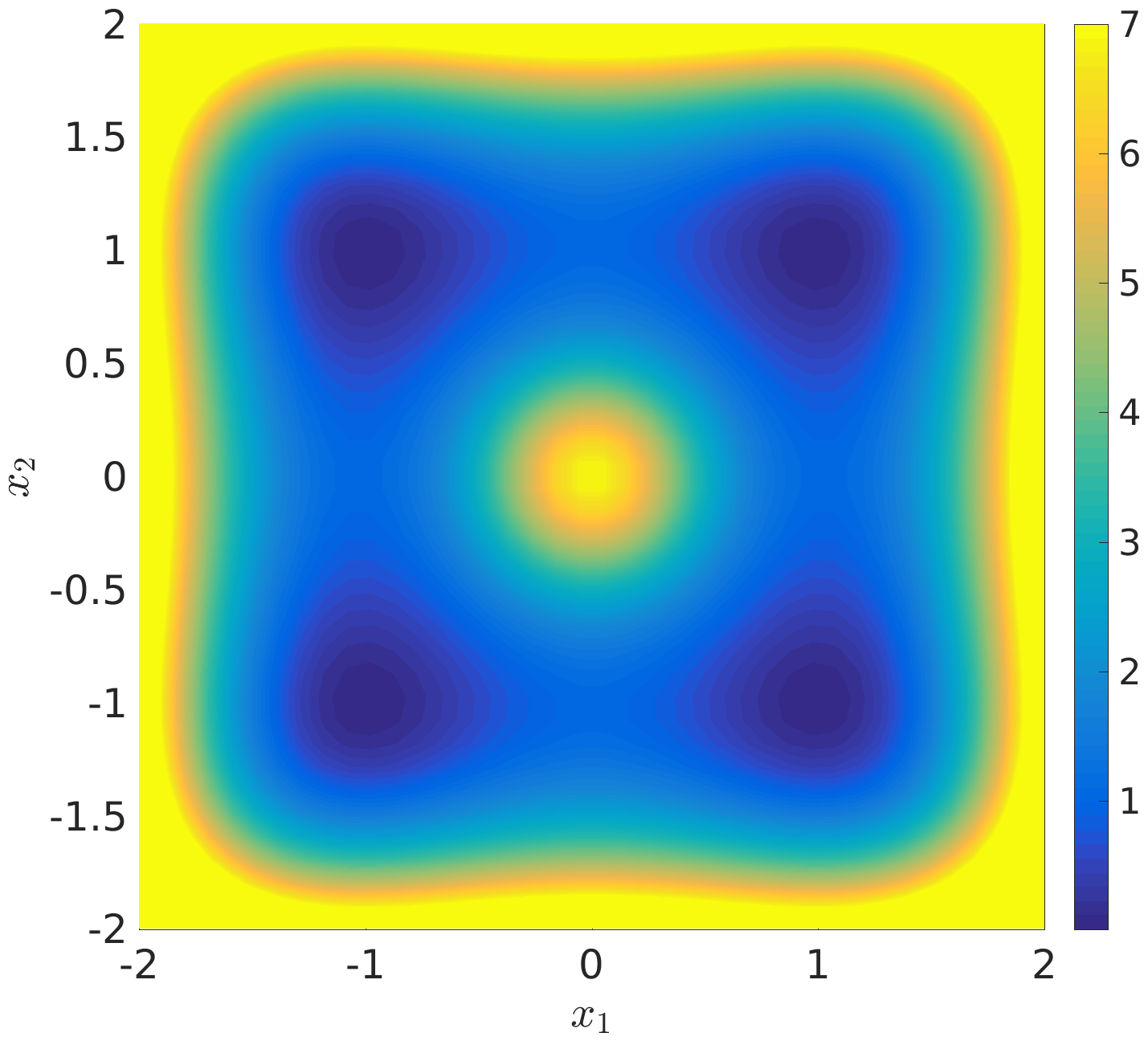}
    \end{minipage}
    \hfill
    \begin{minipage}{0.45\textwidth}
        \centering
        \subfiguretitle{$V_2$}
        \includegraphics[width=\textwidth]{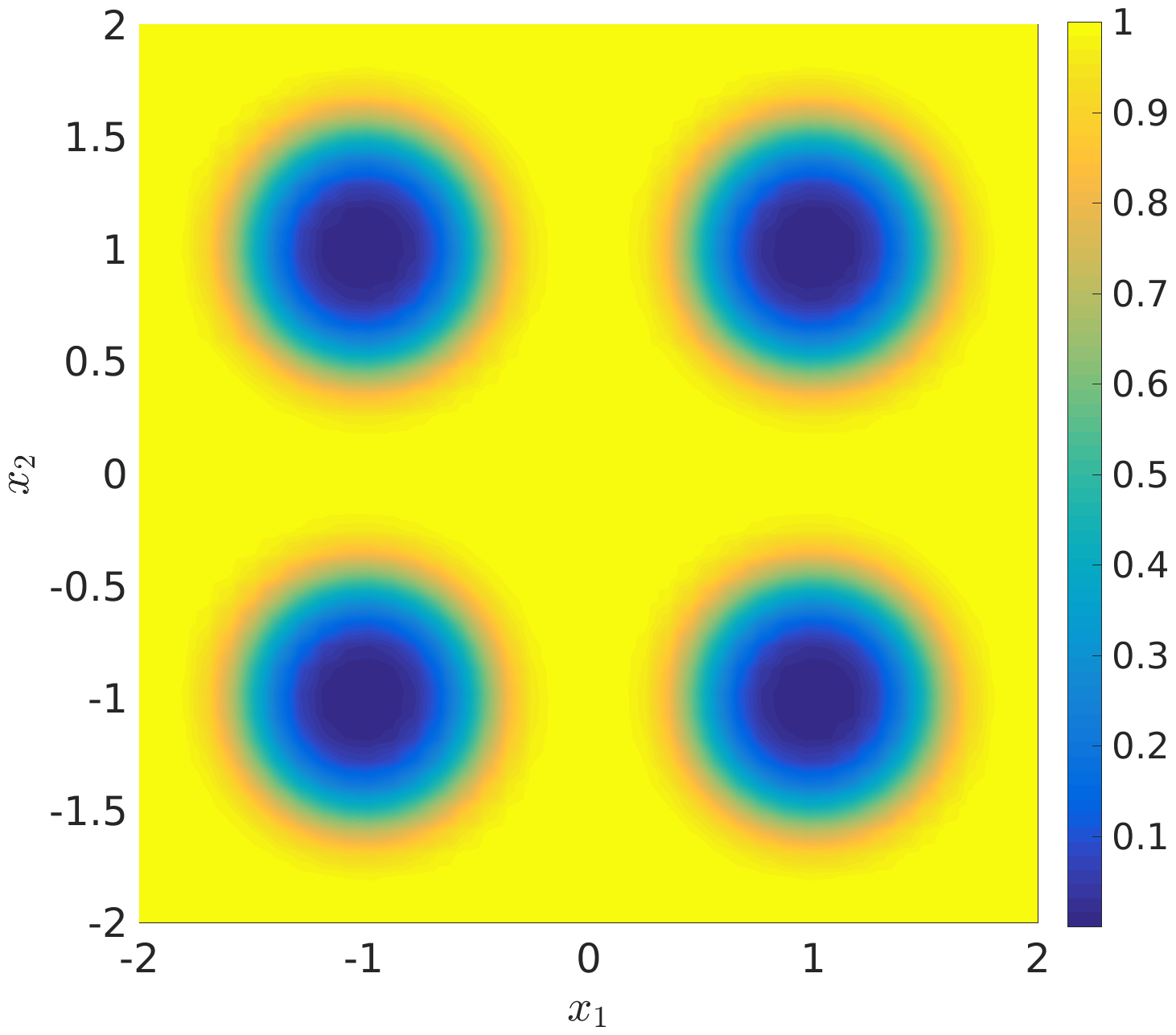}
    \end{minipage}
    \caption{The two quad-well potentials $V_1$ and $V_2$ possess qualitatively different transition regions.}
    \label{fig:quadwellpots}
\end{figure}

Both systems possess metastable sets around the four minima $(\pm 1,\pm 1)$, but~$V_1$ confines its dynamics outside of the metastable sets onto a one-dimensional transition path, whereas~$V_2$ does not impose such restrictions on the dynamics (see Figure~\ref{fig:quadwellpots}). For both potentials the time $t=1$ lies inside the slow-fast time scale gap. Assuming a one-dimensional transition manifold (wrongfully for $V_2$), we use the three linear observables~\eqref{eq:bananaobservables}. A $40\times 40$ grid on $[-2,2]\times[-2,2]$ is used as evaluation points for $\overline{\xi}$. The embedding of these points by $\mathcal{E}\circ \overline{\mathcal{Q}}$ can be seen in Figure \ref{fig:quadwellembeddings}. We observe a one-dimensional structure in the case of the ``hilly'' potential $V_1$, whereas the embedding points of the ``flat'' potential $V_2$ lie on a seemingly two-dimensional manifold. As these embeddings are approximately one-to-one with the respective transition manifolds $\mathbb{M}$, we conclude that in the case of $V_1$ the manifold $\mathbb{M}$ must be one-dimensional, whereas for $V_2$ it is two-dimensional. 

\begin{figure}[htb]
\centering
    \begin{minipage}{0.49\textwidth}
        \centering
        \subfiguretitle{a)}
        \includegraphics[width=\textwidth]{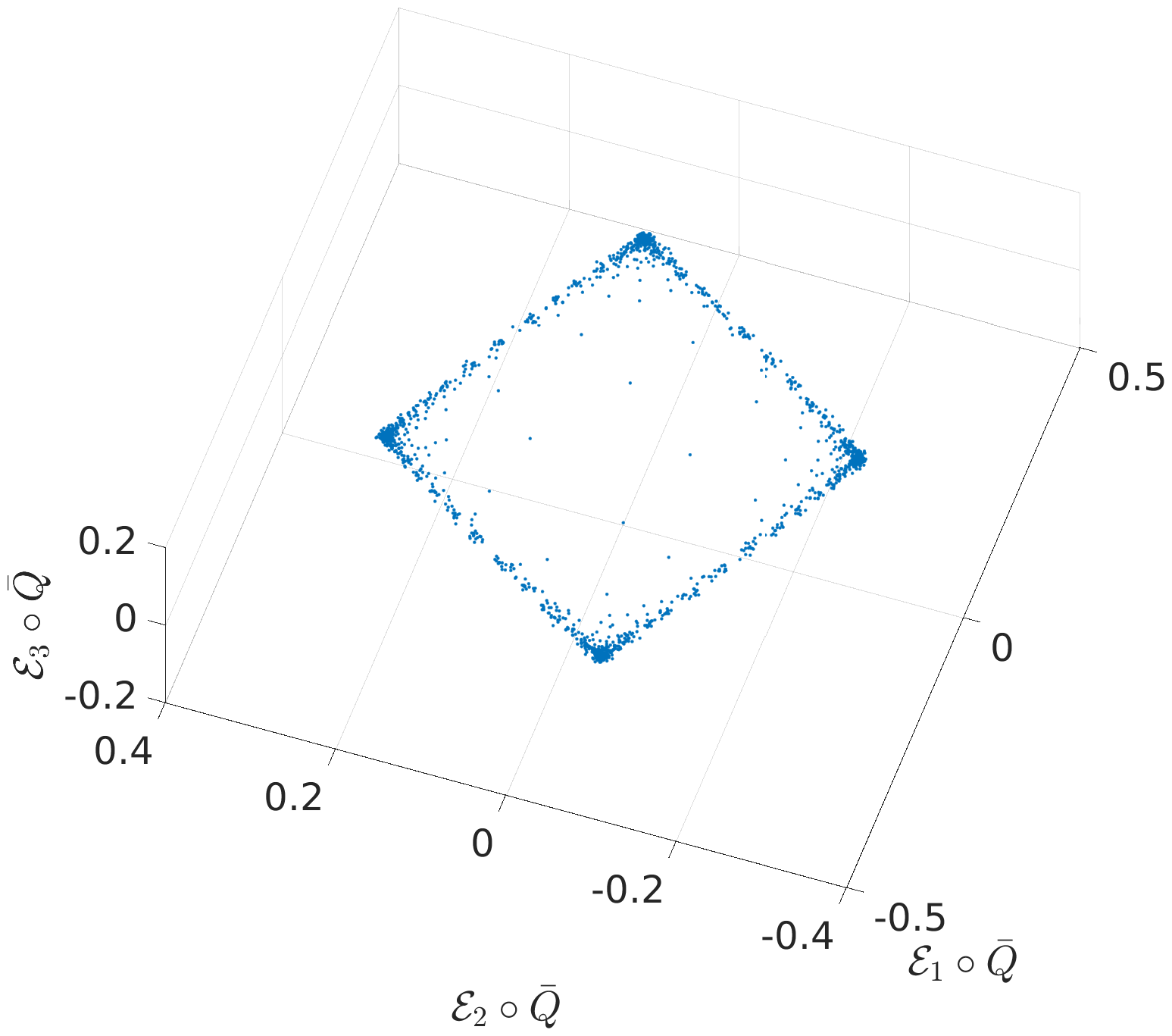}
    \end{minipage}
    \hfill
    \begin{minipage}{0.49\textwidth}
        \centering
        \subfiguretitle{b)}
        \includegraphics[width=\textwidth]{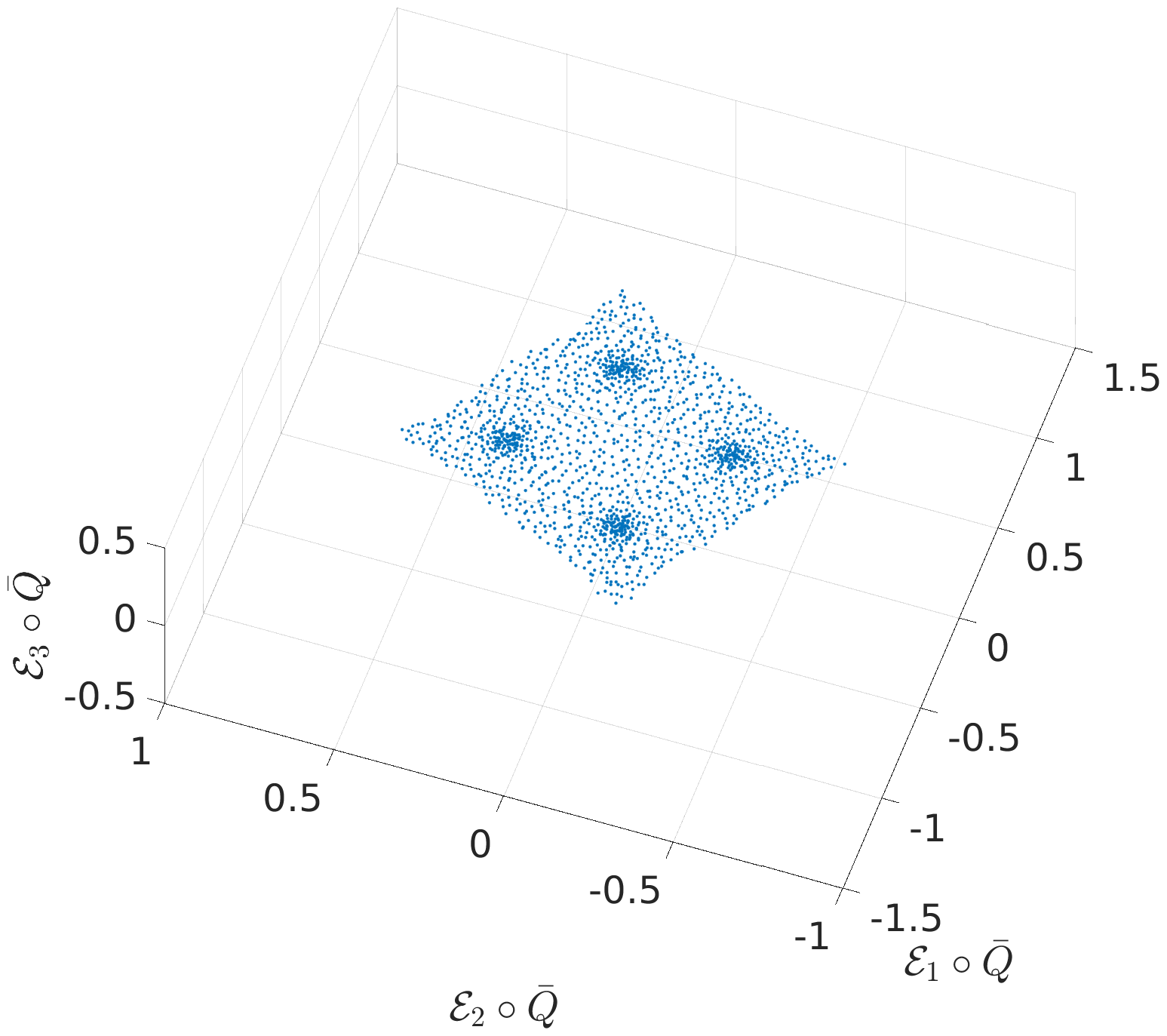}
    \end{minipage}
    \caption{Embedding of the grid points for the a) ``hilly'' and b) ``flat'' four well potential. A one-dimensional structure is only visible in a), i.e. in the presence of a one-dimensional transition path.}
\label{fig:quadwellembeddings}
\end{figure}

\section{Conclusion}
\label{sec:Conclusion}

Our main contributions in this paper are:
\begin{enumerate}[(a)]
\item We developed a mathematical framework to characterize good reaction coordinates for stochastic dynamical systems showing metastable behavior but no local separation of fast and slow time scales.
\item We showed the existence of good low-dimensional reaction coordinates under certain \emph{dynamical} assumptions on the system.
\item We proposed an algorithmic approach to numerically identify good reaction coordinates and the associated low-dimension transition manifold based on local evaluation of short trajectories of the system only.
\end{enumerate}
Our numerical examples show how the procedure works, that it can be used in higher dimensions, and the examples give further evidence that the dynamical assumptions from~(b) are valid in many realistic cases. The application of our approach to relevant biomolecular problems, e.g.\ in protein folding, is ongoing work.

Apart from the application to actual molecular systems, there are several open questions and challenges, which we will address in the future:
\begin{enumerate}[$\bullet$]
\item A rigorous mathematical justification for the dynamical assumption in Definition~\ref{def:reducibleprocess} in terms of the potential~$V$ and the noise intensity~$\beta^{-1}$ in~\eqref{eq:overdampedLangevin} would be desirable. This seems to be a demanding task, as the interplay between potential landscape and the thermal forcing is nontrivial. For $\beta^{-1}\to 0$ the problem can be handled by large deviation approaches; however, understanding increasing~$\beta^{-1}$ is challenging: the strength of noise increases, and additional transitions between metastable sets become more probable, as the barriers in the potential landscape become less significant, and thus the reaction coordinate may increase in dimension.
\item Also related to the previous point, the choice of the correct lag time~$t$ is crucial. Choosing the time too small, the concentration of the transition densities near a low-dimensional manifold in~$L^1$ may not have happened yet, but a too large lag time has severe consequences for the numerical expenses. If no expert knowledge of a proper lag time~$t$ is available, it has to be identified in a pre-processing step, for example using Markov State Model techniques \cite{A19-1}.
\item As discussed in the last part of Section~\ref{sec:NumApprox} and in Figure~\ref{fig:RC_realized}, we need the embedding~$\mathcal{E}$ not to distort transversality close to the transition manifold~$\mathbb{M}$ too much, such that the realized reaction coordinate~$\overline{\xi}$ is indeed a good one. Theoretical bounds shall be developed. This problem seems to be coupled with the problem of how to control the condition number of the embedding and its numerical realization.
\item The dimension~$r$ of the reaction coordinate may not be known in advance, hence we need an algorithmic strategy to identify this on the fly. Fortunately, once the sampling has been made, the evaluation of the embedding mapping~$\mathcal{E}$, and finding intrinsic coordinates on the set of data points embedded in~$\mathbb{R}^k$ has a negligible numerical effort, hence different embedding dimensions~$k$ can be probed via~\eqref{eq:gammadefinition}. Theorem~\ref{thm:embeddingThm} suggests that if the identified dimension of the reaction coordinate is smaller than~$k/2$, then a reaction coordinate of sufficient dimension has been found.
\item To benefit from the dimensionality reduction of the reaction coordinate $\xi$, the dynamics that generates the reduced transfer operator $\mathcal{T}^t_\xi$ has to be described in closed form. We are planing to employ techniques based on the Kramers--Moyal extension \cite{ZHS16} to again receive an SDE for a stochastic process on $\mathbb{R}^r$.
\item The embedding mapping~$\mathcal{E}$ is evaluated by Monte Carlo quadrature~\eqref{eq:Monte Carlo approximation}. Although Monte Carlo quadrature is known to have a convergence rate independent of the underlying dimension~$n$ of~$\X$, there is still an impact of the dimension on the practical accuracy. This we shall investigate as well.
\end{enumerate}

\section*{Acknowledgements}

This research has been partially funded by Deutsche Forschungsgemeinschaft (DFG) through grant CRC 1114 ``Scaling Cascades in Complex Systems'', Project B03 ``Multilevel coarse graining of multi-scale problems'', and by the Einstein Foundation Berlin (Einstein Center ECMath).

\bibliographystyle{abbrv}
\bibliography{findingreactioncoordinates}

\appendix

\section{Properties of $\mathbf{P}_{\boldsymbol{\xi}}$ }
\label{app:Pxi}

\begin{proof}[Proof of Proposition~\ref{prop:PxiProperties}]
(a) This property has been shown by Zhang~\cite{ZHS16} as well, we include the short reasoning for completeness. The linearity of~$P_{\xi}$ is obvious. The property~$P_{\xi}^2=P_{\xi}$ follows from~\eqref{eq:Pxi1} by noting that~$P_{\xi}f$ is constant on~$\mathbb{L}_z$ and that~$\mu_z$ is a probability measure for every~$z$.

(b) From~\eqref{eq:splitintPxi} we have for~$f,g\in L^2_{\mu}(\X)$ that
\begin{eqnarray}
\langle P_{\xi}f,g\rangle_{\mu} & = & \int_{\X} P_{\xi}f(x) g(x)\,d\mu(x) \nonumber\\
& \stackrel{\eqref{eq:splitintPxi}}{=} & \int_{\xi(\X)} \Gamma(z)\widehat{P_{\xi}\!\left(gP_{\xi}f\right)}(z)\,dz \nonumber\\
& \stackrel{(\ast)}{=} & \int_{\xi(\X)} \Gamma(z)\widehat{P_{\xi}f}(z)\widehat{P_{\xi}g}(z)\,dz\,, \label{eq:PxiSymm}
\end{eqnarray}
where~$(\ast)$ follows from the linearity of~$P_{\xi}$, and the fact that~$P_{\xi}f\vert_{\mathbb{L}_{\xi(x)}} = \mathrm{const}$, thus $\widehat{P_{\xi}\!\left(gP_{\xi}f\right)}(z) = \widehat{P_{\xi}f}(z)\widehat{P_{\xi}g}(z)$. Expression~\eqref{eq:PxiSymm} is symmetric in~$f$ and~$g$, hence it follows that~$\langle P_{\xi}f,g\rangle_{\mu} = \langle f,P_{\xi}g\rangle_{\mu}$.

(c) We first prove that~$P_{\xi}$ is an orthogonal projection:
\[
\langle P_{\xi}f, f-P_{\xi}f\rangle_{\mu} \stackrel{(b)}{=} \langle f, P_{\xi}f-P_{\xi}^2f\rangle_{\mu} \stackrel{(a)}{=} \langle f, P_{\xi}f-P_{\xi}f\rangle_{\mu} = 0\,.
\]
Thus,
\[
\|f\|_{L^2_{\mu}}^2 = \|f - P_{\xi}f\|_{L^2_{\mu}}^2 + \|P_{\xi}f\|_{L^2_{\mu}}^2 \ge \|P_{\xi}f\|_{L^2_{\mu}}^2\,,
\]
and the claim follows.

\end{proof}

\section{On the existence of reaction coordinates}
\label{app:RCexist}

To motivate the existence of low-dimensional reaction coordinates, let us assume that the dynamics of consideration has $d+1$ metastable regions $\mathbb{C}_0,\ldots, \mathbb{C}_{d} \subset \X$. Let $\mathbb{C}= \bigcup_i \mathbb{C}_i$. For a selected lag time~$t>0$ we make the following two assumptions:
\begin{enumerate}[1)]
\item Fast local equilibration: If~$x$ is in (or close to)~$\mathbb{C}_i$ then we have
\[
\mathcal{P}^t \delta_x \approx \varrho_i^{qs}
\]
where  $\varrho_i^{qs}$ is the quasi-stationary density of the metastable core~$\mathbb{C}_i$:
\[
\lim_{s\rightarrow\infty} \mathsf{P}\left[ \mathbf{X}_s = y\,\big\vert\, \bm{\tau}_{\mathbb{C}_i} > s\right] = \varrho_i^{qs}(y) dy
\]
with~$\bm{\tau}_{\mathbb{C}_i}$ being the (random) exit time from the set~$\mathbb{C}_i$.
\item Slow transitions: The typical transition time to reach $\mathbb{C} \setminus \mathbb{C}_i$ when starting in $\mathbb{C}_i$ is larger than $t$. In other words,~$t$ is such, that if the process~$\mathbf{X}_s$ transitions from~$x$ to some~$\mathbb{C}_i$, it did not transition through some other~$\mathbb{C}_j$ with high probability.
\end{enumerate}
These two assumptions essentially say that~$t$ is much larger than the fast time scales of the system, but smaller than the dominant time scales.
It follows that, for any $x \in \X$,
\[
\mathcal{P}^t \delta_x \approx \sum_{i=0}^d q_i(x) \varrho_i^{qs}, \quad \sum_{i=0}^d q_i(x) = 1\,,
\]
where by assumption 2) the coefficients $q_i(x)$ are given by the committor functions
\[
q_i(x) = \mathsf{P}\left[ \mathbf{X}_t\;\mbox{reaches $\mathbb{C}_i$ before $\mathbb{C}\setminus \mathbb{C}_i$} \,\big\vert\, \mathbf{X}_0 = x\right].
\]
We say that $\mathcal{P}^t\delta_x$ is an $r$-dimensional structure in $L^1(\X)$ if there is a function $\xi \colon \X \rightarrow \R^r$ that jointly parametrizes all the committor functions, i.e.,~$q_i= \tilde q_i \circ \xi$ with~$\tilde q_i \colon \R^r \rightarrow \R$. If this is the case, then
\[
\overline{\mathcal{Q}}(x) = \mathcal{P}^t \delta_x \approx \sum_{i=0}^d \tilde q_i(\xi(x)) \varrho_i^{qs} =: \mathcal{Q}(x)
\]
and clearly $\dim (\mathcal{Q}(\X)) \leq r$ since $\dim (\xi(\X)) = r$. Moreover, $r\leq d$ since we can explicitly construct $\xi \colon \X \rightarrow \R^{d}$ as $\xi = (q_1, \ldots, q_{d})$. This obviously parameterizes $q_1,\ldots, q_{d}$, and it also parameterizes $q_0$ since $q_0 = 1 - \sum_{i=1}^{d} q_i$.

However,~$r$ may also be smaller than~$d$. As an example, consider the potential with $4$ minima shown in Figure \ref{fig:appendix_commitors} on the left. At low temperatures, the ``hilly'' potential energy landscape confines all transitions between the minima $\mathbb{C}_0, \ldots, \mathbb{C}_3$ to a narrow region close to the red square connecting the four minima. Figure \ref{fig:appendix_commitors} shows the level sets of $q_0$, the level sets of the other committors are given by the rotational symmetry of the problem. All four committors can be jointly parameterized by a single coordinate $\xi$ which describes clockwise movement along the red square and is constant orthogonal to it. Therefore, $r=1$. Figure \ref{fig:appendix_commitors} on the right shows the situation with a ``flat'' energy landscape. Transition paths between the minima are no longer confined to a one-dimensional structure, and the committor level sets are more complicated. We can no longer parameterize all four committors with a single coordinate $\xi$, so $r>1$. On the other hand, $\dim(\X) = 2$ so $r=2$.

\begin{figure}[htb]
    \includegraphics[width=0.49\textwidth]{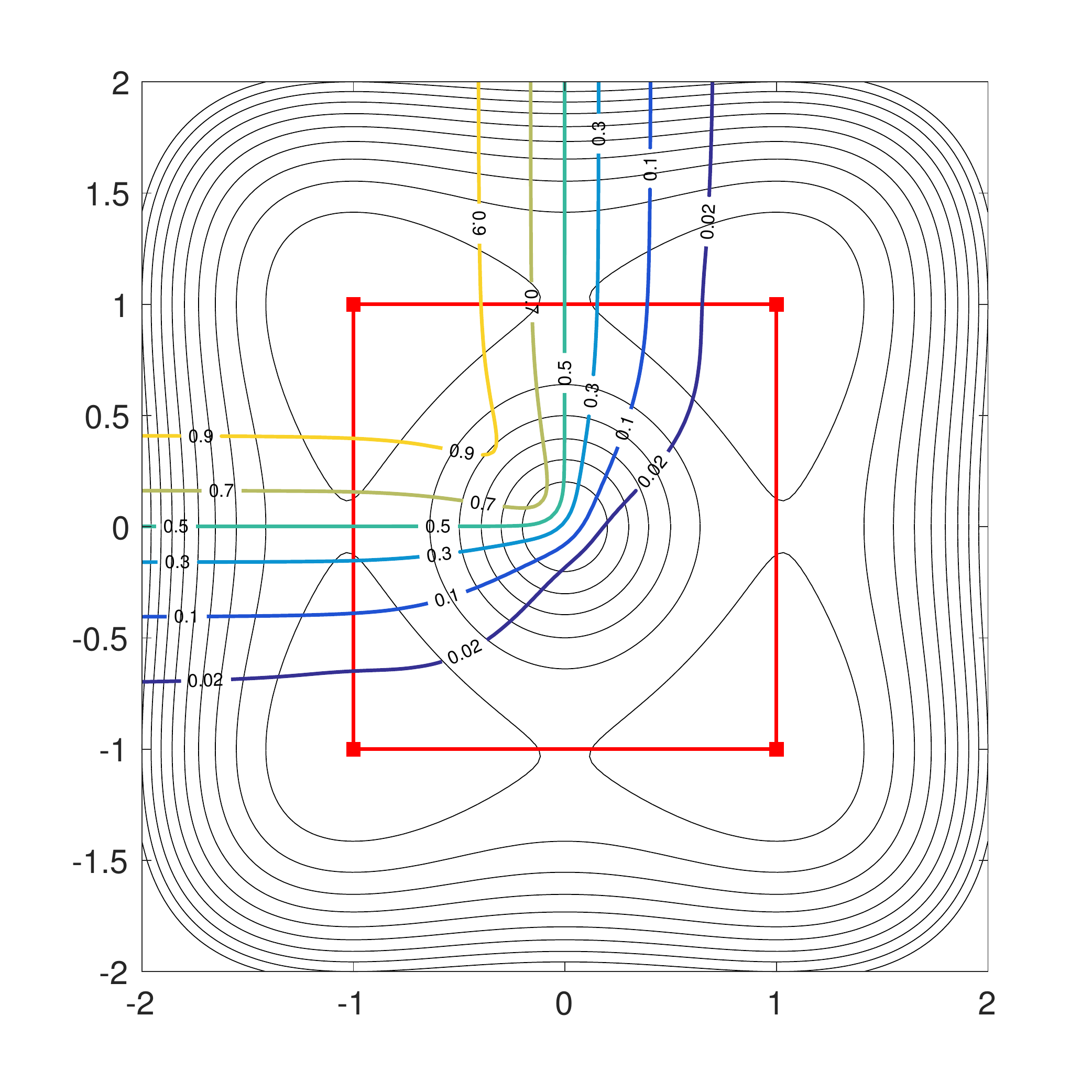}
    \includegraphics[width=0.49\textwidth]{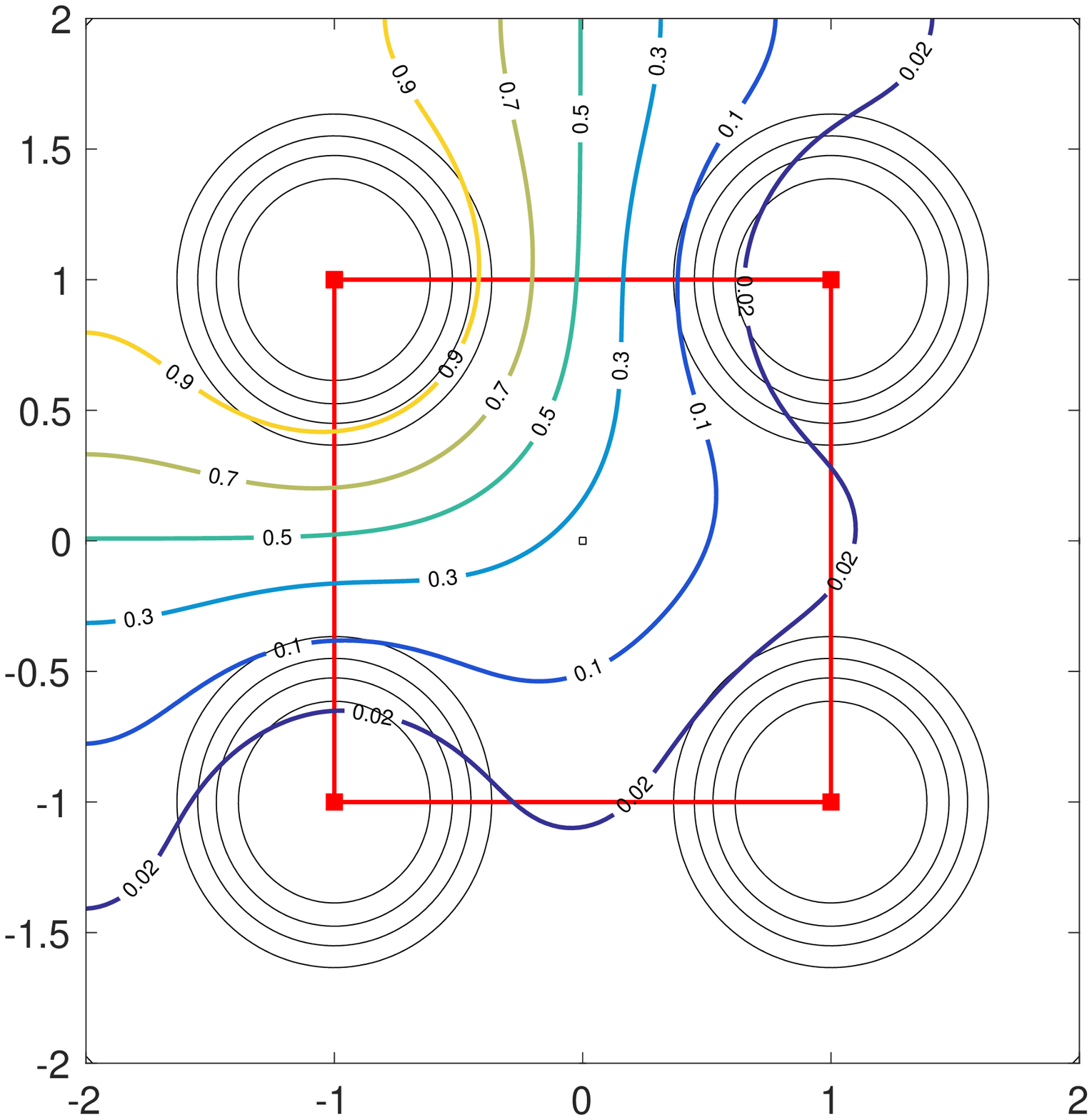}
    \caption{A potential energy landscape with four minima (black contours) and level sets of $q_1$ (colored contours). Left: The ``hilly'' landscape structure confines transition pathways to a narrow region close to the red square connecting the four minima. As a result, all committor level sets are orthogonal to this main transition path. Right: ``Flat'' landscape structure with more complicated level sets of the committors.}
    \label{fig:appendix_commitors}
\end{figure}

This structural difference of the potentials can also be seen when applying our algorithm to construct the reaction coordinate $\overline{\xi}$, see Figure~\ref{fig:quadwellembeddings} and Section~\ref{sec:Hilly flat}.

\end{document}